\documentclass[10pt,reqno]{amsart}

\usepackage{amsmath,amsfonts,amsthm,amssymb,amsxtra}
\usepackage{amsxtra, amssymb, mathrsfs, pstricks}
\usepackage[english]{babel}
\usepackage{amsmath,amsfonts,amstext,amsbsy,amsopn,amssymb,amsthm,amscd}
\usepackage{amsxtra,latexsym}
\usepackage{exscale}
\usepackage{graphicx,mathrsfs}
\usepackage{psfrag}
\usepackage{paralist,eucal,enumerate}
\usepackage{color}

\definecolor{dred}{rgb}{.8,0,0}
\definecolor{ddmagenta}{rgb}{0.7,0,0.9}
\definecolor{ddcyan}{rgb}{0,0.2,1.0}
\definecolor{dblue}{rgb}{0,0,0.7}
\definecolor{ddorange}{rgb}{1,0.5,0}
\definecolor{ddgreen}{rgb}{0,0.4,0.4}

\newtheorem{theorem}{Theorem}[section]
\newtheorem{proposition}[theorem]{Proposition}

\newtheorem{lemma}[theorem]{Lemma}

\theoremstyle{definition}

\newtheorem{definition}[theorem]{Definition}

\newtheorem{notation}[theorem]{Notation}
\newtheorem{problem}[theorem]{Problem}
\newtheorem{maintheorem}{Theorem}

\theoremstyle{remark}
\newtheorem{remark}[theorem]{\bf Remark}

\numberwithin{equation}{section}

\def\trait #1 #2 #3 {\vrule width #1pt height #2pt depth #3pt}
\def\fin{
    \trait .3 5 0
    \trait 5 .3 0
    \kern-5pt
    \trait 5 5 -4.7
    \trait 0.3 5 0
\medskip}

\newcommand{\QED}{\mbox{}\hfill\rule{5pt}{5pt}\medskip\par}

 \def\bbN{{\mathbb N}} 
  \def\bbR{{\mathbb R}}

\newcommand{\R}{\bbR}
\newcommand{\N}{\bbN}

\def\BS{\boldsymbol} 

    \def\bfmu{{\BS\mu}}

\newcommand{\bele}{\begin{lemm}\begin{sl}}
\newcommand{\enle}{\end{sl}\end{lemm}}
\newcommand{\bedef}{\begin{defi}\begin{sl}}
\newcommand{\eddef}{\end{sl}\end{defi}}
\newcommand{\bete}{\begin{teor}\begin{sl}}
\newcommand{\ente}{\end{sl}\end{teor}}
\newcommand{\beos}{\begin{osse}\begin{rm}}
\newcommand{\eddos}{\end{rm}\end{osse}}
\newcommand{\bepr}{\begin{prop}\begin{sl}}
\newcommand{\empr}{\end{sl}\end{prop}}
\newcommand{\bepro}{\begin{prob}\begin{rm}}
\newcommand{\empro}{\end{rm}\end{prob}}
\newcommand{\bede}{\begin{defin}\begin{sl}}
\newcommand{\edde}{\end{sl}\end{defin}}
\newcommand{\beco}{\begin{coro}\begin{sl}}
\newcommand{\enco}{\end{sl}\end{coro}}
\newcommand{\behy}{\begin{hypo}\begin{sl}}
\newcommand{\enhy}{\end{sl}\end{hypo}}

\newcommand{\thspace}{\hspace{3mm}}

\newcommand{\RR}{\mathbb{R}}

\newcommand{\beeq}[1]{\begin{equation}\label{#1}}
\newcommand{\eddeq}{\end{equation}}
\newcommand{\beeqa}[1]{\begin{eqnarray}\label{#1}}
\newcommand{\eddeqa}{\end{eqnarray}}
\newcommand{\beal}[1]{\begin{align}\label{#1}}
\newcommand{\eddal}{\end{align}}
\newcommand{\bespl}[1]{\begin{split}\label{#1}}
\newcommand{\edspl}{\end{split}}
\newcommand{\bega}[1]{\begin{gather}\label{#1}}
\newcommand{\edga}{\end{gather}}
\newcommand{\beeqax}{\begin{eqnarray*}}
\newcommand{\eddeqax}{\end{eqnarray*}}

\newcommand{\no}{\nonumber}

\newcommand{\weaksto}{{\rightharpoonup^*}}
\newcommand{\weakto}{\rightharpoonup}

\newcommand{\dt}{\partial_t}

\newcommand{\itt}{\int_0^t}

\newcommand{\io}{\int_\Omega}

\newcommand{\e}{\varepsilon}

 \DeclareMathOperator{\dive}{div}

\let\TeXchi\chi
\def\chi{{\setbox0 \hbox{\mathsurround0pt
$\TeXchi$}\hbox{\raise\dp0 \copy0 }}}

\newcommand{\ub}{\mathbf{u}}
\newcommand{\uu}{\mathbf{u}}
\newcommand{\vv}{\mathbf{v}}
\newcommand{\ww}{\mathbf{w}}
\newcommand{\vb}{\mathbf{v}}

\newcommand{\teta}{\vartheta}

\newcommand{\aein}{\text{a.e. in\,}}

\newcommand{\zzeta}{{\mbox{\boldmath$\zeta$}}}
\newcommand{\tensore}{\varepsilon({\bf u})}
\newcommand{\tensoret}{\varepsilon({\mathbf{u}_t})}

\newcommand{\forae}{\text{for a.a.}}
\newcommand{\foraa}{\text{for a.a.}}

\def\fine{\hfill\kern4pt \vrule height4pt depth0pt width4pt }

\numberwithin{equation}{section}
\numberwithin{equation}{section}

\newcommand{\Ha}{L^2 (\Omega;\R^d)}
\newcommand{\V}{H^1 (\Omega)}

\newcommand{\boZ}{H_{0}^1(\Omega;\R^d)}
\newcommand{\boY}{H_{\mathrm{Dir}}^2(\Omega;\R^d)}

\newcommand{\dd}{\, \mathrm{d}}
\newcommand{\pairing}[4]{ \sideset{_{#1 }}{_{ #2}}  {\mathop{\langle #3 , #4  \rangle}}}
\newcommand{\oph}[2]{\mathcal{E}\left({#1}{#2}\right)}
\newcommand{\ophname}{\mathcal{E}}
\newcommand{\opjname}{\mathcal{V}}
\newcommand{\opj}[2]{\mathcal{V}\left({#1}{#2}\right)}
\newcommand{\opchi}{\mathcal{B}}

\newcommand{\eps}{\varepsilon}
\newcommand{\condu}{\mathsf{K}}
\newcommand{\ental}{\teta}
\newcommand{\w}{\ental}

\newcommand{\utau}[1]{\uu_{\tau}^{#1}}
\newcommand{\wtau}[1]{\teta_{\tau}^{#1}}
\newcommand{\btau}[1]{\mathfrak{h}_{\tau}^{#1}}
\newcommand{\vtau}[1]{v_{\tau}^{#1}}
\newcommand{\chitau}[1]{\chi_{\tau}^{#1}}

\newcommand{\utaumu}[1]{\uu_{\tau,\nu}^{#1}}
\newcommand{\wtaumu}[1]{\teta_{\tau,\nu}^{#1}}
\newcommand{\chitaumu}[1]{\chi_{\tau,\nu}^{#1}}
\newcommand{\xitaumu}[1]{\xi_{\tau,\nu}^{#1}}

\newcommand{\zetaumu}[1]{\zeta_{\tau,\nu}^{#1}}

\newcommand{\elltaum}[1]{\ell_{\tau,M}^{#1}}
\newcommand{\jtaum}[1]{j_{\tau,M}^{#1}}

\newcommand{\xitau}[1]{\xi_{\tau}^{#1}}

\newcommand{\zetau}[1]{\zeta_{\tau}^{#1}}
\newcommand{\ftau}[1]{\mathbf{f}_{\tau}^{#1}}
\newcommand{\phitau}[1]{\varphi_{\tau}^{#1}}
\newcommand{\gtau}[1]{g_{\tau}^{#1}}
\newcommand{\htau}[1]{h_{\tau}^{#1}}

\newcommand{\dtau}[2]{\mathrm{D}_{\tau,{#1}}(#2)}
\newcommand{\duetau}[2]{\mathrm{D}_{\tau,{#1}}^2(#2)}

\newcommand{\pwc}[2]{\overline{#1}_{#2}}
\newcommand{\pwl}[2]{{#1}_{#2}}
\newcommand{\pwwll}[2]{\widehat{#1}_{#2}}
\newcommand{\upwc}[2]{\underline{#1}_{#2}}

\newcommand{\down}{\downarrow}
\newcommand{\ciro}{\mathcal{C}_\rho}
\newcommand{\BV}{\mathrm{BV}}
\newcommand{\rmC}{\mathrm{C}}

\newcommand{\vism}{\mathbb{V}}
\newcommand{\elm}{\mathbb{E}}
\newcommand{\vib}{\mathsf{v}}
\newcommand{\elb}{\mathsf{e}}
\newcommand{\bilh}[3]{\elb({#1}#2,#3)}
\newcommand{\bilj}[3]{\vib({#1}#2,#3)}
\newcommand{\frakh}{\mathfrak{h}}

\newenvironment{rcomm}{\color{dred} \textsf{R:}\,}{\color{black}}
\newenvironment{bcomm}{\color{dblue} \textsf{B:}\,}{\color{black}}

\newenvironment{new}{\color{violet}}{\color{black}}
\newcommand{\bne}{\begin{new}}
\newcommand{\ene}{\end{new}}
\newcommand{\beric}{\begin{rcomm}}
\newcommand{\eric}{\end{rcomm}}
\newcommand{\bebe}{\begin{bcomm}}
\newcommand{\ebe}{\end{bcomm}}
\newenvironment{rickynew}{\color{ddmagenta}}{\color{black}}
\newcommand{\berin}{\begin{rickynew}}
\newcommand{\erin}{\end{rickynew}}

\newenvironment{revuno}{\color{black}}{\color{black}}
\newcommand{\beruno}{\begin{revuno}}
\newcommand{\eruno}{\end{revuno}}

\newenvironment{revdue}{\color{black}}{\color{black}}
\newcommand{\berdue}{\begin{revdue}}
\newcommand{\erdue}{\end{revdue}}

\newenvironment{revunodue}{\color{black}}{\color{black}}
\newcommand{\berunodue}{\begin{revunodue}}
\newcommand{\erunodue}{\end{revunodue}}

\newenvironment{ourrev}{\color{black}}{\color{black}}
\newcommand{\beo}{\begin{ourrev}}
\newcommand{\eo}{\end{ourrev}}

\newcommand{\EEE}{\color{black}}
\newcommand{\RRR}{\color{black}}
\newcommand{\RRRNEW}{\color{black}}

\begin{document}

\title[A PDE system for phase transitions and damage in  thermoviscoelasticity]
{``Entropic''
solutions to a thermodynamically consistent PDE system for phase
transitions and damage}

\author{Elisabetta Rocca}
\address{Elisabetta Rocca\\ Weierstrass Institute for Applied
Analysis and Stochastics\\ Mohrenstr.~39 \\ D-10117 Berlin \\
Germany\\ and \\
Dipartimento di Matematica \\ Universit\`a di Milano\\ Via Saldini 50 \\ I-20133 Milano\\ Italy}
\email{rocca@wias-berlin.de and elisabetta.rocca@unimi.it}

\author{Riccarda Rossi}
\address{Riccarda Rossi\\  DICATAM - Sezione di Matematica \\ Universit\`{a} di Brescia\\ Via Valotti 9\\ I-25133 Brescia\\ Italy}
\email{riccarda.rossi@unibs.it}

\date{19.11.2014}

\begin{abstract}
In this paper  
 we analyze a PDE system
modelling (non-isothermal) phase transitions
 and damage phenomena
in thermoviscoelastic materials.   The model is thermodynamically consistent: in particular,
no  {\em small perturbation
assumption} is  adopted, which results in the presence of
quadratic terms on  the right-hand side of the temperature equation, only estimated in $L^1$. The whole system has a highly nonlinear character.

We address the existence   of     a weak notion of solution, referred to as ``entropic'', where the temperature equation is formulated with the aid
of an entropy inequality, and of a total energy inequality.  This solvability concept
reflects the basic principles of
thermomechanics, as well as   the thermodynamical consistency of the model.  It
allows us to obtain \emph{global-in-time} existence theorems
 without imposing any restriction on the size
of the initial data.

We prove our results by passing to the limit in   a   time-discretization  scheme,
 carefully tailored to the nonlinear features of the PDE system (with its ``entropic'' formulation), and  of the a priori estimates performed on it.
Our  time-discrete  analysis  could be useful towards  the numerical study of this model.
\end{abstract}

\maketitle



  \vspace{4mm}

\noindent {\bf AMS (MOS) subject clas\-si\-fi\-ca\-tion:}\thspace
35D30,  74G25, 93C55, 82B26, 74A45.

\section{\bf Introduction}
We consider the following PDE system
\begin{align}
& 
\teta_t +\chi_t \teta +\rho\teta \dive
(\ub_t) -\dive(\condu( \teta) \nabla\teta) = g + a(\chi) \tensoret
\vism \tensoret + |\chi_t|^2 \quad\hbox{in }\Omega\times
(0,T),\label{eq0}
\\
&\ub_{tt}-\dive(a(\chi)\vism\tensoret+b(\chi)\elm\tensore-\rho\teta\mathbf{1})={\bf
f}\quad\hbox{in }\Omega\times (0,T),\label{eqI}
\\
&\chi_t +\mu \partial I_{(-\infty,0]}(\chi_t) -\mathrm{div}(|\nabla
\chi|^{p-2} \nabla \chi)  +W'(\chi) \ni - b'(\chi)\frac{\tensore
\elm \tensore}2 + \teta \quad\hbox{in }\Omega \times
(0,T),\label{eqII}
\end{align}
supplemented with the boundary conditions (here $n$ denotes the outward unit normal  to $\partial\Omega$)
\begin{equation}
\label{intro-b.c.}
\condu(\teta) \nabla \teta \cdot n  = h, \qquad \uu =0, \qquad  \partial_{n} \chi  =0 \qquad \text{on } \partial \Omega \times (0,T).
\end{equation}
Equations \eqref{eq0}--\eqref{eqII} were derived in \cite{rocca-rossi-deg} according to
\textsc{M.\ Fr\'emond}'s modeling approach  (see \cite{fremond,fre-newbook}). There, it was shown that
this PDE system describes (non-isothermal) phase transitions, or  (non-isothermal) damage,  in a material body
 occupying a reference
domain $\Omega \subset \R^d$, $d \in \{2,3\}$. We refer to
\cite{rocca-rossi-deg} for a quite detailed survey on the literature
on phase transition and damage problems in thermoviscoelasticity.
  In  \eqref{eq0}--\eqref{eqII},   the symbols $\teta$ and $\uu$ respectively denote the
absolute temperature of the system and the  small  displacement vector, while  $\chi$  is  an internal parameter:
its meaning depends on the phenomenon described by  \eqref{eq0}--\eqref{eqII}, which also determines the choices of
the coefficients $a$ and $b$ in the momentum equation \eqref{eqI}, and of the constant $\mu \in \{0,1\}$ in \eqref{eqII}.
  More precisely,
 \begin{compactitem}
 \item[-]
 the choices $a(\chi)=1-\chi$ and
$b(\chi)=\chi$ correspond to the case of
  \emph{phase transitions} in thermoviscoelastic materials: in this case, $\chi$ is the order parameter,
   standing for the local proportion of one of the two phases. We assume that $\chi$  takes values between $0$ and $1$, choosing $0$ and $1$
as reference values: in  the  case of phase transitions, $\chi=1$ stands for the liquid phase while $\chi=0$ for the solid one and
 one has $0 < \chi< 1$ in  the so-called \emph{mushy regions}.
 Unidirectionality, or irreversibility, of the phase transition process may be encompassed in the model
by taking $\mu=1$ in \eqref{eqII}, which ``activates'' the term  $\partial I_{(-\infty,0]}(\chi_t)$
(i.e.\ the subdifferential in the sense of convex analysis of the indicator function $ I_{(-\infty,0]}$, evaluated at $\chi_t$), yielding the constraint $\chi_t \leq 0$
 a.e.\ in $\Omega \times (0,T)$.
 The meaning  of  $a(\chi)=1-\chi$ and
$b(\chi)=\chi$ in  \eqref{eqI} is that,   in the \emph{purely} solid phase $\chi=0$ only the elastic energy, in addition to the thermal expansion
energy, contributes to the stress
$\sigma = a(\chi)\vism\tensoret+b(\chi)\elm\tensore-\rho\teta\mathbf{1}$
(where $\elm$ and $\vism$ are the elasticity and viscosity tensors, respectively).
 Instead,
in the \emph{purely} liquid, or ``viscous'', phase  $\chi=1$ only the viscosity contribution
 remains,  whereas
 in mushy regions both elastic and viscous effects are present.
\item[-]
The choices $a(\chi)=b(\chi)=\chi$ correspond to   \emph{damage}.
 In this case, $\chi$ is the damage parameter,  assessing the soundness of the material  microscopically, around a point
 in the material domain $\Omega$.   In fact, we have   $\chi=0$ in the presence of complete damage,
  while  $\chi$ takes the value $1$
  when the material is fully sound,
  and $0 < \chi < 1$  describes
\emph{partial damage}.
\end{compactitem}
\beruno The function \eruno $\condu$ in \eqref{eq0} is  the heat conductivity, $W$ in \eqref{eqII} is a  mixing   energy density, which we assume of the form
\[
W= \widehat \beta + \widehat \gamma \qquad \text{with } \widehat{\beta}: \mathrm{dom} (\widehat{\beta}) \to \R \text{ convex, possibly nonsmooth, and }
\widehat\gamma \in \mathrm{C}^2(\R),
\]
while $\mathbf{f}$ is a given bulk force,  and
$g$ and $h$ heat sources.  \beruno The $p$-Laplacian term in  \eqref{eqII} reflects the fact that
we are within a \emph{gradient theory} for phase transitions and damage, like in, e.g.,  \cite{abels,BoBo,bosch,fremond,fne,hk1,la,MieRou06,mrz,mt} where gradient regularizations
are adopted  in different contexts.  \eruno

 Observe that,  in the case when both coefficients $a(\chi)$ and $b(\chi)$ in the momentum equation degenerate to zero
  (which happens, for instance, with $a(\chi)=b(\chi)=\chi$, when complete damage occurs),  the
equation for $\uu$ \beruno looses \eruno its elliptic character. This leads to serious troubles as, for instance, no control of the term
$b'(\chi)\frac{\tensore
\elm \tensore}2 $ on the right-hand side of
\eqref{eqII} is possible.
That is why, in what follows we shall confine our analysis of system \eqref{eq0}--\eqref{eqII}
  only to  the case in which the functions $a,\, b \in \mathrm{C}^1(\R)$ are
 bounded from below away from 0 (cf.~\eqref{data-a} in Sec.\ \ref{s:main}).
  \berdue We refer the reader \erdue to our previous contribution \cite{rocca-rossi-deg}, where we deal
  with  complete damage  and  elliptic degeneracy  of the momentum equation,
  in a \emph{simplified} case.
  In fact, in  \cite{rocca-rossi-deg} we analyzed
   the following \emph{reduced} system
   \begin{equation}
   \label{reduced-intro}
   \begin{aligned}
& 
\teta_t +\chi_t \teta +\rho\teta \dive
(\ub_t) -\dive(\condu( \teta) \nabla\teta) = g 
 \quad\hbox{in }\Omega\times
(0,T),
\\
&\ub_{tt}-\dive(a(\chi)\vism\tensoret+b(\chi)\elm\tensore-\rho\teta\mathbf{1})={\bf
f}\quad\hbox{in }\Omega\times (0,T),
\\
&\chi_t +\mu \partial I_{(-\infty,0]}(\chi_t) -\mathrm{div}(|\nabla
\chi|^{p-2} \nabla \chi)  +W'(\chi) \ni - b'(\chi)\frac{\tensore
\elm \tensore}2 + \teta \quad\hbox{in }\Omega \times
(0,T),
\end{aligned}
\end{equation}
    where the quadratic  contributions in the velocities on the right-hand side in the internal energy balance \eqref{eq0} are neglected
    by means of the \emph{small perturbation assumption} (cf. \cite{germain}).

\beruno Let us also mention that, like in  \cite{rocca-rossi-deg} we confine our analysis
 to the case in which  the  thermal expansion  contribution to the free energy is a linear function of the
 temperature $\teta$,   and the thermal expansion coefficient $\rho$  is independent of $\chi$.
  The case of a  $\chi$-dependent coefficient has been treated for similar PDE systems,
  e.g.\  in \cite{BoBo}, where local-in-time results were obtained, and more recently in \cite{HR} where the existence of global-in-time weak solutions has been proved, but  under the \emph{small perturbation assumption}.  \eruno \beruno Nonetheless, let us mention
   that, especially in case of phase transition phenomena,
the choice of a constant $\rho$  is quite reasonable  (cf., e.g., \cite{KreRoSprWilm} for further comments on this topic).
\eruno
\paragraph{\bf Mathematical difficulties.}
 In this paper, instead,  we address the \emph{full} system \eqref{eq0}--\eqref{eqII}. Let us stress that,
since we keep
the quadratic terms  $a(\chi) \tensoret
\vism \tensoret$  and  $|\chi_t|^2$ on the right-hand side of \eqref{eq0}, the model
is \emph{thermodynamically consistent}, as shown in \cite{rocca-rossi-deg}. However,
\begin{compactitem}
\item[-] the highly nonlinear character of the whole system, with  the multivalued term $\partial I_{(-\infty,0]}(\chi_t)$
and the possibly nonsmooth contribution $\widehat \beta$ to  the energy $W$;
\item[-] the quadratic terms on the right-hand side of \eqref{eq0}, which make
it difficult to get suitable estimates on $(\teta,\uu,\chi)$,
\end{compactitem}
bring about severe difficulties in the analysis of  \eqref{eq0}--\eqref{eqII}. This is the reason why we are going to
develop an existence analysis only for a suitable weak solution concept for
\eqref{eq0}--\eqref{eqII}, which we illustrate in the following lines.
\paragraph{\bf The ``entropic'' formulation.}
 We resort to a weak solution notion  for  \eqref{eq0}--\eqref{eqII} partially mutuated from \cite{fpr09}. There,  a thermodynamically consistent
model for phase transitions, consisting of the temperature and of the phase parameter equations, was analyzed:  the
temperature equation, featuring quadratic terms on its right-hand side, was weakly formulated in terms
of an \emph{entropy inequality}
and of a \emph{total energy inequality}.
In the present  framework,  the pointwise  internal energy balance \eqref{eq0}
is thus replaced
 by this  \emph{entropy inequality}
  \begin{equation}
\label{entropy-ineq-intro}
\begin{aligned}
& \int_s^t \int_\Omega (\log(\teta) + \chi) \varphi_t  \dd x \dd r +
\rho \int_s^t \int_\Omega \dive(\uu_t) \varphi  \dd x \dd r
 - \int_s^t \int_\Omega  \condu(\teta) \nabla \log(\teta) \cdot \nabla \varphi  \dd x \dd r
\\ &
\leq \berdue \int_\Omega (\log(\teta(t))+\chi(t))\varphi(t) \dd x
 -
 \int_\Omega (\log(\teta(s))+\chi(s))\varphi(s) \dd x \erdue-  \int_s^t \int_\Omega \condu(\teta) \frac{\varphi}{\teta}
\nabla \log(\teta) \cdot \nabla \teta  \dd x \dd r
 \\ & \quad -
  \int_s^t  \int_\Omega \left( g +a(\chi) \eps(\uu_t)\vism \eps(\uu_t) + |\chi_t|^2 \right)
 \frac{\varphi}{\teta} \dd x \dd r-  \int_s^t \int_{\partial\Omega} h \frac\varphi\teta  \dd S \dd r,
\end{aligned}
\end{equation}
where $\varphi$ is a sufficiently regular,  \emph{positive}  test function, 
  coupled with   the following \emph{total energy  inequality}
\begin{equation}
\label{total-enid-intro} \mathscr{E}(\teta(t),\uu(t), \uu_t(t), \chi(t))
\leq
 \mathscr{E}(\teta(s),\uu(s), \uu_t (s), \chi(s))  + \int_s^t
\int_\Omega g \dd x \dd r
+ \int_0^t
\int_{\partial\Omega} h \dd S \dd r
  + \int_s^t \int_\Omega \mathbf{f} \cdot
\mathbf{u}_t \dd x \dd r \,,
\end{equation}
  where
\begin{equation}
 \label{total-energy-intro}
 \mathscr{E}(\teta,\uu,\uu_t,\chi):= \int_\Omega \teta \dd x +
  \frac12 \int_\Omega |\uu_t |^2 \dd x + \frac12\int_\Omega b(\chi(t)) \eps(\uu(t))\elm\eps(\uu(t))  \dd x
+\frac1p \int_\Omega |\nabla \chi|^p \dd x  + \int_\Omega W(\chi) \dd
x\,.
 \end{equation}
Both \eqref{entropy-ineq-intro} and \eqref{total-enid-intro}
are, in the general case (cf.\ Remark \ref{rmk:total-eniq} later on),  required to hold  for almost all $t \in (0,T]$ and almost
all $s\in (0,t)$, and for $s=0$.
This  formulation of the  heat equation  has been first developed in \cite{fei, boufeima} in the framework
 of  heat  conduction  in fluids, and then applied to a  phase transition model,
  also derived according to \textsc{Fr\'emond}'s approach
  \cite{fremond}, firstly in \cite{fpr09}. Successively, the so-called \emph{entropic} notion of solution has been
    used  to
prove global-in-time existence results  in  models for special materials like liquid crystals (cf. \cite{ffrs}, \cite{frsz1}, \cite{frsz2}),
 and more recently   in the analysis of  models for the evolution of non-isothermal  binary incompressible immiscible fluids (cf. \cite{ers}).
This notion of solution   for the temperature equation corresponds
exactly to the physically meaningful   requirement   that the system
should satisfy the second and first principle of Thermodynamics.
Indeed, one of the main advantages of this formulation  resides  in
the fact that the thermodynamically consistency of the model
immediately follows from the existence proof. It can be also shown
that this solution concept is consistent with the standard one,
 (cf.\ the discussion in Sec.\ \ref{glob-irrev}, in particular Remark \ref{rmk:consistent},   and in \cite{fpr09}).
\par
 From an analytical viewpoint, observe that the entropy
inequality \eqref{entropy-ineq-intro} has the advantage that all the
troublesome quadratic terms on the right-hand side of \eqref{eq0}
feature as  multiplied  by a negative test  function. This,  and the fact that \eqref{entropy-ineq-intro}
is an inequality,   allows
for upper semicontinuity arguments in the limit passage in  a
suitable approximation of
\eqref{entropy-ineq-intro}--\eqref{total-energy-intro}.

 In addition to
\eqref{entropy-ineq-intro}--\eqref{total-energy-intro}, the
  \emph{entropic formulation} of system \eqref{eq0}--\eqref{eqII} also consists of the momentum balance \eqref{eqI}, given pointwise a.e.\ in
  $\Omega \times (0,T)$, and of the internal variable equation  \eqref{eqII}. The latter is
required to hold pointwise almost everywhere in the  reversible case
$\mu=0$. In the irreversible case $\mu=1$,  we shall confine the
analysis to the case  in which $\widehat \beta $ is the indicator
function $ I_{[0,+\infty)} $ of $[0,+\infty)$,   hence
$W(\chi) = I_{[0,+\infty)}(\chi) + \widehat\gamma(\chi)$.
 For reasons expounded in Sec.\ \ref{glob-irrev}, we shall have to weakly formulate \eqref{eqII}
in terms of
the requirement $\chi_t \leq 0$ a.e.\ in $\Omega\times (0,T)$, of
  the \emph{one-sided variational inequality}
\begin{align}
 \label{ineq-system2-intro}
 \int_\Omega \left( \chi_t-\mathrm{div}(|\nabla \chi|^{p-2} \nabla \chi)+
 \xi + \gamma(\chi)  + b'(\chi)\frac{\varepsilon(\ub)
\elm\varepsilon(\ub)}{2} -\teta\right) \psi \dd x  \geq 0  \text{ for all } \psi \in W^{1,p}(\Omega) \text{ with }\psi \leq 0,
\end{align}
almost everywhere in $(0,T)$  (where $\gamma:=\widehat\gamma
'$),  and of  the  \berunodue \emph{energy-dissipation inequality (for the internal variable $\chi$)} \erunodue
\begin{align}
 \label{ineq-system3-intro}
 \begin{aligned}
 \int_s^t   \int_{\Omega} |\chi_t|^2 \dd x \dd r  & +\io\left(
  \frac1p  |\nabla\chi(t)|^p +  W(\chi(t))\right)\dd x\\ & \leq\io\left(
  \frac1p |\nabla\chi(s)|^p+ W(\chi(s))\right)\dd x
  +\int_s^t  \int_\Omega \chi_t \left(- b'(\chi)
  \frac{\varepsilon(\ub)\elm\varepsilon(\ub)}2
+\teta\right)\dd x \dd r
\end{aligned}
\end{align}
for all $t \in (0,T]$ and almost all  $s \in (0,t)$,  with
$\xi$  a selection in the  (convex analysis) subdifferential  $\partial\widehat \beta (\chi)= \partial
I_{[0,+\infty)}(\chi)$
of $I_{[0,+\infty)}$.
In \cite[Prop.\ 2.14]{rocca-rossi-deg} (see also \cite{hk1}), we prove that,
under additional regularity properties
any weak solution in fact fulfills
\eqref{eqII} pointwise.

 Let us also mention that other approaches to the weak
solvability of coupled PDE systems with an  $L^1$-right-hand side
are available in the literature: in particular, we refer  here to
\cite{zimmer} and \cite{roubiSIAM10}.    In \cite{zimmer},
the notion of {\em renormalized solution} has been used in order to
prove a global-in-time  existence  result for a nonlinear system in
thermoviscoelasticity.   In \cite{roubiSIAM10} the focus
is on rate-independent processes  coupled  with viscosity and inertia
in the displacement equation, and  with  the temperature equation. There
the internal variable equation is not of gradient-flow type as
\eqref{eqII},  but instead features a $1$-positively homogeneous
dissipation potential. For the resulting PDE system,  a weak
solution concept partially mutuated from the theory of
rate-independent processes  by \textsc{A.\ Mielke} (cf., e.g.,
\cite{MieTh04}) is analyzed. An existence result is proved combining
  techniques for rate-independent evolution, with
Boccardo-Gallou\"et type estimates of the temperature gradient in
the heat equation with  $L^1$-right-hand side.

  \paragraph{\bf Our existence results.}
    The main results of this paper,  \underline{Theorems \ref{teor3}} and \underline{\ref{teor1}},  
state  the existence of {\em entropic solutions} for system
(\ref{eq0}--\ref{eqII}), \beruno supplemented with the boundary
conditions \eqref{intro-b.c.} (cf.\ Remark \ref{rmk:discussion-bc}),
\eruno in the irreversible ($\mu=1$) and reversible ($\mu=0$) cases.

More precisely,  in the case of unidirectional evolution for
$\chi$
 we can prove the existence of a global-in-time
entropic solution  (i.e.\ satisfying the  {\em entropy}
\eqref{entropy-ineq-intro} and the \emph{total energy} 
\eqref{total-enid-intro} inequalities,
 the   (pointwise)  momentum balance
\eqref{eqI}, the    {\em one-sided variational inequality}
\eqref{ineq-system2-intro} and the \emph{energy}
\eqref{ineq-system3-intro} inequalities   for $\chi$).
 We work under fairly general assumptions on the nonlinear
functions in \eqref{eq0}--\eqref{eqII}. More precisely, we require
that $a$ and $b$ are sufficiently smooth and bounded from below by a
positive constant, 
 $b$ convex,
 and we
standardly assume that $W =  I_{[0,+\infty)} + \widehat\gamma$,
 with $\widehat\gamma$ smooth and $\lambda$-convex.  A
crucial role is played by the requirement that the heat
 conductivity function $\condu = \condu(\teta)$  grows at least  like
$\teta^\kappa$ with
  $\kappa>1$.
  \beruno
The reader
 may consult \cite{zr} for various examples in which a superquadratic growth in $\teta$ for the heat
conductivity $\condu$ is imposed,  whereas  \cite{kle12}  discusses experimental findings
according to which a class of polymers exhibit a subquadratic growth for $\condu$.
 Another crucial hypothesis is that the exponent $p$ in the gradient regularization of the
  equation for $\chi$ fulfills $p>d$.   Gradient regularizations of $p$-Laplacian type, with $p>d$,
  have been adopted for several damage models, cf.\ e.g.\ \cite{bmr,MieRou06,mrz,krz2}.
   This, mathematically speaking, ensures that $\chi $ is
  estimated in $W^{1,p}(\Omega) \subset \mathrm{C}^0
  (\overline\Omega)$.
   From the viewpoint of physics, since the gradient of  $\chi$
accounts for
interfacial energy effects in phase transitions, and  for the influence of damage at a material point,
undamaged in its neighborhood, in  damage models, we may observe that the term $\frac{1}{p}|\nabla\chi|^p$ models
nonlocality of the phase transition or the damage process. 
 \eruno

Moreover, under some restriction on $\kappa$  (i.e.\ $\kappa\in
(1,5/3)$ for space dimension $d=3$), we can also obtain an enhanced
regularity for $\teta$ and that  conclude that the {\em total energy
inequality }  actually holds as an {\em equality}.

In the reversible case ($\mu=0$), instead,   under the same
assumptions above described (but with a general $\widehat\beta$), we
improve the estimates, hence the regularity, of the internal
variable $\chi$. Therefore, we prove the existence of a weak
formulation of \eqref{eq0}--\eqref{eqII}, featuring, in addition to
\eqref{entropy-ineq-intro}, \eqref{total-enid-intro}, and
\eqref{eqI}, a \emph{pointwise} formulation of
 equation \eqref{eqII}. Again, in the case of the aforementioned
 restriction on $\kappa$, we enhance the time-regularity of $\teta$. What is more,   also exploiting the improved formulation of the equation for
 $\chi$, we are able to conclude existence for a stronger
  formulation of the
heat equation \eqref{eqI}, of variational type.
Instead,  a  uniqueness result  seems to be out of reach, at the
moment, not only in the irreversible  but also in the reversible
 cases (cf. Remarks~\ref{rmk:uni-irr}  and   \ref{rmk:uni-rev}). Only
for  the \emph{isothermal} reversible system  a continuous
dependence result, yielding uniqueness,   can be proved exactly
like in \cite[Thm.3]{rocca-rossi-deg}.

Finally, in the last Section~\ref{s:final}  we address the
analysis of system \eqref{eq0}--\eqref{eqII}, with $\mu=1$, in the
case the $p$-Laplacian regularization in \eqref{eqII} is replaced by
the standard Laplacian operator. We approximate it by adding a
$p$-Laplacian term, modulated by a small parameter $\delta$, on the
left-hand side of \eqref{eqII}, so that Thm.\ \ref{teor1} guarantees
the existence of approximate solutions
$(\teta_\delta,\uu_\delta,\chi_\delta)$.  Then, we  let $\delta$
tend to zero.
 In this context, the enhanced elliptic
regularity estimates on the momentum equation exploited in the proof
of Thm.\ \ref{teor3}, and
 which  would here yield some suitable compactness
for the quadratic term $a(\chi_\delta) \eps(\partial_t
\uu_\delta)\vism \eps(\partial_t\uu_\delta) $ on the right-hand side
of \eqref{eq0}, are no longer available. In fact,  they rely on the
requirement $p>d$. A crucial step for proving the existence of (a
slightly weaker notion of) entropic solutions to system
\eqref{eq0}--\eqref{eqII} (cf.\ \underline{Theorem \ref{teorSec6}}), then
consists in  deriving  some suitable strong convergence for
$(\partial_t \uu_\delta)_\delta$ with an ad hoc technique, strongly
relying on the
  fact that $\mu=1$, and
 on the additional assumption that  $b$ is  non-decreasing.

 Our main existence results Thms.\ \ref{teor3} and \ref{teor1}
are proved  by passing to the limit in a
 time-discretization  scheme, unique for  the
reversible and the irreversible cases,
 carefully tuned to the nonlinear features of the PDE system. In
 particular, it is devised in such a way as to obtain  that
the piecewise constant and piecewise linear interpolants of the
discrete
 solutions satisfy  the discrete versions of the entropy inequality \eqref{entropy-ineq-intro}, of   total energy inequality
 \eqref{total-energy-intro}, and of the energy inequality
 \eqref{ineq-system3-intro} in the case $\mu=1$. Moreover, with
 delicate calculations we are also able to translate on the
 time-discrete level a series of a priori estimates on the
 heat equation,  having  a nonlinear character.
%
This   detailed  time-discrete  analysis could be interesting in view of  further
numerical studies of this model.

\beo For the limit passage we resort to various compactness results
available in the literature, and additionally prove the compactness
Theorem \ref{th:mie-theil},  based on the theory of  Young measures
with values in infinite-dimensional (reflexive) Banach spaces. \eo


\noindent \textbf{Plan of the paper.} In Section~\ref{s:main} we fix
some notation, state some preliminaries that will be used in the
rest of the paper, list our assumptions on the data as well as our
  main global-in-time existence results.
In Section~\ref{s:aprio} we perform a series of \emph{formal}
a-priori estimates on our system. We render them rigorously  
in Section~\ref{s:time-discrete}, where we set up our time-discrete
scheme. \underline{Theorems \ref{teor3}} and \underline{\ref{teor1}}
are proved by passing to the limit in the approximated entropy and
energy inequality, as well as in the discretized versions of
\eqref{eqI} and \eqref{eqII}, throughout Sec.\ \ref{s:pass-limit}.
Section \ref{s:final} is then devoted to the proof of
\underline{Theorem \ref{teorSec6}}. \beo Finally, the Appendix
contains  a short recap of the theory of Young measures in
infinite-dimensional Banach spaces, and the proof of Theorem
\ref{th:mie-theil}. \eo


\section{\bf Setup and results}
\label{s:main}
After fixing some notation and results which shall be used throughout the paper, in Section \ref{ss:assumptions}
 we collect
 our working assumptions on the nonlinear functions  $\condu$, $a$, $b$, and  $W$ in the PDE system \eqref{eq0}--\eqref{eqII},  and on the data. Then,
 in Secs.\ \ref{glob-irrev} and \ref{ss:glob-rev} we discuss the weak formulations of
 (the
initial-boundary value problem  for)
 \eqref{eq0}--\eqref{eqII} in the irreversible and reversible cases, respectively
  corresponding to
$\mu=1$ and $\mu=0$ in  \eqref{eqII}.

 \subsection{Preliminaries}
\label{ss:prelims}
\begin{notation}
\label{not:2.1} \upshape
 Throughout the paper, given a
Banach space $X$
we shall
denote by $\|\cdot\|_{X}$ 
its norm,
and  use the symbol $\pairing{}{X}{\cdot}{\cdot}$ for the duality
pairing between $X'$ and $X$.
 Moreover,  we shall denote by
  ${\rm BV}([0,T];X)$ (by $\mathrm{C}^0_{\mathrm{weak}}([0,T];X)$, respectively),
 the space
of functions from $[0,T]$ with values in $ X$ that are defined at
every  $t \in [0,T]$ and  have  bounded variation on  $[0,T]$  (and
are \emph{weakly} continuous   on  $[0,T]$, resp.).

Let $\Omega \subset \R^d$ be a bounded domain,
 $d \in \{2,3\}$. We set $Q:= \Omega \times (0,T)$.
 We   identify both
$L^2 (\Omega)$ and $\Ha$ with their dual spaces, and denote by
$(\cdot,\cdot)$ the scalar product in $\R^d$, by
$(\cdot,\cdot)_{L^2(\Omega)}$ both the scalar product in
 $L^2(\Omega)$,
and  in $\Ha$, and by
 $\boZ$ and $\boY$
the  spaces
$$
\begin{aligned}
& \boZ:=\{\vv \in H^1(\Omega;\R^d) \,:\ \vv= 0 \ \hbox{ on
}\partial\Omega \,\}, \text{ endowed with the  norm } \|
\vv\|_{H_0^1(\Omega;\R^d)}^2: = \int_{\Omega} \e(\vv) \colon \e(\vv)\,
\dd x,
\\
&\berdue  \boY:=\{\vv \in H^2(\Omega;\R^d)\,:\ \vv ={0} \ \hbox{ on
}\partial\Omega \,\}. \erdue
\end{aligned}
$$
 Note that   $\|\cdot\|_{H_0^1(\Omega;\R^d)}$ is a  norm  equivalent to the standard one on $H^1(\Omega;\R^d)$.
We will use the symbol $\mathcal{D} (\overline Q)$ for the space of
the $C^\infty$-functions with compact support on $Q:= \Omega \times
(0,T)$ and for $q>1$ we will adopt the notation \begin{equation}
\label{label-added}
 W_+^{1,q}(\Omega):= \left\{\zeta \in
W^{1,q}(\Omega)\, : \ \zeta(x) \geq 0  \quad \foraa\, x \in
\Omega \right\}, \quad \text{ and analogously for }
W_-^{1,q}(\Omega).
\end{equation}

We denote by $A_p$ the $p$-Laplacian operator with zero Neumann
boundary conditions, viz.\
\[
\text{$ A_p: W^{1,p}(\Omega) \to W^{1,p}(\Omega)'$ given by
$\pairing{}{W^{1,p}(\Omega)}{A_pu}{v}:= \int_\Omega |\nabla u|^{p-2}
\nabla u \cdot \nabla v \dd x$\,.}
\]
In the weak formulation of the momentum equation \eqref{eqI},
besides $\opjname$ and $\ophname$
 we will also make use of the
operator
\begin{equation}
\label{op_ciro} \mathcal{C}_\rho: L^2 (\Omega) \to
H^{-1}(\Omega;\R^d) \quad \text{defined by} \quad
\pairing{}{H^1(\Omega;\R^d)}{\ciro(\theta)}{\vv}:= - \rho \int_\Omega
\theta \dive (\vv)\, \mathrm{d}x.
\end{equation}



Finally, throughout the paper we shall denote by the symbols
$c,\,c',\, C,\,C'$  various positive constants depending only on
known quantities. Furthermore, the symbols $I_i$,  $i = 0, 1,... $,
will be used as place-holders for several integral terms popping in
the various estimates: we warn the reader that we will not be
self-consistent with the numbering, so that, for instance, the
symbol $I_1$ will occur several times with different meanings.
\end{notation}
\noindent \textbf{Recaps of mathematical elasticity.}
The elasticity and viscosity  tensors fulfill
\begin{equation}
\label{funz_g-l} \elm=(e_{ijkh}), \,   \vism=(v_{ijkh})
  \in \mathrm{C}^{1}(\Omega;\R^{d \times d \times d \times d})\,
\end{equation}
with coefficients satisfying the classical symmetry and ellipticity
conditions (with the usual summation convention)
\begin{equation}
\label{ellipticity}
\begin{aligned}
& e_{ijkh}=e_{jikh}=e_{khij}\,,\quad  \quad
v_{ijkh}=v_{jikh}=v_{khij}
\\
&  \exists \, \alpha_0>0 \,:  \qquad e_{ijkh} \xi_{ij}\xi_{kh}\geq
\alpha_0\xi_{ij}\xi_{ij} \quad   \forall\, \xi_{ij}\colon \xi_{ij}=
\xi_{ji}
\\
& \exists \, \beta_0>0 \,:  \qquad v_{ijkh} \xi_{ij}\xi_{kh}\geq
\beta_0\xi_{ij}\xi_{ij} \quad   \forall\, \xi_{ij}\colon \xi_{ij}=
\xi_{ji}.
\end{aligned}
\end{equation}
Observe that with \eqref{ellipticity} we also encompass in our analysis the case of
an anisotropic and inhomogeneous material.

In order to give the variational formulation of the momentum equation, we need to introduce the bilinear forms
related to the $\chi$-dependent elliptic operators appearing
in~\eqref{eqI}.
 Hence,  given a \emph{non-negative} function $\eta
\in L^\infty (\Omega)$ (later, $\eta= a(\chi)$ or $\eta = b(\chi)$),
  let us consider the  bilinear
symmetric forms $\bilh{\eta}{\cdot}{\cdot}, \,
\bilj{\eta}{\cdot}{\cdot}: \, \boZ \times \boZ \to \RR$ defined for
all $ \ub, \vb \in \boZ$ by
\begin{equation}
\label{bilinear-forms}
\begin{aligned}
& \bilh {\eta}{\uu}{\vv}:= \pairing{}{H^1(\Omega;\R^d)}{-\dive(\eta
\elm \tensore)}{\vv} =\sum_{i,j, k, h=1}^d \int_\Omega\eta\,
e_{ijkh}\,\varepsilon_{kh}(\ub)\varepsilon_{ij}(\vb),
\\
& \bilj {\eta}{\uu}{\vv}:= \pairing{}{H^1(\Omega;\R^d)}{-\dive(\eta
\vism \tensore)}{\vv}= \sum_{i,j, k, h=1}^d \int_\Omega\eta\,
v_{ijkh}\,\varepsilon_{kh}(\ub)\varepsilon_{ij}(\vb).
\end{aligned}
\end{equation}
Thanks to \eqref{ellipticity}
and Korn's inequality (see
eg~\cite[Thm.~6.3-3]{ciarlet}), the forms
$\bilh{\eta}{\cdot}{\cdot}$ and $\bilj{\eta}{\cdot}{\cdot}$ fulfill
\begin{equation}
\label{korn}
\exists\, C_1>0 \ \forall\, \ub,\,\vb \in  \boZ\, : \quad
\begin{cases}
\bilh{\eta}{\uu}{\uu}
 \geq \inf_{x \in \Omega}(\eta(x))\,C_1\Vert{\bf
u}\Vert^2_{\V}, \\  \bilj{\eta}{\uu}{\uu}\geq  \inf_{x \in
\Omega}(\eta(x))\,C_1\Vert{\bf u}\Vert^2_{\V}.
\end{cases}
\end{equation}
It follows from \eqref{funz_g-l} that they are also continuous, namely
\begin{align}
\label{a:conti-form}
\exists\, C_2>0 \ \forall\, \ub,\,\vb \in  \boZ\, : \quad
 |\bilh{\eta}{\uu}{\vv}| +
|\bilj{\eta}{\uu}{\vv} |
 \leq C_2 \| \eta \|_{L^\infty (\Omega)} \| \mathbf{u} \|_{\V} \| \mathbf{v}
 \|_{\V}.
\end{align}
 We  shall denote by $\ophname(\eta\, \cdot):\boZ\to
H^{-1}(\Omega;\R^d)$ and $\opjname(\eta\, \cdot):\,\boZ\to
H^{-1}(\Omega;\R^d)$ the linear operators associated with the forms
$\bilh{\eta}{\cdot}{\cdot}$
 and $\bilj{\eta}{\cdot}{\cdot}$, respectively,  that is
\begin{equation}
\label{operator-notation-quoted}
\pairing{}{H^1(\Omega;\R^d)}{\oph{\eta}{\vv}}{\ww}:=
\bilh{\eta}{\vv}{\ww},
 \quad
\pairing{}{H^1(\Omega;\R^d)}{\opj{\eta}{\vv}}{\ww}:=
\bilj{\eta}{\vv}{\ww} \qquad \text{for all }{\bf v},\,{\bf w}\in
\boZ.
\end{equation}
\begin{remark}[A caveat on notation]
\label{rmk:caveat} \berdue Actually, it would be more appropriate to
use the symbols $\mathsf{e}_\eta(\cdot,\cdot)$ and $
\mathsf{v}_\eta(\cdot,\cdot)$
 in place of $\bilh{\eta}{\cdot}{\cdot}, \,
\bilj{\eta}{\cdot}{\cdot}$
 to signify that for fixed
$\eta \in L^\infty(\Omega)$, the bilinear forms defined in
\eqref{bilinear-forms} act on the pair $(\mathbf{u}, \mathbf{v})$.
 However,
 in most occurrences, we would use this notation with   $\eta$
 replaced by the ``heavier''  symbols $a(\chi)$ or $b(\chi)$. Thus,
 for notational simplicity we prefer to stay with the less correct
 notation from \eqref{bilinear-forms}.
The same considerations apply to the operators defined in
\eqref{operator-notation-quoted}.  \erdue
\end{remark}
It can be checked via an approximation argument
  that the following
regularity results hold
\begin{subequations}
\begin{align}
& \label{reg-pavel-a} \text{if $\eta \in L^\infty(\Omega)$ and
${\uu}\in \boZ$, \ then  \, $\oph{\eta}{\uu}, \, \opj{\eta}{\uu} \in
H^{-1} (\Omega;\R^d)$,} \\
& \label{reg-pavel-b} \text{if \berdue $\eta \in  W^{1,d+\epsilon}
(\Omega)$ for some $\epsilon>0$ \erdue and ${\uu}\in \boY$, \ then  \,
$\oph{\eta}{\uu}, \, \opj{\eta}{\uu} \in \Ha$.}
\end{align}
\end{subequations}

\beo Finally, let us recall
 the following elliptic regularity result, holding in the case $\Omega$ has a $\mathrm{C}^2$-boundary (cf.\
 \eqref{smoothness-omega} below) and also due to \eqref{funz_g-l}, namely
\begin{align}
\label{cigamma}
\begin{aligned}
\exists \, C_3,\, C_4>0 \quad \forall\,  \uu \in
\boY\, : \qquad
\begin{cases}
 C_{3} \| \uu \|_{H^2(\Omega)}  \leq \|\dive (\elm\eps
(\uu))\|_{L^2(\Omega)} \leq C_{4} \| \uu \|_{H^2(\Omega)}\,,\\
C_{3} \| \uu \|_{H^2(\Omega)}  \leq \|\dive (\beo\vism\eo\eps
(\uu))\|_{L^2(\Omega)}\leq C_{4} \| \uu \|_{H^2(\Omega)}\,.
\end{cases}
\end{aligned}
\end{align}
 For this, we refer
e.g.~\cite[Lemma~3.2,  p.\
260]{necas}) or \cite[Chap.\ 6, p.\ 318]{Hughes}.
\eo

\noindent
\textbf{Useful  inequalities.} In order to make the paper as self-contained as possible,
we recall  here  the   Gagliardo-Nirenberg inequality
(cf.~\cite[p.~125]{nier}) in a particular case: for
 all $r,\,q\in [1,+\infty],$ and for all $v\in L^q(\Omega)$ such that
$\nabla v \in L^r(\Omega)$, there holds
\begin{equation}\label{gn-ineq}
\|v\|_{L^s(\Omega)}\leq C_{\mathrm{GN}}
\|v\|_{W^{1,r}(\Omega)}^{\theta} \|v\|_{L^q(\Omega)}^{1-\theta} \qquad
\text{ with } \frac{1}{s}=\theta
\left(\frac{1}{r}-\frac{1}{d}\right)+(1-\theta)\frac{1}{q}, \ \  0
\leq \theta \leq 1, \end{equation}
 the positive constant $C_{\mathrm{GN}}$ depending only on
$d,\,r,\,q,\,\theta$. Combining the compact embedding
\begin{equation}
\label{dstar}
 \boY \Subset W^{1,d^\star{-}\eta}(\Omega;\R^d),
\quad \text{with } d^{\star}=
\begin{cases} \infty & \text{if }d=2,
\\
6 & \text{if }d=3,
\end{cases}
 \quad \text{for all $\eta >0$},
\end{equation}
(where for $d=2$ we mean that $\boY \Subset W^{1,q}(\Omega;\R^d)$ for all $1 \leq q <\infty$),
   with \cite[Thm. 16.4, p. 102]{LM}, we have
\begin{equation}
\label{interp} \forall\, \varrho>0 \ \ \exists\, C_\varrho>0 \ \
\forall\, \uu \in \boY\,: \ \
\|\e(\uu)\|_{L^{d^\star{-}\eta}(\Omega)}\leq \varrho
\|\uu\|_{H^2(\Omega)}+C_\varrho\|\uu\|_{L^2(\Omega)}.
\end{equation}
We will also use the following \emph{nonlinear}  Poincar\'{e}-type inequality
 (cf.\  e.g.\ \cite[Lemma 2.2]{gmrs}), with  $m(w)$ the mean value of $w$:
 \begin{equation}
 \label{poincare-type}
 \forall\, q>0 \quad \exists\, C_q >0 \quad \forall\, w \in H^1(\Omega)\, : \qquad
 \| |w|^{q} w \|_{H^1(\Omega)} \leq C_q (\| \nabla (|w|^{q} w )\|_{L^2(\Omega)} + |m(w)|^{q+1})\,.
\end{equation}

\subsection{Assumptions}
\label{ss:assumptions} \noindent In most of this paper, we shall
\beo work under the following
\par
\noindent
\textbf{Hypothesis (0).}
We suppose that
\begin{equation}
\label{smoothness-omega}
\Omega\subset\RR^d, \quad d\in \{2,3\} \ \ \text{is
 a bounded connected domain,   with $\mathrm{C}^2$-boundary $\partial\Omega$}
\end{equation}
and that
the viscosity tensor is given by
\begin{align}
\label{eqn:visc}
\vism=\omega \elm  \qquad \text{for a  constant } \omega>0.
\end{align}
\begin{remark}
\label{rmk:smoothness}
\upshape
The smoothness requirement   \eqref{smoothness-omega}
will allow us to apply the elliptic regularity results
in~\eqref{cigamma}.

Concerning   \eqref{eqn:visc},   let us mention in advance
that  it will only be used
in the proof of the $H^2(\Omega;\R^d)$-regularity for the discrete displacements, cf.\ Lemma \ref{lemma:ex-discr} ahead ensuring the existence of solutions
to the time-discretization scheme for system \eqref{eq0}--\eqref{eqII}. As we will see,  this regularity property is crucial for the rigorous proof of the elliptic regularity estimate
for the displacements, see the \emph{Fifth  estimate} (formally) derived in Sec.\ \ref{s:aprio}, which is in turn  essential in the proof of our main results, Theorems
\ref{teor3} and \ref{teor1}.

 Instead, in the proof of Thm.\ \ref{teorSec6} we  will not need to perform the aforementioned  elliptic regularity argument, at the price of proving the existence of a  weaker
notion of solution for the irreversible system (cf.\  \eqref{reg-uSec6} and  Remark~\ref{omegaLip}  in Sec.\ \ref{s:final}).
Hence, we will be able to dispense with
 conditions \eqref{smoothness-omega} and \eqref{eqn:visc}. This is the reason why, despite $\vism$ and $\elm$ are a multiple of each other
by \eqref{eqn:visc}, we have kept the two symbols $\vism$ and $\elm$ throughout the paper.
\end{remark}
\eo

\noindent  We  list  below
our basic assumptions on the functions $\condu$, $a$, $b$, and  $W$ in system
\eqref{eq0}--\eqref{eqII}. 

\noindent \textbf{Hypothesis (I).} We suppose that
\begin{align}
\label{hyp-K}
\begin{aligned}
& \text{the function }   \condu:[0,+\infty)\to(0,+\infty)  \  \text{
is
 continuous and}
\\
& \exists \, c_0, \, c_1>0 \quad   \kappa>1   \ \
\forall\teta\in[0,+\infty)\, :\quad
c_0 (1+ \teta^{\kappa}) \leq \condu(\teta) \leq c_1 (1+\teta^{\kappa})\,.
\end{aligned}
\end{align}
We will  denote by $\widehat{\condu}$ the primitive $\widehat{\condu} (x):= \int_0^x \condu(r) \dd r $ of $\condu$.

\noindent \textbf{Hypothesis (II).} We require
\begin{align}
& \label{data-a} a \in \mathrm{C}^1(\R), \ b \in \mathrm{C}^2(\R)  \
\text{ and } \exists\, c_2>0 \, : \quad a(x), \ b(x) \geq c_2 \ \text{for all } x
\in\RR
\end{align}
 and that the function $b$ is convex.  The latter requirement could
 be weakened to $\lambda$-convexity, i.e.\
 that $b{''}$ is bounded from below (cf.\ also \eqref{lambda-convex}), see.\ Remark  \ref{b-l-convex} later on.
\par
\noindent \textbf{Hypothesis (III).} We suppose that the potential
$W$ in~\eqref{eqII} is given by $ W= \widehat{\beta} +
\widehat{\gamma},$ where
\begin{align}
 \label{databeta}  &
 \berdue \widehat\beta: \R \to \R \cup \{ +\infty\}  \,
  \text{ has nonempty domain $\text{\rm dom}(\widehat{\beta})$,  is l.s.c. and  convex },  \erdue 
  \quad \widehat{\gamma} \in
 {\rm C}^{2}(\R), \quad
  \\
  \label{data-W}
  &\exists\, c_W \in \R \qquad
 \beo W(r)\geq 
 c_W \eo\quad \forall r\in {\rm dom}(\widehat\beta)\,.
\end{align}
 Moreover, we impose that
\begin{equation}
\label{lambda-convex}
\exists\, \lambda >0 \ \ \forall\, r \in \R\, : \ \ \widehat{\gamma}{''}(r) \geq -\lambda.
\end{equation}
Hereafter, we shall use the notation
\[
\beta:=\partial\widehat\beta, \qquad
 \gamma:=
\widehat\gamma'.
\]
Observe that, we have not required that  $\text{\rm dom}(\widehat{\beta}) \subset [0,+\infty)$, which would  enforce the
(physically feasible) positivity of the  phase/damage variable $\chi$. In fact, for the analysis of the irreversible case (i.e.\ with
$\mu=1$), we will have to confine the discussion to the case $\widehat{\beta} = I_{[0,+\infty)}$, cf.\ Hypothesis (IV)
later on.
Instead, in the reversible case $\mu=0$, we will allow for  a general $\widehat\beta$ (complying with Hypothesis (III)).
\begin{remark}[A generalization of the $p$-Laplacian]
\label{rmk:general-p-Lapl}
\upshape
In fact, our analysis of system \eqref{eq0}--\eqref{eqII}  extends to the case \berdue that \erdue the
$p$-Laplacian operator $-\mathrm{div} (|\nabla \chi|^{p-2} \nabla \chi)$, with $p>d$, is replaced
by an  elliptic operator
$\opchi :W^{1,p} (\Omega ) \to W^{1,p} (\Omega )^*$ of the form
\begin{equation}
\label{p-Lapl-general}
\pairing{}{W^{1,p}(\Omega)}{\opchi (\chi)}{v} := \io \nabla_\zeta \phi(x,\nabla \chi(x))\cdot\nabla
v(x ) \dd x,
\end{equation}
where $\phi:\Omega\times \RR^d\to
[0,+\infty) $ is a Carath\'eodory integrand such that
\[
\begin{aligned}
 &  \text{the map }\phi(x,\cdot):  \R^d\to [0,+\infty) \ \text{ is convex, with $\phi(x,0)=0$,  and in
$ \mathrm{C}^1(\RR^d)$ for a.a.\ $x \in \Omega$},
\\
&
\exists\, c_3,\,c_4,\,c_5>0  \ \ \forae\, x \in \Omega \ \
\forall\, \zzeta \in \R^d\, : \ \
\left\{
\begin{array}{ll} &
\phi(x,{\zzeta})\geq
c_3|{\zzeta}|^p-c_4,
\\
 &  |\nabla_\zeta \phi(x,{\zzeta})|\leq c_5(1+|{
\zzeta}|^{p-1})\,.
\end{array}
\right.
\end{aligned}
\]
This more general framework was analyzed in \cite{rocca-rossi-deg}, to which we refer the reader for all details.
\end{remark}

\paragraph{\bf Problem and Cauchy data.}
We suppose that  the data $\mathbf{f}$, $g$, and $h$
fulfill
\begin{align}
\label{bulk-force} & \mathbf{f}\in L^2(0,T;\Ha),
\\
 \label{heat-source} &  g \in L^1(0,T;L^1(\Omega)) \cap L^2 (0,T; H^1(\Omega)'),\quad g\geq 0 \quad\hbox{a.e.  in }\Omega\times (0,T)\,,
 \\
 \label{dato-h}
 & h \in L^1 (0,T; L^2(\partial \Omega)), \quad h \geq 0 \quad\hbox{a.e.  in }\partial \Omega\times (0,T)\,,
\end{align}
and that the initial data comply with
\begin{align}
& \label{datoteta} \teta_0 \in L^{1}(\Omega), \quad \exists\,
\teta_*>0\,: \quad\inf_\Omega \teta_0\geq\teta_*>0\,, \quad
\log\teta_0\in L^1(\Omega),
\\
 \label{datou}
&\ub_0\in  \boY,\quad \vb_0\in \beo H_0^1(\Omega;\RR^d)\,, \eo
\\
 & \label{datochi}
 \chi_0\in   W^{1,p}(\Omega),\quad \widehat{\beta}(\chi_0)\in
L^1(\Omega).
\end{align}
\berdue Let us mention in advance that the strict positivity
requirement on $\teta_0$ and the non-negativity of $g$ and $h$ serve
to the purpose of ensuring the existence of an entropic solution
$(\teta,\uu,\chi)$ to (the initial-boundary value problem for
system) \eqref{eq0}--\eqref{eqII}, with $\teta$ \emph{strictly
positive}. The latter property has a crucial physical meaning, as
$\teta$ is the \emph{absolute temperature} of the system. It also
underlies our notion of weak solution for the heat equation,
involving the term $\log(\teta)$. \erdue

\subsection{A global existence  result for the  irreversible system}
\label{glob-irrev}
Before stating precisely
our notion of weak solution to (the initial-boundary value problem for) system \eqref{eq0}--\eqref{eqII}
 in  the case of  \emph{unidirectional evolution}, 
 let us briefly motivate the  weak formulations for the heat balance equation \eqref{eq0}, and for the
 phase/damage parameter subdifferential inclusion
 \eqref{eqII} (with $\mu=1$). They will be coupled with the  pointwise (in time and space) formulation of the
 momentum equation \eqref{eqI}  (cf.\ \eqref{weak-momentum}  later on).
\paragraph{\bf Entropy and total energy inequalities for the heat equation.}
For \eqref{eq0}, we adopt the weak formulation of  proposed in
\cite{boufeima, fei, fpr09}. It consists of a so-called ``entropy
inequality", and of a  \berdue ``total energy (in)equality". \erdue The
former is obtained by formally dividing \eqref{eq0} by $\teta$, and
testing it by a smooth test function $\varphi$. Integrating over
space and time leads to
\begin{equation}
\label{later-4-comparison}
\begin{aligned}
\int_0^T \int_\Omega \big(\partial_t \log(\teta) + \chi_t  &  + \rho \mathrm{div}(\uu_t) \big) \varphi \dd x \dd t
 + \int_0^T \int_\Omega \mathsf{K}(\teta) \nabla \log(\teta) \nabla \varphi  \dd x \dd t
\\ &
  -  \int_0^T \int_\Omega  \mathsf{K}(\teta) \frac{\varphi}{\teta}  \nabla \log(\teta)  \nabla \teta  \dd x \dd t
 \\
 &
= \int_0^T \int_\Omega  (g+a(\chi) \eps(\uu_t)\vism \eps(\uu_t) + |\chi_t|^2) \frac\varphi\teta  \dd x \dd t
+ \int_0^T \int_{\partial\Omega} h \frac\varphi\teta  \dd S \dd t
\end{aligned}
\end{equation}
for all $\varphi \in \mathcal{D}(\overline Q)$. Then, the entropy
inequality \eqref{entropy-ineq} later on follows.  The total energy
inequality   \eqref{total-enid}  associated with system
\eqref{eq0}--\eqref{eqII} is obtained by testing \eqref{eq0} by $1$,
\eqref{eqI} by $\uu_t$, and \eqref{eqII} by $\chi_t$. 

  Let us  mention in advance  that
the entropy inequality
\eqref{entropy-ineq} below has the advantage that all the troublesome quadratic quantities on the right-hand side of
\eqref{eq0} are tested by the \emph{negative} function $-\varphi$.
 This will  allow for upper semicontinuity arguments in the limit
passage for proving the existence of weak solutions, cf.\ Sec.\
\ref{s:pass-limit} later on. Let us also mention in advance that,
when dropping the unidirectionality constraint (i.e., in the case $\mu=0$), under an additional
condition (cf.\ Hypothesis (V)), we will be able to get an existence
result for
 an improved formulation of \eqref{eq0}, cf.\
 Theorem \ref{teor1} below.
\paragraph{\bf Weak formulation of the flow rule for $\chi$.}
A significant  difficulty in the analysis of system \eqref{eq0}--\eqref{eqII}
is due to the triply nonlinear character of \eqref{eqII}, featuring, in addition to the $p$-Laplacian and to
$\beta=\partial\widehat\beta$ which contributes to $W'$, the
 (maximal monotone) operator
 $\partial I_{(-\infty, 0]}$. Since the latter is unbounded, it is not possible to  perform
  comparison estimates in \eqref{eqII} and  an estimate for the terms $A_p \chi$ and $\beta(\chi)$
  (treated as single-valued in the context of this \beruno heuristic \eruno discussion) could be obtained only by testing
  \eqref{eqII} by $\partial_t (A_p\chi +\beta(\chi))$. However, the related calculations, involving an integration
  by parts in time on the right-hand side of \eqref{eqII},  cannot be carried out
  in the present case. That is why, we need to resort to a weak formulation of   \eqref{eqII}
  which does not feature the term $A_p\chi +\beta(\chi)$. We draw it from \cite{hk1,hk2}, and as therein
we confine the analysis  to the particular case in which \par
\noindent \textbf{Hypothesis (IV).}
\begin{equation}
\label{e:particular-widehatbeta}
 \widehat{\beta}=I_{[0,+\infty)}.
\end{equation}
This still ensures
 the constraint
 \begin{equation}
 \label{physical-admissib}
 \chi \in [0,1] \quad \text{ a.e.\ in $\Omega\times
(0,T)$}
\end{equation}
 provided we  start
from an initial datum $\chi_0 \leq 1$ a.e.\ in $\Omega$, we will
obtain by irreversibility that
 $\chi(t) \leq \chi_0 \leq 1$
  a.e.\ in $\Omega$, for almost all $t \in (0,T)$.

To motivate the weak formulation of \eqref{eqII}
from \cite{hk1,hk2}, we observe that \eqref{eqII}
rephrases as
\begin{subequations}
\label{ineq-system}
\begin{align}
 \label{ineq-system1}  & \chi_t \leq 0 \qquad
 \text{ in } \Omega \times (0,T),
 \\
 \label{ineq-system2}  &
 \left( \chi_t-\mathrm{div}(|\nabla \chi|^{p-2} \nabla \chi)+
 \xi + \gamma(\chi)  + b'(\chi)\frac{\varepsilon(\ub)
\elm\varepsilon(\ub)}{2} -\teta\right) \psi \geq 0  \text{ for all } \psi \leq 0 \qquad
 \text{ in } \Omega \times (0,T),
 \\
 \label{ineq-system3}  &
 \left( \chi_t-\mathrm{div}(|\nabla \chi|^{p-2} \nabla \chi)+
 \xi + \gamma(\chi)  + b'(\chi)\frac{\varepsilon(\ub)
\elm\varepsilon(\ub)}{2} -\teta\right) \chi_t \leq 0  \qquad
 \text{ in } \Omega \times (0,T),
\end{align}
\end{subequations}
with $\xi \in \partial I_{[0,+\infty)}(\chi)$ in $\Omega
\times (0,T)$.
Our weak formulation of \eqref{eqII} in fact consists of \eqref{ineq-system1},
of the integrated version of \eqref{ineq-system2}, with negative test functions from
$W^{1,p}(\Omega)$, and of the energy inequality obtained by integrating \eqref{ineq-system3}.
In \cite[Prop.\ 2.14]{rocca-rossi-deg} (see also \cite[Thm.\ 4.6]{hk1}), we prove that,
under additional regularity properties,
any weak solution in the sense of \eqref{constraint-chit}--\eqref{energ-ineq-chi} in fact fulfills
\eqref{eqII} pointwise.

We are now
in the position to specify our weak solution concept, for  which we borrow the terminology
from~\cite{fpr09}.
\begin{definition}[Entropic solutions to the irreversible system]
\label{prob-rev1}
{\sl
Let $\mu=1$. Given initial data $(\teta_0,\uu_0,\vv_0,\chi_0)$
fulfilling \eqref{datoteta}--\eqref{datochi},
we call a triple $(\teta,\uu,\chi)$ an \emph{entropic solution} to the (initial-boundary value problem)
for system \eqref{eq0}--\eqref{eqII}, with the
boundary conditions \eqref{intro-b.c.},
if
\begin{align}
\label{reg-teta}  & \teta \in  L^2(0,T; H^1(\Omega))\cap L^\infty(0,T;L^1(\Omega)),
\\
& \label{reg-log-teta}
\log(\teta) \in  L^2(0,T; H^1(\Omega)),
\\
& \label{reg-u} \uu \in    H^1(0,T;\boY) \cap W^{1,\infty}
(0,T;\boZ)  \cap H^2 (0,T;L^2(\Omega; \RR^d))\,,
\\
& \label{reg-chi} \chi \in L^\infty (0,T;W^{1,p} (\Omega)) \cap H^1
(0,T;L^2 (\Omega)),
\end{align}
$(\teta,\uu,\chi)$ complies with
 the initial conditions
\begin{align}
 \label{iniu}  & \uu(0,x) = \uu_0(x), \ \ \uu_t (0,x) = \vb_0(x) & \forae\, x \in
 \Omega,
 \\
 \label{inichi}  & \chi(0,x) = \chi_0(x) & \forae\, x \in
 \Omega,
\end{align}
 (while the initial condition for $\teta$ is implicitly formulated in \eqref{total-enid} below),
and with  the {\em entropic formulation} of \eqref{eq0}--\eqref{eqII}, consisting of
\begin{itemize}
\item[-]
the \emph{entropy}  inequality for almost all $t \in (0,T]$ and almost all
$s\in (0,t)$, and for $s=0$:
\begin{equation}
\label{entropy-ineq}
\begin{aligned}
& \int_s^t \int_\Omega (\log(\teta) + \chi) \varphi_t  \dd x \dd r  -
\rho \int_s^t \int_\Omega \dive(\uu_t) \varphi  \dd x \dd r
 - \int_s^t \int_\Omega  \condu(\teta) \nabla \log(\teta) \cdot \nabla \varphi  \dd x \dd r
\\ &
\leq
\beo\int_\Omega (\log(\teta(t))\beo+\chi(t)\eo){\varphi(t)} \dd x
 -  \int_\Omega (\log(\teta(s))\beo+\chi(s)\eo){\varphi(s)} \dd x
 \eo 
 -  \int_s^t \int_\Omega \condu(\teta) \frac{\varphi}{\teta}
\nabla \log(\teta) \cdot \nabla \teta  \dd x \dd r
\\
&
-
  \int_s^t  \int_\Omega \left( g +a(\chi) \eps(\uu_t)\vism \eps(\uu_t) + |\chi_t|^2 \right)
 \frac{\varphi}{\teta} \dd x \dd r
 -  \int_s^t \int_{\partial\Omega} h \frac\varphi\teta  \dd S \dd r
\end{aligned}
\end{equation}
for all $\varphi $ in $\mathrm{C}^0 ([0,T]; W^{1,d+\epsilon}(\Omega))$ for some $\epsilon>0$, and  $\varphi \in H^1 (0,T; L^{6/5}(\Omega))$,  with $\varphi \geq 0$;
\item[-] the \emph{total energy inequality}  for almost all $t \in (0,T]$ and almost all
$s\in (0,t)$, and for $s=0$:
\begin{equation}
\label{total-enid} \mathscr{E}(\teta(t),\uu(t), \uu_t(t), \chi(t))
\leq
 \mathscr{E}(\teta(s),\uu(s), \uu_t (s), \chi(s))  + \int_s^t
\int_\Omega g \dd x \dd r
+ \int_s^t
\int_{\partial\Omega} h \dd S \dd r
  + \int_s^t \int_\Omega \mathbf{f} \cdot
\mathbf{u}_t \dd x \dd r \,,
\end{equation}
where
 for $s=0$ we read $\teta_0$,
and
\begin{equation}
 \label{total-energy}
 \mathscr{E}(\teta,\uu,\uu_t,\chi):= \int_\Omega \teta \dd x +
  \frac12 \int_\Omega |\uu_t |^2 \dd x + \frac12\beo \bilh{b(\chi)}{\uu}{\uu}\eo
+ \frac1p \int_\Omega |\nabla \chi|^p \dd x + \int_\Omega W(\chi) \dd
x\,;
 \end{equation}
\item[-] the momentum equation
\begin{equation}
\label{weak-momentum}
\ub_{tt}+\opj{a(\chi)}{\ub_t}+\oph{b(\chi)}{\uu} +
\ciro( \teta) =\mathbf{f} \quad
\text{  a.e.\  in } \Omega \times (0,T);
\end{equation}
\item[-] the weak formulation of \eqref{eqII}, viz.\
\begin{align}
\label{constraint-chit}
&
\chi_t(x,t) \leq 0 \qquad \foraa\, (x,t) \in \Omega \times (0,T),
\\
 \label{ineq-chi}
 &
\begin{aligned}
  \int_\Omega  \Big( \chi_t (t) \psi     +
|\nabla\chi(t)|^{p-2} \nabla \chi(t) \cdot \nabla \psi  + \xi(t) \psi +
\gamma(\chi(t)) \psi 
 & + b'(\chi(t))\frac{\varepsilon(\ub(t))
\elm\varepsilon(\ub(t))}{2}\psi  -\teta(t) \psi \Big)
\,
\mathrm{d}x 
  \geq 0 \\ &  \text{for all }  \psi
\in 
W_-^{1,p}(\Omega), \quad \foraa\, t \in (0,T),
\end{aligned}
\\
&
\nonumber
\text{where $\xi \in
\partial I_{[0,+\infty)}(\chi)$ in the  sense that}
\\
&
\label{xi-def}  \xi \in L^1(0,T;L^1(\Omega)) \qquad\text{and}\qquad
\pairing{}{W^{1,p}(\Omega)}{\xi(t)}{\psi-\chi(t)} \leq 0 \  \
\forall\, \psi \in W_+^{1,p}(\Omega), \ \foraa\, t \in (0,T),
\\
&
 \nonumber \text{ as well as  the {\sl \berunodue energy-dissipation
inequality\erunodue} \berunodue (for $\chi$) \erunodue for all $t \in (0,T]$, for $s=0$,  and for almost all $0
< s\leq t$}
\\
&
\label{energ-ineq-chi}
\begin{aligned}
 \int_s^t   \int_{\Omega} |\chi_t|^2 \dd x \dd r  & +\io\left(
  \frac1p  |\nabla\chi(t)|^p +  W(\chi(t))\right)\dd x\\ & \leq\io\left(
  \frac1p |\nabla\chi(s)|^p+ W(\chi(s))\right)\dd x
  +\int_s^t  \int_\Omega \chi_t \left(- b'(\chi)
  \frac{\varepsilon(\ub)\elm\varepsilon(\ub)}2
+\teta\right)\dd x \dd r.
\end{aligned}
\end{align}
\end{itemize}
}
\end{definition}
\begin{remark}[Consistency of the entropic and the classical formulations of \eqref{eq0}]
\upshape \label{rmk:consistent} It  can be checked that, in case the
functions $\teta$ and $\chi $ are \emph{sufficiently smooth},
inequalities \eqref{entropy-ineq}--\eqref{total-enid}, combined with
\eqref{eqI} and \eqref{eqII}, yield the (pointwise formulation
of the)  heat equation \eqref{eq0}.

\berunodue To check this, by contradiction suppose that (the weak
formulation of) \eqref{eq0} does not hold. Since \eqref{eq0} is
equivalent to \eqref{entropy-ineq} with identity sign, we then would
have
 that
\eqref{entropy-ineq} holds
 with a  strict inequality sign.
Hence, we  could test \eqref{eqI} by $\uu_t$, \eqref{eqII} by
$\chi_t$, and choose $\varphi=\teta$ (which is admissible for a
sufficiently smooth $\teta$) in \eqref{entropy-ineq} (with a strict
inequality). Summing up the relations thus obtained, we would
conclude the total energy balance \eqref{total-enid}  is not
satisfied. \erunodue

However, at the moment the necessary regularity for $\teta$ and
$\chi $  to carry out this argument  is out of reach.
    \end{remark}
\berdue
\begin{remark}[Validity of the total energy inequality]
\label{rmk:total-eniq}
 Let us mention here that originally in \cite{fei, boufeima} the total energy inequality
 \eqref{total-enid} was required
  to be hold as an \emph{equality}
on \emph{every} sub-interval $(s,t) \subset [0,T]$.
   However, in the present setting, in general we will be
  able to  obtain it only as an \emph{inequality} on $(s,t)$
\emph{for almost all} $s, \, t \in (0,T)$.

Indeed, we will prove \eqref{total-enid} by passing to the limit in
its time-discrete version, involving an approximate total energy
functional evaluated at approximate solutions.
 Due to the lack of
suitable estimates on the  latter sequences, and to the presence of
nonlinear and nonsmooth terms in the energy (related to the  high
order and non-smooth nonlinearities
    in the $\chi$-equation \eqref{eqII}), we will be able
to prove the pointwise convergence of the approximate total energy
functional only almost everywhere on $(0,T)$.

Yet,  in the   irreversibile case $\mu=1$ we will slightly improve
\eqref{total-enid} under a further condition on $\mathsf{K}$ (see
Thm.\ \ref{teor3}). We will considerably enhance it
 in the reversible case $\mu=0$ and under suitable growth conditions on the heat conductivity $\condu$ (cf. Thm.~\ref{teor1}).
\end{remark}
\begin{remark}[Total energy inequality and energy-dissipation inequality for $\chi$]
\label{rmk:total-vs-endiss} As already pointed out, the total energy
inequality \eqref{total-enid} (formally) results from testing the
heat equation by $1$, the momentum equation by $\uu_t$,  the flow
rule for $\chi$ by $\chi_t$, and integrating in time. The latter
test also gives rise to the energy-dissipation inequality
\eqref{energ-ineq-chi}.

 However, let us stress that, in the present
setting, \eqref{total-enid} and \eqref{energ-ineq-chi} cannot be
obtained one from another. Indeed, to do so, it would be necessary
to test the entropy inequality by $\teta$ (which would correspond to
testing the heat equation by $1$), which is not an admissible test
function  for \eqref{entropy-ineq} due to its low regularity
\eqref{reg-teta}.
\end{remark}
\erdue
We now state our existence result for  system \eqref{eq0}--\eqref{eqII} in the case $\mu=1$.
As far as the time-regularity of $\teta$ goes, observe that we will just prove $\BV$-in-time regularity for $\log(\teta)$ (cf.\ \eqref{BV-log} below).
Indeed,
   we will obtain $\mathrm{BV}$-in-time
regularity for $\teta$, as well,
 under an additional restriction
on the exponent $\kappa$ in Hypothesis (I) (note that the range of the admissible values below depends on the space dimension), viz.\

\noindent
 \textbf{Hypothesis (V).} The exponent $\kappa$ in
\eqref{hyp-K} satisfies
\begin{equation}
\label{range-k-admissible}
\kappa \in (1, 5/3) \quad\hbox{if $d=3$ and } \kappa \in (1, 2) \quad\hbox{if $d=2$ }.
\end{equation}
\begin{maintheorem}[Existence of entropic solutions, $\mu=1$]
\label{teor3} Let $\mu=1$. Assume  \beo \textbf{Hypotheses (0)--(III)}  \eo and, in addition, \textbf{(IV)}
(i.e., $\widehat{\beta}=I_{[0,+\infty)}$), as well as
 conditions
\eqref{bulk-force}--\eqref{datochi}  on the data $\mathbf{f},$ $g$, $h$,
$\teta_0,$ $\uu_0,$ $\vv_0,$ $\chi_0$.
 Then, there exists an entropic solution (in the sense of Definition \ref{prob-rev1})
  $(\teta,\uu,\chi)$ to the initial-boundary value problem for system
\eqref{eq0}--\eqref{eqII},
such that
\begin{equation}
\label{BV-log}
 \beo \log(\teta) \in L^\infty(0,T;W^{1,d+\epsilon}(\Omega)^*) \qquad \text{for all } \epsilon >0, \eo
\end{equation}
and:
\begin{enumerate}
 \item
 $\xi$ in \eqref{xi-def} is given by
\begin{equation}
\label{xi-specific}
\xi(x,t) = -  \mathcal{I}_{\{\chi=0\}}  (x,t) \left(\gamma(\chi(x,t)) + b'(\chi(x,t)) \frac{\eps(\uu(x,t)) \elm(x)  \eps(\uu(x,t))  }{2}  - \teta(x,t)\right)^+,
\end{equation}
 for almost all $(x,t) \in \Omega \times (0,T)$,
where  $\mathcal{I}_{\{\chi=0\}}$  denotes the characteristic function of
the set  $\{(x,t) \in \Omega \times (0,T)\, : \  \chi(x,t)=0\}$,
\item
$\exists\,\underline{\teta}>0$  such that
\begin{equation}\label{strictpos}
\teta(x,t)\geq  \underline{\teta}>0   \quad \forae\, (x,t) \in \Omega \times
(0,T).
\end{equation}
\end{enumerate}

Furthermore, if   in addition $\condu$ satisfies \textbf{Hypothesis (V)},
there holds
\begin{equation}
\label{furth-reg-teta} \teta\in \BV([0,T];
W^{2,d+\epsilon}(\Omega)^*) \qquad \text{for every } \epsilon>0,
\end{equation}
and the total energy inequality \eqref{total-enid} holds \berdue \underline{for all} \erdue $t \in [0,T]$, for $s=0$, and for almost all $s \in (0,t)$.
 \end{maintheorem}
Observe that  \eqref{furth-reg-teta}
yields that there exists $D \subset [0,T]$, at most infinitely countable, such that $\teta \in \mathrm{C}^0 ([0,T]\setminus D; W^{2,d+\epsilon}(\Omega)^*)$.
We will develop the proof
in Section \ref{s:pass-limit},
 by passing to the limit in the time-discretization scheme  carefully devised in Section \ref{s:time-discrete}.
\begin{remark}[Uniqueness and extensions]\label{rmk:uni-irr}
\begin{enumerate}
\item
Uniqueness of solutions for the \emph{irreversible} system,
even in the isothermal case, is still an open problem. This is
mainly due to the doubly nonlinear character of \eqref{eqII} (cf.\
also \cite{CV} for non-uniqueness examples for a general doubly
nonlinear equation).
\item
Theorem \ref{teor3} could be easily extended to the case in which the indicator function $I_{(-\infty,0]}$ in \eqref{eqII}
is replaced by
\begin{equation}
\label{1-homog}
\begin{gathered}
 \widehat{\alpha}: \R \to [0,+\infty] \quad \text{convex, $1$-positively homogeneous, with } \mathrm{dom}(\widehat\alpha) \subset (-\infty,0] \text{ and } 0 \in \alpha(0).
\end{gathered}
\end{equation}
\end{enumerate}
\end{remark}
\subsection{A global existence  result for the  reversible system}
\label{ss:glob-rev}
In the case $\mu=0$, we are able to cope with a weak solvability notion for
system \eqref{eq0}--\eqref{eqII}  stronger than the one from Definition \ref{prob-rev1}. Indeed, it
features a  \emph{pointwise} formulation for the internal parameter equation \eqref{eqII}, while keeping the entropic formulation
for the heat equation \eqref{eq0}.
Under the additional Hypothesis (V),
we will also improve the weak formulation
of the heat equation
(cf.\ \eqref{eq-teta} below). As a byproduct,
we will manage to prove the total energy \emph{identity} \text{for all} $t \in [0,T]$.
\begin{definition}[Entropic solutions to the reversible system]
\label{prob-rev2}
{\sl
Let $\mu=0$. Given initial data $(\teta_0,\uu_0,\vv_0,\chi_0)$
fulfilling \eqref{datoteta}--\eqref{datochi},
we call a triple $(\teta,\uu,\chi)$ an \emph{entropic solution} to the (initial-boundary value problem)
for system \eqref{eq0}--\eqref{eqII}, with the
boundary conditions \eqref{intro-b.c.},  if
it has the regularity
\eqref{reg-teta}--\eqref{reg-chi}, it complies
 with
  the initial conditions \eqref{iniu}--\eqref{inichi},
    and
with
\begin{itemize}
\item[-]
the \emph{entropy}  inequality
\eqref{entropy-ineq};
\item[-] the \emph{total energy inequality}
\eqref{total-enid}
 {for  almost all  $t \in (0,T]$, for $s=0$, and for almost all $s \in (0,t)$};
\item[-] the momentum equation \eqref{weak-momentum};
\item[-] the internal parameter equation
\begin{equation}
\label{weak-phase}
\chi_t + A_p\chi + \xi + \gamma(\chi)= -b'(\chi) \frac{\eps(\uu)\elm \eps(\uu)}2 +\teta \qquad \aein\, \Omega \times (0,T),
\end{equation}
with
\begin{equation}
\label{xi-def-reversible}
 \xi \in L^2(0,T;L^2(\Omega))
 \text{ s.t. } \quad
 \xi(x,t) \in
\beta(\chi(x,t))  \ \foraa\, (x,t)\in \Omega \times (0,T).
 \end{equation}
 \end{itemize}
}
\end{definition}

Our second
 main result states the existence of  an entropic solution
$(\teta,\uu,\chi)$ (in the sense of
Definition \ref{prob-rev2}) to the PDE system \eqref{eq0}--\eqref{eqII}. Furthermore,
we show that, under the additional Hypothesis (V), the formulation of the heat equation \eqref{eq0} improves to a standard
variational formulation (cf.\ \eqref{eq-teta} below), albeit with suitably smooth test functions,  and the total energy inequality \eqref{total-enid} holds as an equality.
 We shall refer to the solutions thus obtained as \emph{weak}.
\begin{maintheorem}[Existence of entropic and weak solutions, $\mu=0$]
\label{teor1} Let  $\mu=0$. Assume  \beo  \textbf{Hypotheses (0)--(III)}, \eo  and
condi\-tions \eqref{bulk-force}--\eqref{datochi} on the data
$\mathbf{f}, \, g, \, h,
\, \teta_0,\, \uu_0, \, \vv_0, \, \chi_0$. Then,
 there exists an entropic solution (in the sense of Definition \ref{prob-rev2})
 $(\teta,\uu,\chi)$ to the initial-boundary value problem for system
\eqref{eq0}--\eqref{eqII}, such that the strict positivity property
\eqref{strictpos} holds for $\teta$,
 and such that $\chi$ has the enhanced regularity
\begin{equation}
\label{enhanced-chi}
\chi \in L^2(0,T; W^{1+\sigma,p}(\Omega))
\qquad \text{for all } 0 <\sigma <\frac1p.
\end{equation}

Moreover, if   $\condu$ also complies with  \textbf{Hypothesis (V)},
then $\teta$ has the enhanced regularity \beo
\begin{equation}
\label{better-4-w}
\teta\in W^{1,1}(0,T;
W^{2,d+\epsilon}(\Omega)^*) \qquad \text{for every } \epsilon>0
\end{equation}
 \eo (cf.\ \eqref{furth-reg-teta}), and
 the heat equation   \eqref{eq0}  is fulfilled in the following  \beo\emph{improved} form
 for almost all $t\in (0,T)$
\begin{align}
&
\begin{aligned}
\label{eq-teta}   &
\pairing{}{W^{2,d+\epsilon}(\Omega)}{\partial_t\teta}{\varphi}
  + \io \chi_t
\teta\varphi \dd x  + \rho \io
\hbox{\rm div}(\ub_t) \teta\varphi \dd x
+ \io \condu(\teta) \nabla \teta\nabla\varphi \dd
x
\\
& = \io \left(g+\frac{\eps(\uu_t) \vism \eps(\uu_t)}2
+|\chi_t|^2\right) \varphi  \dd x  + \int_{\partial\Omega} h \varphi   \dd S
\quad \text{for all }
\varphi \in   W^{2,d+\epsilon}(\Omega)) \text{ for some $\epsilon>0$}. 
\end{aligned}
\end{align}
\eo
In this case, the triple $(\teta,\uu,\chi)$ complies with the  \emph{total energy equality}
\begin{equation}
\label{total-enid-asequality}
\mathscr{E}(\teta(t),\uu(t), \uu_t(t), \chi(t))
=
 \mathscr{E}(\teta(s),\uu(s), \uu_t (s), \chi(s))  + \int_s^t
\int_\Omega g \dd x \dd r
+ \int_0^t
\int_{\partial\Omega} h \dd S \dd r
  + \int_s^t \int_\Omega \mathbf{f} \cdot
\mathbf{u}_t \dd x \dd r \,,
\end{equation}
\underline{for all} $0 \leq s\leq t \leq T$.
\end{maintheorem}

The  \emph{proof} will be given in Section \ref{s:pass-limit},   passing to the
limit in the time-discretization scheme set up in Sec.\
\ref{s:time-discrete}.
We mention in advance that the argument for \eqref{eq-teta} and for the total energy identity \eqref{total-enid-asequality}
\emph{for all} $t \in [0,T]$ relies on obtaining, for the sequence $(\uu_k,\chi_k)$ of
approximate solutions, the \emph{strong} convergences
\begin{equation}
\label{strong-convergences}
\uu_k \to \uu  \quad \text{ in $H^1 (0,T; \boZ)$,} \qquad
\chi_k \to \chi \quad \text{ in $H^1 (0,T; L^2(\Omega))$.}
\end{equation}
 This allows us to pass to the limit on the right-hand side of the approximate version of
\eqref{eq-teta}. In turn,   the proof of \eqref{strong-convergences}  is based on a $\limsup$-argument,
 for which it is
essential to have preliminarily obtained the \emph{pointwise} formulation \eqref{weak-phase} of the equation for $\chi$.
This is the reason why we have not been able to obtain the improved formulation \eqref{eq-teta}
in the irreversible case $\mu=1$.
\begin{remark}[Uniqueness in the reversible case]\label{rmk:uni-rev}
 As in the irreversible case, a uniqueness result for the full system seems to be out of reach.
Instead,   for the \emph{isothermal} case in \cite[Thm.
3]{rocca-rossi-deg} we have proved
uniqueness  and continuous
dependence of the solutions on the data. This result has been obtained in the case \berdue that \erdue the $p$-Laplacian operator $-\mathrm{div} (|\nabla \chi|^{p-2}\nabla \chi)$ is replaced by an elliptic operator of the type described in Remark \ref{rmk:general-p-Lapl},
fulfilling an additional non-degeneracy condition, cf.\ Hypothesis (VII) in \cite{rocca-rossi-deg}: for instance,
we may consider $-\mathrm{div} ((1+|\nabla \chi|^{2})^{(p-2)/2}\berdue\nabla\chi\erdue) $.
\end{remark}
\begin{remark}[Alternative boundary conditions for the displacement]
\label{rmk:discussion-bc}
\upshape
\beruno Our existence results Theorems \ref{teor3}
and \ref{teor1} carry over to the case of a time-dependent Dirichlet loading $g$
(in place of the homogeneous Dirichlet condition in \eqref{intro-b.c.})
 for the displacement $\uu$, under suitable conditions on $g$. The latter have
 to ensure the validity of the elliptic regularity estimate on $\uu$
 (cf.\ the forthcoming \emph{Fifth estimate} in Sec.\ \ref{s:aprio}), which leads to the regularity
 \eqref{reg-u} and  plays a crucial role for our analysis.

 Moreover, the proofs of Thms.\   \ref{teor3}  and  \ref{teor1}    could be carried out with
suitable modifications    in the case of Neumann boundary conditions   for $\uu$
on the whole of $\partial\Omega$,  as well. We would also be able to handle the case of
Neumann conditions on a portion $\Gamma_0$ of $\partial \Omega$  and
Dirichlet conditions on $\Gamma_1:=\partial \Omega \setminus
\Gamma_0$ ($|\Gamma_0|, \, |\Gamma_1|>0$), provided that the
closures of the sets $\Gamma_0$ and $\Gamma_1$ do not intersect.
 Indeed, without the latter  geometric condition,
  the elliptic regularity results
  at the core of the Fifth estimate and thus of  \eqref{reg-u}
   may fail to hold, see~\cite[Chap.~VI,~Sec.~6.3]{ciarlet}.

   Nonetheless,  in Sec.\ \ref{s:final}, where we will address the analysis of system
   \eqref{eq0}--\eqref{eqII}, with unidirectional evolution ($\mu=1$), in the case the $p$-Laplacian regularization in  \eqref{eqII} is replaced by the
   Laplace operator,
   more general boundary conditions on $\uu$ could be considered. Indeed, therein we will not be in the position
  to perform any
elliptic regularity estimate on $\uu$ (and therefore we will conclude the existence of a \emph{weaker} notion of solution). Hence,
mixed Dirichlet-Neumann conditions on $\uu$  could be taken into account in that setting  (cf. also Remark~\ref{omegaLip}).
\eruno
\end{remark}
\section{\bf (Formal) A priori estimates}
\label{s:aprio}
\noindent
In this section, we perform a series of  \emph{formal} estimates on
system \eqref{eq0}--\eqref{eqII}. All of these estimates
will  be rigorously justified on the
time-discrete approximation scheme proposed in
Section~\ref{s:time-discrete},
\beo with the exception of the \emph{Sixth estimate}, to be rendered in a \emph{weaker} version, cf.\ the comments
prior to the statement of Proposition \ref{prop:discrete-aprio}, and Remark \ref{rmk:after6}. \eo

 Yet, we believe that, in order to enhance the readability of the paper,
 it is worthwhile to
develop all the significant calculations on the (easier) time-continuous level. This is  especially useful
 for the Second and the Third a priori estimates,
which have a non-standard character and are in fact tailored to handle the quadratic terms on the right-hand side of~\eqref{eq0}.

More in detail, we start by showing the strict positivity of the temperature $\teta$,  via a comparison argument in the same lines
as the one for proving positivity in \cite[Subsection~4.2.1]{fpr09}.  All the ensuing estimates rely on this property, starting from the basic energy estimate (i.e.\ the one corresponding to the total energy inequality \eqref{total-enid}).  After this, we test  \eqref{eq0} by $\teta^{\alpha-1}$, with $\alpha \in (0,1)$. This
enables us 
 somehow to confine the troublesome quadratic terms to the left-hand side.
Carefully using the Gagliardo-Nirenberg inequality, we infer a bound for $\teta^\alpha$ in $ L^2(0,T;H^1(\Omega))$.
Ultimately, exploiting the fact that the heat flux $\condu$
controls $\teta^\kappa$ (cf.\  \eqref{hyp-K}) we conclude an estimate for $\teta$ in $ L^2(0,T;H^1(\Omega))$.
This \beruno being \eruno done,  we are in the position to perform all the remaining estimates, i.e.\ subtracting the temperature equation tested by $1$ from the total energy inequality \eqref{total-enid};
performing an elliptic regularity estimate on the momentum equation \eqref{eqII}, and comparison estimates in \eqref{eq0} and \eqref{eqII}.

We mention in advance that, with the exception of the last one, all
of the ensuing estimates hold both in the reversible ($\mu=0$), and
in the irreversible ($\mu=1$) cases. \berdue We warn the reader
that   in what follows we will use the same symbol  $C$  for several
different constants, even varying  from line to line and depending
only on the data of the problem, on $\Omega$ and on $T$. \erdue
\paragraph{\bf Positivity of $\teta$ [$\mu\in \{0,1\}$].}
\beruno Moving \eruno all the quadratic terms in \eqref{eq0} to the  right-hand   side, we obtain
\[
\begin{aligned}
\teta_t-\dive(\condu (\teta)\nabla\teta)
 & = g + a(\chi)\eps(\uu_t)\vism \eps(\uu_t) +|\chi_t|^2
 -\chi_t\teta - \rho \teta \mathrm{div}(\uu_t)
\\ & \geq
g+ c |\eps(\uu_t)|^2 +  \berunodue \frac12 |\chi_t|^2 \erunodue -C
\teta^2
\geq -C\teta^2 \quad \aein \, \Omega \times (0,T),
\end{aligned}
\]
where we have written \eqref{eq0} in a formal way, disregarding the (positive) boundary datum $h$.
Indeed, for the first inequality we have used
that $\vism$ is positive definite, that $a$ is strictly positive, and the fact that
\begin{equation}
\label{eps-estim}
| \dive(\uu_t)  | \leq c(d)
|\tensoret|  \quad \text{a.e.\ in $\Omega \times (0,T)$}
\end{equation}
 with $c(d)$ a positive
constant only depending on the space dimension $d$. The second estimate also relies on the fact that $g \geq 0$ a.e.\ in $\Omega \times (0,T)$.
Therefore we conclude that
  $v$  solving  the Cauchy problem
\[
v_t=-\frac12 v^2, \quad v(0)=\teta_*>0
\]
is a subsolution of \eqref{eq0}. Hence, a comparison argument yields
\begin{equation}\label{teta-pos}
\teta(\cdot,t)\geq v(t)>\teta_*>0\quad \hbox{for all }t\in [0,T]\,.
\end{equation}
\paragraph{\bf First estimate [$\mu\in \{0,1\}$].}
Test \eqref{eq0} by 1, \eqref{eqI} by $\uu_t$, \eqref{eqII} by $\chi_t$ and integrate over $(0,t)$, $t\in (0,T]$.
Adding the resulting equations and taking into account cancellations, we obtain
\begin{equation}\label{calc1}
\begin{aligned}
 & \int_\Omega \teta(t) \dd x + \frac12 \int_\Omega |\uu_t (t)|^2 \dd x
+ \frac12 \bilh{b(\chi(t))}{\uu(t)}{\uu(t)} +\frac1p
\int_\Omega |\nabla \chi(t)|^p \dd x + \int_\Omega W(\chi(t)) \dd x
\\&
 =
\int_\Omega \teta_0 \dd x + \frac12 \int_\Omega |\vv_0|^2 \dd x +
\frac12 \bilh{b(\chi_0)}{\uu_0}{\uu_0}
+\frac1p \int_\Omega |\nabla \chi_0|^p \dd x + \int_\Omega W(\chi_0) \dd x \\
&\quad + \int_0^t \int_\Omega g \dd x \dd s + \int_0^t \int_{\partial\Omega} h \dd S \dd s  +  \int_0^t \int_\Omega
\mathbf{f} \cdot \mathbf{u}_t \dd x \dd s \,,
\end{aligned}
\end{equation}
viz.  the total energy  equality \eqref{total-enid-asequality}. For \eqref{calc1},
 we have also used the integration-by-parts formula
\begin{align}\label{int-parts}
&\int_0^t \bilh{b(\chi(t))}{\uu(s)}{\uu_t(s)}  \dd s+ \frac12 \int_0^t \int_\Omega  b'(\chi)\chi_t \tensore \elm\tensore \dd x \dd s =  \frac12 \bilh{b(\chi(t))}{\uu(t)}{\uu(t)} -
\frac12\bilh{b(\chi_0)}{\uu_0}{\uu_0}
\end{align}
as well as the
fact that
 $\int_0^t \int_\Omega  \partial
I_{(-\infty,0]}(\chi_t) \chi_t \dd x \dd s =  \int_0^t \int_\Omega I_{(-\infty,0]}(\chi_t) \dd x \dd s=  0 $
(where we have formally written $ \partial
I_{(-\infty,0]}(\chi_t) $ as a single-valued operator). Using  \eqref{bulk-force}--\eqref{datochi} for the  data
$f$, $g$, $h$ and the initial data
$(\teta_0,\uu_0,\chi_0)$,
the positivity of $\teta$
(cf. \eqref{teta-pos}), and
the 
\beo fact that  $W$ is bounded from below (cf.\ \eqref{data-W}  from
Hypothesis (III)), \eo also in view of the Poincar\'e inequality
  we conclude the following estimate
\begin{equation}\label{est1}
\| \teta \|_{L^\infty (0,T;L^1(\Omega))} + \| \uu\|_{W^{1,\infty}
(0,T;L^2(\Omega;\R^d))} + \|b(\chi)^{1/2} \tensore \|_{L^\infty
(0,T;L^2(\Omega;\R^{d\times d}))} + \| \nabla\chi \|_{L^p
(0,T;L^{p}(\Omega))}\leq C\,,
\end{equation}
as well as
\begin{equation}\label{est1-ter}
\|W(\chi)\|_{L^\infty(0,T; L^1(\Omega))}\leq C\,.
\end{equation}

\paragraph{\bf Second estimate [$\mu\in \{0,1\}$].}
Let $F(\teta) = \teta^\alpha/\alpha$, with $\alpha \in (0,1)$.
 We test \eqref{eq0} by $F'(\teta):= \teta^{\alpha-1}$ ,
and integrate on $(0,t)$ with $t \in (0,T]$.
 We thus have
\[
\begin{aligned}
  &
   \int_\Omega F(\teta_0)\dd x+
\int_0^t \int_\Omega  g F'(\teta) \dd x \dd s
+ \int_0^t \int_{\partial \Omega} h F'(\teta) \dd S \dd s
+ \int_0^t \int_\Omega
a(\chi) \tensoret \vism \tensoret F'(\teta) \dd x \dd s\\
&+
\int_0^t \int_\Omega |\chi_t|^2 F'(\teta) \dd x \dd s
=
   \int_\Omega F(\teta(t))\dd x + \int_0^t \int_\Omega \chi_t \teta F'(\teta) \dd x \dd s +\rho \int_0^t \int_\Omega \teta \dive(\uu_t) F'(\teta)  \dd x \dd  s \\
&+
  \int_0^t \int_\Omega \condu(\teta) \nabla \teta \nabla (F'(\teta)) \dd x \dd s  \end{aligned}
\]
whence (cf.\ \eqref{korn}, \berdue \eqref{ellipticity}, \erdue and the positivity \eqref{heat-source} and \eqref{dato-h} of $g$ and $h$)
\[
\begin{aligned}
 & \frac{4(1-\alpha)}{\alpha^2} \int_0^t \int_\Omega
\condu(\teta) |\nabla (\teta^{\alpha/2})|^2 \dd x \dd s + c_2 \berdue \beta_0\erdue\int_0^t
\int_\Omega|\tensoret|^2
F'(\teta) \dd x \dd s+ \int_0^t \int_\Omega |\chi_t|^2 F'(\teta) \dd
x \dd s \\ &  \leq   \int_\Omega |F(\teta_0)|\dd x +I_1
+I_2+I_3,
\end{aligned}
\]
where we have used \eqref{korn} \berdue (with $\beta_0$ from
\eqref{ellipticity}), \erdue and \eqref{data-a}.
 We estimate
\[
\begin{aligned}
I_1= \int_\Omega |F(\teta(t))|\dd x \leq \frac1{\alpha}\int_\Omega \max\{
\teta(t), 1\}^\alpha \dd x \leq  \frac1{\alpha}\int_\Omega \max\{  \teta(t), 1\}
\dd x \leq C
\end{aligned}
\]
since $\alpha <1$ and taking into account the previously obtained
\eqref{est1}. Analogously we can estimate $\int_\Omega
|F(\teta_0)|\dd x$; moreover,
\[
I_2 = \int_0^t \int_\Omega |\chi_t \teta F'(\teta)| \dd x \dd s
\leq \frac14 \int_0^t \int_\Omega |\chi_t|^2 F'(\teta) \dd x \dd s +
 \int_0^t \int_\Omega F'(\teta)\teta^2 \dd x \dd s.
\]
\berdue Using \eqref{data-a},  inequality \eqref{eps-estim}, and
Young's inequality, we have that
\[\begin{aligned}
I_3 =|\rho|  \int_0^t \int_\Omega | \teta \dive(\uu_t) F'(\teta)|  \dd
x \dd  s
   \leq   \frac {\berdue \beta_0\erdue c_2} 4 \int_0^t \int_\Omega |\tensoret|^2 F'(\teta) \dd x \dd s +
C  \int_0^t \int_\Omega F'(\teta)\teta^2 \dd x \dd s\,.
\end{aligned}
\]

All in all, we conclude
\begin{equation}
\label{calc2.1}\begin{aligned} &
 \frac{4(1-\alpha)}{\alpha^2} \int_0^t  \int_\Omega  \condu(\teta)  |\nabla (\teta^{\alpha/2})|^2 \dd x \dd s
+ \frac{\berdue 3\beta_0\erdue c_2}{4}\int_0^t \int_\Omega
|\tensoret|^2 F'(\teta) \dd x \dd s+ \berdue\frac34\erdue \int_0^t \int_\Omega
|\chi_t|^2 F'(\teta) \dd x \dd s\\ &  \leq C + C  \int_0^t
\int_\Omega \teta^{\alpha+1} \dd x \dd s.
\end{aligned}
\end{equation}
Now, we  fix $q \geq 4$ and introduce the auxiliary quantity $\eta:=
\max\{  \teta, 1 \}$. Observe that $\eta $ is still in
$H^1(\Omega)$, and that, for $q$ sufficiently big  (see below) we have
\begin{equation}
\label{stima-aux1} \frac\alpha 2 \geq \frac{\alpha +1}{q} \text{
whence } \eta^{(\alpha+1)/q} \leq   \eta^{\alpha/2} \doteq w.
\end{equation}
\erdue
Therefore, taking into account that
\[
  \int_0^t \int_\Omega \condu(\teta) |\nabla (\teta^{\alpha/2})|^2 \dd x \dd s  \geq c_1 \iint_{\{  \teta \geq 1\}}
  |\nabla (\teta^{\alpha/2})|^2 \dd x \dd s = c_1 \int_0^t \int_\Omega |\nabla w |^2 \dd x \dd s,
\]
thanks to \eqref{hyp-K},
we infer from \eqref{calc2.1} and \eqref{stima-aux1}   that
\begin{equation}\label{calc2.2}\begin{aligned}
 &  \int_0^t \int_\Omega |\nabla w|^2 \dd x \dd s
 \leq C + C  \int_0^t \| w \|_{L^q(\Omega)}^q\dd s. \end{aligned}
\end{equation}
We now apply the Gagliardo-Nirenberg inequality  for $d=3$ (for $d=2$
even better estimates hold true), yielding
\begin{equation}\label{gagliardo}
\| w \|_{L^q(\Omega)} \leq c_1  \| \nabla w\|_{L^2(\Omega;\R^d)}^\theta
\| w \|_{L^r(\Omega)}^{1-\theta} + c_2   \| w \|_{L^r(\Omega)}
\end{equation}
with $ 1 \leq  r  \leq  q $ and $\theta $ satisfying $1/q= \theta/6
+ (1-\theta)/r$. Hence $\theta= 6 (q-r)/q (6-r)$.
 Observe that  $\theta \in (0,1)$ if $q<6$ and \berdue that \erdue this restriction on $q$
 implies that, for \eqref{stima-aux1} we need to have $\alpha \in  (1/2,
 1)$.
 Plugging the Gagliardo-Nirenberg estimate  into \eqref{calc2.2}, using Young's inequality
  \berdue with exponents $2/\theta q$ and $2/(2-\theta q)$
and with suitable weights in such a way as to absorb the term $
\| \nabla w\|_{L^2(\Omega;\R^d)}^2$ into the  left-hand side of
\eqref{calc2.2},
   we
 ultimately conclude\erdue
\begin{equation}
\label{calc2.3} \frac c2  \int_0^t \int_\Omega |\nabla w|^2 \dd x
\dd s \leq   C + C  \int_0^t \| w  \|_{L^r(\Omega)}^{2q
(1-\theta)/(2-q\theta)}\dd s +  C'  \int_0^t    \| w
\|_{L^r(\Omega)}^q  \dd s \,.
\end{equation}
Now, choosing $r \leq 2/\alpha$, we have that
\[
  \| w \|_{L^r(\Omega)} = \left(\int_\Omega  \eta^{r\alpha/2}  \dd x\right)^{1/r} \leq \left(\int_\Omega\eta \dd x\right)^{1/r} \leq  C  \| \teta\|_{L^\infty (0,T;L^1(\Omega))} + |\Omega|\leq C\,,
\]
 where the latter inequality is  due to estimate \eqref{est1}.  Combining
the above estimate with \eqref{calc2.3} we infer a bound for
$w=\eta^{\alpha/2}$ in $L^2 (0,T;H^1(\Omega)) \cap L^\infty
(0,T;L^r(\Omega))$. Ultimately, also in view of \eqref{calc2.2}, we
conclude  that
\begin{equation}
\label{est2} \| \teta^{\alpha/2} \|_{L^2 (0,T;H^1(\Omega)) \cap
L^\infty (0,T;L^r(\Omega))} \leq C.
\end{equation}

\paragraph{\bf Third estimate  [$\mu\in \{0,1\}$].}
It follows from \eqref{calc2.1} and \eqref{hyp-K}  that
\begin{equation}
\label{additional-info}
\begin{aligned}
C \geq   \int_0^t \int_\Omega \condu(\teta) |\nabla (\teta^{\alpha/2})|^2 \dd x \dd s
 & \geq
 c_1 \int_0^t \int_\Omega \teta^\kappa  |\nabla (\teta^{\alpha/2})|^2 \dd x \dd s
 \\
 & =
 \int_0^t \int_\Omega |\teta^{\kappa+\alpha - 2}| |\nabla  \teta|^2   \dd x \dd s
 \\
  & = \int_0^t \int_\Omega  |\nabla   (\teta^{(\kappa+\alpha)/2})|^2   \dd x \dd s
 \end{aligned}
  \end{equation}
  with $\alpha \in [1/2, 1)$ arbitrary.

 From \eqref{additional-info}
and the strict positivity of $\teta$ \eqref{teta-pos}
it follows that
\[
\int_0^t \int_\Omega |\nabla \teta|^2 \dd x \dd s \leq C,
\]
provided that $\kappa +\alpha-2 \geq 0$. Observe that, since $\kappa>1$ we can choose
$\alpha \in [1/2, 1)$  such that this inequality holds.
  Hence,  taking into account  estimate \eqref{est1} and
 applying Poincar\'e inequality,
  we deduce
 \begin{equation}
\label{crucial-est3.2} \| \teta   \|_{L^{2}  (0,T; H^1(\Omega))}
\leq C .
\end{equation}

\berdue We now interpolate between estimate \eqref{crucial-est3.2} and
estimate \eqref{est1} for $\|\teta\|_{L^\infty (0,T; L^1(\Omega))}$,
using the Gagliardo-Nirenberg inequality \eqref{gn-ineq} that gives
$\|\teta\|_{L^h(\Omega)} \leq \|\teta\|_{H^1(\Omega)}^{\theta}
\|\teta\|_{L^1(\Omega)}^{1-\theta} $ with $\theta \in (0,1)$ and
$h\in [1,\infty]$ related by $\tfrac1h = \theta(\tfrac12{-}\tfrac1d)
+1-\theta$. Hence, we get \erdue
\begin{equation}\label{estetainterp}
\|\teta\|_{L^h(\Omega\times (0,T))}\leq C\quad\hbox{with }h=8/3
\quad \hbox{if } d=3, \quad h=3 \quad \hbox{if } d=2\,.
\end{equation}
For later use, we also point out that \berdue
\begin{equation}
\label{add-ref} \| \nabla \teta^{(\kappa -\alpha)/2}\|_{L^2 (0,T;
L^2(\Omega))} \leq C. \end{equation} Indeed, it suffices to observe
that
\[
\int_\Omega |\nabla \teta^{(\kappa -\alpha)/2}|^2 \dd x =
\int_\Omega \teta^{\kappa-\alpha-2} |\nabla \teta|^2 \dd x \leq
\frac1{\teta_*^{2\alpha}}\int_{\Omega} \teta^{\kappa+\alpha-2}
|\nabla \teta|^2 \dd x \leq C,
\]
where the first inequality derives from the positivity property
\eqref{teta-pos}, and the last one from
 estimate
\eqref{additional-info}. Combining \eqref{additional-info} and
\eqref{add-ref}
  with estimate \eqref{est1}, \erdue
and using a nonlinear version of the Poincar\'e inequality (cf.\
e.g.\ \eqref{poincare-type}), we infer
\begin{equation}
\label{necessary-added}
\| \teta^{(\kappa -\alpha)/2} \|_{L^2 (0,T; H^1(\Omega))}, \, \| \teta^{(\kappa +\alpha)/2} \|_{L^2 (0,T; H^1(\Omega))}  \leq C.
\end{equation}

\paragraph{\bf  Fourth  estimate [$\mu\in \{0,1\}$].} We test \eqref{eq0} by $1$, integrate over $(0,t)$, and subtract the resulting identity from the total energy balance
\eqref{calc1}. We thus obtain
\begin{align}\no
& \frac 12\io |\uu_t(t)|^2\dd x
+\int_0^t \bilj{a(\chi)}{\uu_t}{\uu_t}\dd s+\frac12 \bilh{b(\chi(t))}{\uu(t)}{\uu(t)}
+\itt\io |\chi_t|^2\dd x\dd s+\io \frac1p |\nabla\chi(t)|^p + W(\chi(t))\dd x
\no
\\
& =
\frac12 \io|\uu_0|^2\dd x
+ \frac12\bilh{b(\chi_0)}{\uu_0}{\uu_0} +\io \frac1p|\nabla\chi_0|^p+\io  W(\chi_0)\dd x+ \itt\io\teta \left(\rho \dive \uu_t+\chi_t\right)\dd x \dd
s\\
\no
&\quad+\itt \io {\bf f}\, \uu_t\dd x \dd s.
\end{align}
 Using now \eqref{datou}--\eqref{datochi} to estimate the initial data $(\uu_0,\chi_0)$,
\eqref{bulk-force}  on $\mathbf{f}$,   Hyp.\ (III) (which in particular yields that $W$ is bounded from below), and combining  estimate
\eqref{crucial-est3.2}
 on $\teta$ with \eqref{eps-estim},
 we obtain
\begin{equation}\label{est5}
\|\chi_t\|_{L^2(\Omega\times(0,T))}+
\|a(\chi)^{1/2}\tensoret\|_{L^2(\Omega\times(0,T); \RR^{d\times d})} \leq
C\,,
\end{equation}
 whence $\|\uu_t\|_{L^2(0,T; H^1_0(\Omega;\RR^d))}\leq C$, by
 \eqref{data-a}. \beo Furthermore, in view of \eqref{est1} we also
 gather
 \begin{equation}
 \label{est5-added}
\| \chi \|_{L^\infty (0,T;W^{1,p}(\Omega))} \leq C.
 \end{equation}
 \eo

\paragraph{\bf  Fifth  estimate  [$\mu\in\{0,1\}$].}
We use here the crucial assumption that $p>d$.  \berdue We test \eqref{eqI} by $-\mathrm{div}(\vism\eps(\uu_t))$ and integrate on time (cf. also \cite[Sec.~3]{rocca-rossi-deg}),
thus obtaining
\begin{equation}
\label{added-4-clarity}
\begin{aligned}
 & -\int_0^t  \uu_{tt}\,\cdot\, \mathbf{\dive} (\vism\eps(\uu_t)) \dd
x \dd s  + \int_0^t \int_{\Omega} \dive(a(\chi)\vism\eps(\uu_t))
\,\cdot\, \mathbf{\dive} (\vism\eps(\uu_t)) \dd x \dd s
\\
&
 =-
 \int_0^t \int_{\Omega}
\dive(b(\chi)\elm\eps({\uu}))  \,\cdot\, \mathbf{\dive}
(\vism\varepsilon(\uu_t)) \dd x \dd s + \rho \int_0^t \int_\Omega
\nabla \teta \cdot \mathbf{\dive} (\vism\eps(\uu_t)) \dd x \dd s -
\int_0^t \int_\Omega \mathbf{f}\,\cdot \mathbf{\dive}
(\vism\eps(\uu_t)) \dd x \dd s\,.
\end{aligned}
\end{equation}
Then, we consider the occurring terms individually. The kinetic term gives
$$\int_0^t \io  \uu_{tt} \mathbf{\dive} (\vism\eps(\uu_t)) \dd
x \dd s =  \io \frac12 \eps(\uu_t (t))\vism \eps(\uu_t (t)) \dd x
-\io \frac12 \eps(\uu_t (0))\vism \eps(\uu_t (0)) \dd x.
$$
  For the viscous term
  on the left-hand side of \eqref{added-4-clarity}
   we
 rely  on   \eqref{data-a}  and  the elliptic regularity result in
   \eqref{cigamma} and we obtain
   \begin{equation}
   \label{est-to-fill-2}
\begin{aligned}
\int_0^t \int_{\Omega} \mathbf{\dive} (a(\chi)\vism\eps(\uu_t))
\,\cdot\, \mathbf{\dive} (\vism\eps(\uu_t)) \dd x \dd
s=&\int_0^t\int_\Omega a(\chi) \mathbf{\dive} (\vism
\varepsilon(\uu_t))\,\cdot\,\mathbf{\dive} (\vism\varepsilon(\uu_t))
\dd x \dd s
 \\
&+
 \int_0^t \int_\Omega \nabla a(\chi)\vism \tensoret\,\cdot\,{\mathbf{\dive} (\vism\varepsilon(\uu_t))}
\dd x \dd s
 \\
&\geq C  \int_0^t \io |\mathbf{\dive} (\vism\eps(\uu_t))|^2 \dd
x \dd s +I_1  \\
&\geq c\int_0^t \| \uu_t \|_{H^2(\Omega;\RR^d)}^2 \dd s+I_1.
\end{aligned}
\end{equation}
We then move
$I_1$ to the right-hand side of \eqref{added-4-clarity}
  and estimate
  \begin{equation}
  \label{est-to-fill-bis}
\begin{aligned}
|I_1|&=\left| \int_0^t\int_\Omega\nabla a(\chi)\vism \tensoret\,\cdot\,\mathbf{\dive} (\vism\eps(\uu_t)) \right| \dd x \dd s \\
& \leq C\int_0^t\|\nabla a(\chi)\|_{L^{d+\zeta}(\Omega;\R^d)}\|\tensoret\|_{L^{d^{\star}-\zeta}(\Omega;\RR^{d\times d})}\|{\mathbf{\dive} (\vism\varepsilon(\uu_t))}\|_{L^2(\Omega; \RR^d)}
\dd s
 \\
& \leq \delta \int_0^t\|\uu_t\|_{H^2(\Omega;\RR^d)}^2 \dd s +C_\delta\int_0^t
\|\nabla a(\chi)\|_{L^{d+\zeta}(\Omega;\R^d)}^2  \|\tensoret\|_{L^{d^{\star}-\zeta}(\Omega;\RR^{d\times d})}^2 \dd s \\
&\leq \delta \int_0^t\|\uu_t\|_{H^2(\Omega;\RR^d)}^2 \dd s +C_\delta\varrho^2\int_0^t\|\chi\|_{W^{1,p}(\Omega)}^2\|\uu_t\|_{H^2(\Omega;\RR^d)}^2 \dd s +C_\delta C_\varrho\int_0^t\|\chi\|_{W^{1,p}(\Omega)}^2\|\uu_t\|_{L^2(\Omega;\RR^d)}^2 \dd s,
\end{aligned}
\end{equation}
  exploiting \eqref{data-a} and  \eqref{interp},  for some
positive constants $\delta $ and $\varrho$ that we will choose later
and for some $C_\delta, \, C_\varrho>0$. For the purely elastic
contribution on the right-hand side of \eqref{added-4-clarity} we
argue in this way: using the assumption $p>d$, we can fix $\zeta>0$
such that $p\geq d+\zeta$ and we get, due also to \eqref{data-a},
\begin{equation}
\label{est-to-fill-1}
\begin{aligned}
 & -
 \int_0^t \int_{\Omega}
\dive(b(\chi)\elm\eps({\uu}))  \,\cdot\, \mathbf{\dive}
(\vism\varepsilon(\uu_t)) \dd x \dd s
\\ &
= -\int_0^t\int_\Omega\nabla b(\chi)\elm \tensore {\mathbf{\dive}
(\vism\varepsilon(\uu_t))}
\dd x \dd s -\int_0^t\int_\Omega b(\chi){\mathbf{\dive} (\elm \varepsilon(\uu))}{\mathbf{\dive}(\vism\varepsilon(\uu_t))} \dd x \dd s \\
&\leq C\int_0^t\|\nabla b(\chi)\|_{L^{d+\zeta}(\Omega;\R^d)}\|\tensore\|_{L^{d^{\star}-\zeta} (\Omega;\RR^{d\times d})}\|{\mathbf{\dive} (\vism\varepsilon(\uu_t))}\|_{L^2(\Omega;\RR^d)}
\dd s+C\int_0^t\|\ub\|_{H^2(\Omega;\RR^d)}\|\ub_t\|_{H^2(\Omega;\RR^d)} \dd s\EEE \\
&\leq  \sigma\int_0^t\|\ub_t\|_{H^2(\Omega;\RR^d)}^2 \dd s +C_\sigma \int_0^t  \left(
\|\chi\|_{W^{1,p}(\Omega)}^2\|\ub\|_{H^2(\Omega; \RR^d)}^2+\|\ub\|_{H^2(\Omega;\RR^d)}^2\right)  \dd s\,.
\end{aligned}
\end{equation}
 Here,
$d^{\star}$ is from  \eqref{dstar} and we have exploited inequality \eqref{interp}
with a constant $\sigma$ that we will choose later, and some $C_\sigma>0$. Moreover,
 we have used \eqref{cigamma} and the fact that $\|\nabla(b(\chi))\|_{L^{d+\zeta}(\Omega)}\leq C\|\chi\|_{W^{1,p}(\Omega)}$.
 For the thermal expansion term we have that
\begin{equation}
\label{est-to-fill-3} \left| \rho \int_0^t \int_\Omega \nabla \teta
\cdot \mathbf{\dive} (\vism\eps(\uu_t)) \dd x \dd s \right|
\leq \eta \int_0^t\|\uu_t\|_{H^2(\Omega;\RR^d)}^2 \dd s+C_\eta\int_0^t
\|\nabla\teta\|_{L^2(\Omega;\R^d)}^2 \dd s
\end{equation}
holds true for  some   positive constant $\eta$   to be fixed  later
and for some $C_\eta>0$. Collecting
\eqref{est-to-fill-2}--\eqref{est-to-fill-3},
the previously proved  estimates
 \eqref{est1},
\eqref{crucial-est3.2},  and exploiting \eqref{bulk-force}  on
$\mathbf{f}$ to estimate the last term on the right-hand side of
\eqref{added-4-clarity},
 we conclude  that
\begin{align}\nonumber
\frac{\beta_0}{2} \io |\eps(\uu_t (t) ) |^2  \dd x +  c \int_0^t \|
\uu_t \|_{H^2(\Omega;\RR^d)}^2 \dd s \leq &\,  C \io |\eps(\vv_0 )
|^2  \dd x + C \| \mathbf{f}\|_{L^2 (0,T;\Ha)}^2  +
\frac{c}2 \int_0^t \| \uu_t \|_{H^2(\Omega;\RR^d)}^2 \dd s \\
\nonumber
&+C\left(1+\|\uu_0\|_{H^2(\Omega;\RR^d)}^2+\int_0^t\int_0^s\|\uu_t\|_{H^2(\Omega;\RR^d)}^2\, \dd r \dd s\right)\,,
\end{align}
with $\beta_0$ from \eqref{ellipticity}, \erdue
 where we have used the fact that $\int_0^t\|\uu\|_{H^2(\Omega;\RR^d)}^2 \dd s \leq \|\uu_0\|_{H^2(\Omega;\RR^d)}^2+\itt\int_0^s \|\uu_t\|_{H^2(\Omega;\RR^d)}^2 \dd r \dd s $ and chosen $\sigma$, $\delta$, $\varrho$ and $\eta$ sufficiently small.
Taking into account condition \eqref{bulk-force} on $\mathbf{f}$, the assumptions on the initial data \eqref{datou}, and using a standard Gronwall lemma, we
conclude
\begin{equation}
\label{palla} \| \uu_t \|_{ L^{2}(0,\berdue T\erdue; \boY)\cap
L^{\infty}(0,\berdue T\erdue; \boZ) } \leq C.
\end{equation}
 By comparison in \eqref{eqI}, taking into account the regularity property \eqref{reg-pavel-b},
we also get
\begin{equation}
\label{utt-comparison} \| \uu_{tt} \|_{L^{2}(0,\berdue T\erdue; \Ha)}
\leq C.
\end{equation}

\paragraph{\bf Sixth  estimate  [$\mu\in\{0,1\}$]} 
We multiply \eqref{eq0} by
$\frac{w}{\teta}$, with $w$ a test function in $W^{1,d}(\Omega)  \cap L^\infty(\Omega)$  (in particular, this is true for $w \in W^{1,d+\epsilon}(\Omega)$ with $\epsilon>0$).
We integrate
in space, only.  \berdue Thus, using  the place-holders $H :=   - \chi_t -\rho
\mathrm{div}(\uu_t)$ and $ J:= \frac1\teta (g+ a(\chi) \eps(\uu_t)
\vism \eps(\uu_t) + |\chi_t|^2)$, \erdue we obtain (cf.\ \eqref{later-4-comparison})
that
\[
\begin{aligned}
&
\left|  \int_\Omega \partial_t \log(\teta) w \dd x \right|
\\
& \berdue
=  \left|   \int_\Omega \left( H w -  \frac{\mathsf{K}(\teta)}\teta \nabla \teta
\cdot \nabla w   - \frac{\mathsf{K}(\teta)}{\teta^2}
|\nabla\teta|^2 w + Jw \right)   \dd x
 +\int_{\partial\Omega} h \frac{w}{\teta}    \dd S
\right| \erdue
\\
&
 \leq \left|\int_\Omega H w \dd x \right|
+ \left| \int_\Omega \frac{\mathsf{K}(\teta)}\teta \nabla \teta
\cdot \nabla w \dd x  \right| +
 \left| \int_\Omega  \frac{\mathsf{K}(\teta)}{\teta^2}
|\nabla\teta|^2 w   \dd x  \right| + \left|\int_\Omega J w \dd x \right|
+ \left|\int_{\partial\Omega} \berdue h\frac{w}\teta \erdue \dd S\right|
\\
&
 \doteq I_1+I_2+I_3+I_4 +I_5.
\end{aligned}
\]
Estimate \eqref{est5} yields that $\|H\|_{L^2(0,T; L^2(\Omega))} \leq C$,
therefore $|I_1| \leq \mathcal{H}(t) \|w \|_{L^2(\Omega)}$ with
$\mathcal{H}(t)= \| H(\cdot,t)\|_{L^2(\Omega)} \in L^2(0,T)$.
Analogously, also in view of \eqref{heat-source} and of
\eqref{teta-pos} we have  
 that
\begin{equation}
\label{q-big1} |I_4| \leq \frac{1}{\teta_*}\mathcal{J}(t)
\|w\|_{L^\infty(\Omega)} \qquad \text{with } \mathcal{J}(t) := \|
J(\cdot,t)\|_{L^1(\Omega)} \in L^1(0,T).
\end{equation}
Moreover, \berdue  $|I_5| \leq  \frac{1}{\teta_*}  \|h(t)\|_{L^2(\partial \Omega)} \| w\|_{L^2(\partial \Omega)} $, \erdue with $  \|h(t)\|_{L^2(\partial \Omega)}    \in L^1(0,T)$ thanks to \eqref{dato-h}.
Using the growth condition \eqref{hyp-K} for $\condu$, we estimate
\begin{equation}
\label{q-big2} |I_2| \leq C \int_\Omega \teta^{\kappa-1} |\nabla
\teta| |\nabla w| \dd x + C \int_\Omega \frac1\teta |\nabla \teta|
|\nabla w| \dd x \doteq I_{2,1} + I_{2,2}. \end{equation}
Thanks to the previously proved positivity \eqref{teta-pos}, we have
\[
I_{2,2} \leq \frac C{\teta^*} \mathcal{O}(t)  \| \nabla
w\|_{L^2(\Omega;\RR^d)} \qquad \text{with }  \mathcal{O}(t) := \| \nabla
\teta(t) \|_{L^2(\Omega;\RR^d)} \in L^2(0,T)
\]
by  \eqref{crucial-est3.2}. We estimate $I_{2,1}$ via the
H\"older inequality, taking into account \eqref{additional-info} and
\eqref{necessary-added}, whence, for $d \in \{2,3\}$,
\[
\begin{aligned}
 I_{2,1} \leq C &  \| \teta^{(\kappa
+\alpha-2)/2} \nabla \teta \|_{L^2(\Omega;\RR^d)} \|\teta^{(\kappa
-\alpha)/2} \|_{L^{6}(\Omega)}  \| \nabla w\|_{L^{{3}}(\Omega;\RR^d)} \doteq C \mathcal{O}^{*}(t) \| \nabla w\|_{L^{{3}}(\Omega;\RR^d)}
\\
&\quad   \text{with }  \mathcal{O}^{*}(t):=  \| \teta(t)^{(\kappa
+\alpha-2)/2} \nabla \teta (t) \|_{L^2(\Omega;\RR^d)} \|\teta(t)^{(\kappa
-\alpha)/2} \|_{L^{6}(\Omega)}  \in L^1(0,T).
\end{aligned}
\]
 Finally, we have
 \begin{equation}
\label{q-big3} |I_3| \leq C \int_\Omega \teta^{\kappa-2} |\nabla
\teta|^2 |w |\dd x + C \int_\Omega \frac1{\teta^2} |\nabla \teta|^2
 |w |\dd x \doteq I_{3,1} + I_{3,2}. \end{equation}
The positivity property \eqref{teta-pos}
again guarantees
\[
I_{3,2}  \leq \frac C{\teta_*^2}   \mathcal{O}(t)^2 \| w\|_{L^\infty (\berdue\Omega\erdue)}\qquad \text{with } \mathcal{O}(t)^2\in L^1(0,T)
\]
while, using that $\teta^{\kappa -2} \leq c \teta^{\kappa+\alpha-2} +c'$, we infer
\begin{equation}
\label{i32}
\begin{aligned}
I_{3,2} &  \leq
  \|w \|_{L^\infty(\Omega)} \left( c \int_\Omega \teta^{\kappa+\alpha-2} |\nabla \teta|^2 \dd x + c' \int_\Omega  |\nabla \teta|^2 \dd x
  \right)
  \doteq  \|w \|_{L^\infty(\Omega)}  \mathcal{O}_*(t)
  \\
  &
\qquad \text{with }  \mathcal{O}_*(t)= c \int_\Omega \teta(t)^{\kappa+\alpha-2} |\nabla \teta(t)|^2 \dd x + c' \int_\Omega  |\nabla \teta(t)|^2 \dd x \in L^1(0,T),
\end{aligned}
\end{equation}
thanks to \eqref{additional-info} and \eqref{crucial-est3.2}.

Collecting all of the  above calculations,
 we conclude that
\begin{equation}\label{est6}
\|\partial_t\log(\teta)\|_{L^1(0,T;  (W^{1,d}(\Omega) \cap L^\infty (\Omega))^*)} \leq C.
\end{equation}
\paragraph{\bf  Seventh  estimate [$\mu\in\{0,1\}$], $\kappa \in (1,5/3)$ if $d=3$ and $\kappa \in (1,2)$ if $d=2$}
Assume in addition \textbf{Hypothesis (V)}.
We multiply
 \eqref{eq0} by a test function $w \in  W^{1,\infty}(\Omega)$ (which e.g.\ holds if $w \in W^{2,d+\epsilon}(\Omega)$ for
 $\epsilon>0$).
  By comparison we have
 \[
\begin{aligned}
\left|  \int_\Omega \teta_t w \dd x \right|
 \leq \left|\int_\Omega L w \dd x \right|
+ \left| \int_\Omega \mathsf{K}(\teta) \nabla \teta
\cdot \nabla w \dd x  \right| + \left|\int_{\partial\Omega} hw \dd S\right| \doteq I_1+I_2+I_3,
\end{aligned}
\]
where we have set $L= -\chi_t\teta-\rho\teta \mathrm{div}(\uu_t)+g+a(\chi)\eps(\uu_t)\vism \eps(\uu_t) +|\chi_t|^2$.
Therefore,
\[
|I_1| \leq \mathcal{L}(t) \|w\|_{L^\infty (\Omega)} \quad \text{with } \mathcal{L}(t):=\|L(t)\|_{L^1(\Omega)} \in L^1(0,T), \quad
|I_3| \leq \| h(t) \|_{L^2(\partial
\Omega)} \| w\|_{L^2(\partial
\Omega)}  \text{ with } h\in L^1(0,T)
\]
thanks to  \eqref{crucial-est3.2},  \eqref{est5} and \eqref{dato-h}, respectively.  As for $I_2$,
in view of \eqref{hyp-K},
 taking into account \eqref{additional-info} and using  the  H\"older inequality, we obtain
\begin{equation}
\label{citata-dopo-ehsi}
|I_2|\leq   C\| \teta^{(\kappa-\alpha+2)/2} \|_{L^2(\Omega)} \|\teta^{(\kappa+\alpha-2)/2} \nabla \teta\|_{L^2(\Omega;\RR^d)} \|\nabla w\|_{L^\infty (\Omega;\RR^d)}
+   C \| \nabla \teta\|_{L^2(\Omega;\RR^d)}  \|\nabla w\|_{L^2 (\Omega;\RR^d)}.
 \end{equation}
Observe that, since $\alpha$ can be chosen arbitrarily close to $1$, in view of estimate
\eqref{estetainterp}
we have that $\teta^{(\kappa-\alpha+2)/2}$ is bounded in $L^2(0,T; L^2(\Omega))$ if and only if
$\kappa <\frac53$ if $d=3$, and $\kappa <2$ if $d=2$. 
Under this restriction on $\kappa$, we have that $|I_2|\leq C  \mathcal{L}^*(t) \|\nabla w \|_{L^\infty (\Omega)} $ for some
$\mathcal{L}^* \in L^1(0,T)$.
Ultimately, we conclude that
\begin{equation}
\label{bv-esti-temp}
\|\teta_t\|_{L^1(0,T; W^{1,\infty}(\Omega)^*)} \leq C.
\end{equation}
\paragraph{\bf  Eighth estimate  [$\mu=0$].}
In view of the previously obtained estimates \eqref{est1}, \eqref{crucial-est3.2},  \eqref{est5},
and \eqref{palla},
 a comparison in
equation \eqref{eqII} yields that  (recall that $\xi $ is a selection in $\beta(\chi)$ a.e.\ in $\Omega \times (0,T)$),
\[
\| A_p(\chi) +\xi \|_{L^2(0,T; L^2(\Omega))} \leq C.
\]
\berdue Now, in view of  the monotonicity of the operator $\beta: \R \rightrightarrows \R$   (cf., e.g.,
 \cite[Lemma 3.3]{akagi}), from the above estimate we
deduce \erdue
\begin{equation}
\| A_p(\chi)\|_{L^2(0,T; L^2(\Omega))}+ \|\xi \|_{L^2(0,T; L^2(\Omega))} \leq C.
\end{equation}
In view of the regularity results  \cite[Thm.\ 2, Rmk.\ 2.5]{savare98}, we finally infer the enhanced regularity
\eqref{enhanced-chi} for $\chi$.
\QED

\begin{remark}[\bf The  $p$-Laplacian regularization]
\upshape
\label{rmk:role-p-Lapl}
A close perusal at the above calculations shows that the fact that
 $p>d$ for the $p$-Laplacian  term in the $\chi$-equation \eqref{eqII} has been used only for carrying out the calculations in the \emph{Fifth estimate}. All the other estimates do not depend on the condition $p>d$, and
 would therefore hold if the operator $A_p$ in  \eqref{eqII} were replaced by the  Laplacian.

  In turn, the \emph{Fifth estimate} for $\uu$ will play a crucial role in  the limit passage arguments at the basis of the proofs of
 Theorems \ref{teor3} and \ref{teor1}: it will ensure \berdue compactness in the strong topology of \erdue $H^1 (0,T; \boY)$ (cf.\
  Lemma \ref{l:compactness}) for the sequences of  approximate solutions constructed in Sec.\ \ref{s:time-discrete}. Relying on this, we will be able to pass to the limit with the quadratic term $|\tensoret|^2$ on the right-hand side of \eqref{eq0}.

  Nonetheless, in Sec.\ \ref{s:final} we will show that, in the case $\mu=1$ of unidirectional evolution, it is ultimately possible to drop the constraint $p>d$ and in fact we will obtain an existence result for the entropic formulation of system \eqref{eq0}--\eqref{eqII},  in the case \eqref{eqII} simply features the  Laplacian (i.e. for $p=2$).  
\end{remark}


\section{\bf Time discretization}
\label{s:time-discrete} \noindent In Section \ref{ss:3.1} we set up
a \emph{single}
 time-discretization scheme
 for both the irreversible ($\mu=1$) and for the reversible ($\mu=0$) systems. We then show in Section \ref{ss:3.1bis}
  that the piecewise
  constant and piecewise linear interpolants of the discrete solutions satisfy
  the approximate versions of the total energy inequality, the entropy inequality, and
  of equations \eqref{eqI}--\eqref{eqII}.  Finally, in Section \ref{ss:3.2} we rigorously prove the a priori estimates from
  Section \ref{s:aprio} in the time-discrete context.
%
\begin{notation}
\upshape \label{not-alpha} In what follows, also in view of the
extension \eqref{1-homog} mentioned at the end of Sec.\
\ref{glob-irrev},  we will use $\widehat{\alpha}$ and $\alpha$ as
place-holders for $I_{(-\infty,0]}$ and $\partial I_{(-\infty,0]}$.
\end{notation}
\subsection{Setup of the time discretization}
\label{ss:3.1}
 We consider
an equidistant partition of $[0,T]$, with time-step $\tau>0$ and
nodes $t_\tau^k:=k\tau$, $k=0,\ldots,K_\tau$. In this framework, we
approximate the data $\mathbf{f}$, $g$, and $h$
 by local means, i.e.
setting for all $k=1,\ldots,K_{\tau}$
\begin{equation}
\label{local-means} \ftau{k}:=
\frac{1}{\tau}\int_{t_\tau^{k-1}}^{t_\tau^k} \mathbf{f}(s)\dd s\,,
\qquad  \gtau{k}:= \frac{1}{\tau}\int_{t_\tau^{k-1}}^{t_\tau^k} g(s)
\dd s\,, \qquad \htau{k}:= \frac{1}{\tau}\int_{t_\tau^{k-1}}^{t_\tau^k} h(s)
\dd s\,.
\end{equation}
Consider the following initial data
\begin{align}\label{IC2}
\wtau{0}:=\w_{0}, \qquad \utau{0}:=\uu_{0},\qquad
\utau{-1}:=\uu_{0}-\tau \vv_0, \qquad \chitau{0}:=\chi_{0}.
\end{align}

We construct discrete solutions to system \eqref{eq0}--\eqref{eqII}
by solving the  following  elliptic system, featuring the  operator $\mathcal{A}^k: X  \to H^1(\Omega)^*$, with
 \begin{equation}
\label{spaces}
\begin{gathered}
 X= \{ \theta  \in H^1(\Omega)\, : \  \int_\Omega \condu(\theta) \nabla \theta \cdot \nabla v  \dd x  \text{ is well defined for all } v \in H^1 (\Omega)\},
 \quad \mathcal{A}^k: X  \to  H^1(\Omega)^* \text{  defined by }\\
\pairing{}{H^1(\Omega)}{ \mathcal{A}^k(\theta) }{v}:=
  \int_\Omega \condu(\theta) \nabla \theta \cdot \nabla v \dd x - \int_{\partial \Omega} \htau{k} v \dd S\,.
\end{gathered}
\end{equation}
\begin{problem}
\label{prob:rhoneq0}
Starting from $(\utau{0},$ $\utau{-1},$
$\chitau{0},$ $\wtau{0})$ as in \eqref{IC2}, find
 $\{\wtau{k}, \utau{k}, \chitau{k}\}_{k=1}^{K_\tau}
\subset X \times \boY \times W^{1,p}(\Omega)$ fulfilling
\begin{align}
& \label{eq-discr-w} \frac{\wtau{k} -\wtau{k-1}}{\tau} +
\frac{\chitau{k} -\chitau{k-1}}{\tau}\wtau{k} +
\rho\dive\left(\frac{\utau{k}-\utau{k-1}}{\tau}\right)\wtau{k}
+  \mathcal{A}^k(\wtau{k})   = \gtau{k} \\
\no
&\qquad\qquad+
 a(\chitau{k-1} )   \eps\left(\frac{\utau{k}-\utau{k-1}}{\tau}\right)\vism\eps\left(\frac{\utau{k}-\utau{k-1}}{\tau}\right)+\left|\frac{\chitau{k}
-\chitau{k-1}}{\tau}\right|^2
 + \frac{ \tau^{1/2} }{2}  \left|\frac{\chitau{k}
-\chitau{k-1}}{\tau}\right|^2
\quad \text{in } H^1(\Omega)^*,
\\
& \label{eq-discr-u} \frac{\utau{k} -2\utau{k-1} +
\utau{k-2}}{\tau^2} + \opj{ a(\chitau{k-1} )  }{\frac{\utau{k}-\utau{k-1}}{\tau}} +
\oph{b(\chitau{k})}{\utau{k}}  + \ciro(\wtau{k}) 
= \ftau{k} \quad \aein\, \Omega,
\\
& \label{eq-discr-chi}
\begin{aligned}
 \frac{\chitau{k} -\chitau{k-1}}{\tau}   + \sqrt{\tau}   \frac{\chitau{k} -\chitau{k-1}}{\tau}  + \mu \zetau{k}
+ A_p(\chitau{k}) +\xitau{k}+\gamma(\chitau{k}) 
  \ni -
b'( \chitau{k} )\frac{\eps(\utau{k-1})\elm
\eps(\utau{k-1})}2 + \wtau{k}
  \quad \aein\, \Omega\,,
  \end{aligned}
\end{align}
where
$\mathbb{I}  \in \R^{d\times d \times d \times d}$ denotes the identity tensor and
\begin{align}
\label{xitau}
&
\xitau{k} \in  \beta(\chitau{k})  && \aein\, \Omega,
\\
&
\label{zetau}
\zetau{k} \in \alpha \left( \frac{\chitau{k} -\chitau{k-1}}{\tau}  \right)  && \aein\, \Omega.
\end{align}
\end{problem}
\begin{remark}[Features of the time-discretization scheme]
\upshape
\label{rmk:discrete-features}
A few observations on Problem \ref{prob:rhoneq0}
are in order.

 First of all, let us point out that  the scheme is fully implicit   and, in particular,  \eqref{eq-discr-chi} is coupled to the  system  \eqref{eq-discr-w}--\eqref{eq-discr-u} by the
 implicit term $ \wtau{k}$  on the right-hand side. This will be crucial for    proving the strict positivity
\eqref{strict-pos-wk} below for the discrete temperature $\wtau{k}$.
Indeed, our argument for \eqref{strict-pos-wk} is the discrete
version of the comparison argument  developed at the beginning of
Section \ref{s:aprio} and strongly relies on the structure of the
discrete temperature equation  \eqref{eq-discr-w}. However, in the
case of unidirectional evolution, we could have decoupled the
discrete equation for $\chi$ from
\eqref{eq-discr-w}--\eqref{eq-discr-u}, replacing
\eqref{eq-discr-chi} by
\begin{equation}
\label{chi-discre-irrv}
 \frac{\chitau{k} -\chitau{k-1}}{\tau}  + \mu \zetau{k}
+ A_p(\chitau{k}) +\xitau{k}+\gamma(\chitau{k}) 
  \ni -
b'( \chitau{k} )\frac{\eps(\utau{k-1})\elm
\eps(\utau{k-1})}2 + \wtau{k-1}
  \quad \aein\, \Omega\,,
\end{equation}
and, accordingly, replacing the coupling term $\frac{\chitau{k} -\chitau{k-1}}{\tau}\wtau{k} $ on the left-hand side of
\eqref{eq-discr-w} by $\frac{\chitau{k} -\chitau{k-1}}{\tau}\wtau{k-1}$. In Remark \ref{rmk:still-pos} below, we will show how  it is still possible to
prove the strict positivity of the discrete temperature for this partially decoupled scheme.

 Second, observe  that  $  \frac{\tau^{1/2}}{2}  \left|\frac{\chitau{k}
-\chitau{k-1}}{\tau}\right|^2$ appears on the right-hand side of \eqref{eq-discr-w} and, accordingly, $\sqrt \tau \frac {\chitau{k}
-\chitau{k-1}}{\tau}$ features on the left-hand side of \eqref{eq-discr-chi}.
 These terms have been added for technical reasons, related to the proof of the discrete version of the  total energy inequality \eqref{total-enid},
 cf.\  the  comments above  Proposition \ref{prop:discr-enid}.  Clearly, they will disappear
 when  passing to the limit with $\tau \down 0$.

Because of the implicit character of system \eqref{eq-discr-w}--\eqref{eq-discr-chi}, for the existence proof
(cf.\ Lemma \ref{lemma:ex-discr} below)
we shall have to resort to
a fixed-point type result from the theory for  elliptic systems featuring pseudo-monotone operators, drawn from  \cite[Chap.\ II]{roub-NPDE}.
Indeed, we will not apply it directly to system \eqref{eq-discr-w}--\eqref{eq-discr-chi}, but to an approximation of \eqref{eq-discr-w}--\eqref{eq-discr-chi},
 i.e.\ system \eqref{eq-discr-w-TRUNC}--\eqref{eq-discr-chi-TRUNC} below,
 obtained in the following way.
We will  need to
\begin{compactenum}
\item truncate $\condu$,  along the lines of \cite{gmrs},  in such a way as to have a bounded function in the elliptic operator in the temperature equation \eqref{eq-discr-w}. Therefore, the truncated operator   $\condu_M$, with $M$ a positive parameter,
shall be defined on $H^1(\Omega)$ (in place of $X$), with values in $H^1(\Omega)^*$ (in place of $X^*$). Accordingly, we shall truncate all occurrences of $\teta$
in a quadratic term;
\item  following \cite{roubiSIAM10},  add the higher order terms
 $-\nu \mathrm{div}(|\eps(\utau{k})|^{\eta-2} \mathbb{I} \eps(\utau{k})) $
 and $ \nu |\chitau{k}|^{\eta-2} \chitau k  $, with
  $\nu>0$ and $\eta>4$,  on the left-hand sides of \eqref{eq-discr-u} and \eqref{eq-discr-chi}, respectively. Their role is to compensate the quadratic terms
 on the right-hand side of  \eqref{eq-discr-w}. As a result,
 both for $d=2$ and for $d=3$
   the pseudo-monotone operator by means of which we will rephrase system
    \eqref{eq-discr-w-TRUNC}--\eqref{eq-discr-chi-TRUNC}
 will turn out to be coercive, in its $\teta$-component, with respect to the $H^{1}(\Omega)$-norm;
 \item in the case $\mu =1$, in order to cope with the (possible) unboundedness of the operator $\alpha$ we will have to replace it with its Yosida-regularization
 $\alpha_\nu$  (cf. \cite{brezis}), with $\nu$ the same parameter as above.
 \end{compactenum}

 Then, in the proof of Lemma \ref{lemma:ex-discr} we will
 \begin{compactenum}
\item prove the existence of solutions to  the approximate  discrete  system \eqref{eq-discr-w-TRUNC}--\eqref{eq-discr-chi-TRUNC};
\item pass to the limit in 
 \eqref{eq-discr-w-TRUNC}--\eqref{eq-discr-chi-TRUNC} first as the truncation parameter $M \to \infty$
 and conclude an existence result for  an approximation of system \eqref{eq-discr-w}--\eqref{eq-discr-chi}, still depending on the parameter $\nu>0$;
 \item
 pass to the limit  in this approximate system as $\nu \to 0$ and conclude the existence of solutions to \eqref{eq-discr-w}--\eqref{eq-discr-chi}.
\end{compactenum}
 We postpone to Remark \ref{rmk:2parameters} some comments on the reason why we need to keep the two limit passages as $M \to \infty$ and $\nu \to 0$ distinct.
\end{remark}

Our existence result for Problem \ref{prob:rhoneq0}  reads
\begin{lemma}[Existence for the time-discrete Problem~\ref{prob:rhoneq0}, $\mu\in \{0,1\}$]
\label{lemma:ex-discr}
 Assume   \beo \textbf{Hypotheses (0)--(III)}, \eo and assumptions
\eqref{bulk-force}--\eqref{datochi} on the data $\mathbf{f},\,g,\, h,\,
\teta_0,\, \uu_0,\,\vv_0,\, \chi_0$.
Then,
there exists $\bar{\tau}>0$ such that for all $0 <\tau \leq \bar{\tau}$
 Problem \ref{prob:rhoneq0},
admits at least one solution  $\{(\wtau{k}, \utau{k},
\chitau{k})\}_{k=1}^{K_\tau}$.

Furthermore,
  any so\-lu\-tion   $\{(\wtau{k}, \utau{k}, \chitau{k})\}_{k=1}^{K_\tau}$   of Problem
  \ref{prob:rhoneq0}
  fulfills
  \begin{equation}
\label{strict-pos-wk}
 \wtau{k}(x) \geq \underline{\teta}>0  \quad \foraa\, x \in \Omega
 \end{equation}
  for some $\underline{\teta} = \underline{\teta}(T)$.
\end{lemma}
%

\begin{proof}
We split the proof in some steps.

\paragraph{\bf Step $1$: approximation.}
 As already mentioned, we construct our approximation of system \eqref{eq-discr-w}--\eqref{eq-discr-chi}
by truncating $\condu$ in \eqref{eq-discr-w} and the quadratic terms in $\teta$, replacing $\alpha$ with its Yosida approximation $\alpha_\nu$, and adding higher order
terms to \eqref{eq-discr-u} and \eqref{eq-discr-chi}. Namely,
let
\begin{equation}
\label{def-k-m} \condu_M(r):=  \left\{ \begin{array}{ll} \condu(-M) & \text{if
} r <-M,
\\
\condu(r)   & \text{if } |r| \leq M,
\\
\condu(M) & \text{if } r >M
\end{array}
\right.
\end{equation}
and accordingly introduce the operator
\begin{equation}
\label{M-operator}
\mathcal{A}_M^k: H^1(\Omega)  \to H^1(\Omega)^*
\  \text{  defined by }
\pairing{}{H^1(\Omega)}{\mathcal{A}_M^k(\theta)}{v}:= \int_\Omega \condu_M(\theta) \nabla \theta \cdot \nabla v \dd x -
\int_{\partial \Omega} \htau{k}v \dd S.
\end{equation}
Observe that, thanks to \eqref{hyp-K}  there still holds $\condu_M(r) \geq c_{0} $ for all $r \in \R$, and therefore
\begin{equation}
\label{ellipticity-retained}
\pairing{}{H^1(\Omega)}{ \mathcal{A}_M^{k}  (\theta)}{\theta} \geq c_0 \int_\Omega |\nabla \theta|^2 \dd x \qquad \text{for all } \theta \in H^1(\Omega).
\end{equation}
We also introduce the truncation operator $\mathcal{T}_M : \R \to \R$
\begin{equation}
\label{def-truncation-m} \mathcal{T}_M(r):= \left\{ \begin{array}{ll} -M &
\text{if } r <-M,
\\
r   & \text{if } |r| \leq M,
\\
M  & \text{if } r >M.
\end{array}
\right.
\end{equation}
Furthermore,   for a given $\nu>0$ we denote by $\alpha_\nu$ the Yosida approximation of $\alpha$ with parameter $\nu$.

Then, we consider the
  following approximation of  system
\eqref{eq-discr-w}--\eqref{eq-discr-chi}:
\begin{align}
& \label{eq-discr-w-TRUNC} \frac{\wtau{k} -\wtau{k-1}}{\tau} +
\frac{\chitau{k} -\chitau{k-1}}{\tau}
\mathcal{T}_M(\wtau{k})  +
\rho\dive\left(\frac{\utau{k}-\utau{k-1}}{\tau}\right) \mathcal{T}_M(\wtau{k})
+  \mathcal{A}_M^k(\wtau{k})   = \gtau{k} \\
\no
&\qquad+
a( \chitau{k-1}  )\eps\left(\frac{\utau{k}-\utau{k-1}}{\tau}\right)\vism\eps\left(\frac{\utau{k}-\utau{k-1}}{\tau}\right)+\left|\frac{\chitau{k}
-\chitau{k-1}}{\tau}\right|^2   +  \frac{\tau^{1/2}}{2}   \left|\frac{\chitau{k}
-\chitau{k-1}}{\tau}\right|^2
 \quad \text{in $H^1(\Omega)^*$,}
\\
& \label{eq-discr-u-TRUNC}
\begin{aligned}
&
 \frac{\utau{k} -2\utau{k-1} +
\utau{k-2}}{\tau^2} + \opj{ a(\chitau{k-1} )}{\frac{\utau{k}-\utau{k-1}}{\tau}} +
\oph{b(\chitau{k})}{\utau{k}}  + \ciro(\mathcal{T}_M(\wtau{k}) ) - \nu  \mathrm{div}(|\eps(\utau{k})|^{\eta-2} \mathbb{I} \eps(\utau{k}))
= \ftau{k} \\
&  \hspace{12cm}   \text{in }  W_0^{1,\eta}(\Omega;\R^d)^*,
\end{aligned}
\\
& \label{eq-discr-chi-TRUNC}
\begin{aligned}
 &\frac{\chitau{k} -\chitau{k-1}}{\tau}   +  \sqrt \tau \frac {\chitau{k}
-\chitau{k-1}}{\tau}  + \mu \alpha_{\nu}   \left( \frac{\chitau{k} -\chitau{k-1}}{\tau}  \right)
+ A_p(\chitau{k}) + \xitau{k}+\gamma(\chitau{k})   +  \nu | \chitau{k}|^{\eta-2}\chitau k  \\
& \hspace{5cm}  =  -
b'( \chitau{k}  )\frac{\eps(\utau{k-1})\elm
\eps(\utau{k-1})}2 + \mathcal{T}_M(\wtau{k})\qquad \aein\, \Omega\,,
  \end{aligned}
\end{align}
with $\xitau{k} \in \beta(\chitau{k})$ a.e.\ in $\Omega$.

\paragraph{\bf Step $2$: existence of solutions for the approximate  system.}
Observe that system
\eqref{eq-discr-w-TRUNC}--\eqref{eq-discr-chi-TRUNC} can be recast
as
\begin{align}
&
\label{sistemone1}
\begin{aligned}
\wtau{k} &+\left(\chitau{k} -\chitau{k-1}\right) \mathcal{T}_M(\wtau{k}) +
\rho\dive\left({\utau{k}-\utau{k-1}}\right) \mathcal{T}_M(\wtau{k}) +\tau
\mathcal{A}_M^k(\wtau{k})
\\
& -
\tau a( \chitau{k-1} )\eps\left(\frac{\utau{k}-\utau{k-1}}{\tau}\right)\vism\eps\left(\frac{\utau{k}-\utau{k-1}}{\tau}\right) -  \tau
\left|\frac{\chitau{k}
-\chitau{k-1}}{\tau}\right|^2  -\frac{\tau^{3/2}}{2}  \left|\frac{\chitau{k}
-\chitau{k-1}}{\tau}\right|^2\\
&
=
\wtau{k-1}
+\tau\gtau{k} \quad \text{in $H^1(\Omega)^*$,}
\end{aligned}
\\
&
\label{sistemone2}
\begin{aligned}
 \utau{k}+\tau
\opj{a( \chitau{k-1} )}{(\utau{k}-\utau{k-1})}+\tau^2
\oph{b(\chitau{k})}{\utau{k}} &  + \tau^2 \ciro(\mathcal{T}_M(\wtau{k}))
- \nu \tau^2  \mathrm{div}(|\eps(\utau{k})|^{\eta-2}  \mathbb{I} \eps(\utau{k}))
\\
 & =
2\utau{k-1} - \utau{k-2}+\tau^2\ftau{k} \quad   \text{in }
W_0^{1,\eta}(\Omega;\R^d)^*,
\end{aligned}
\\
&
\label{sistemone3}
\begin{aligned}
  &
\begin{aligned}
  \chitau{k}  + \sqrt \tau \chitau k    & + \mu  \tau\alpha_{\nu}  \left( \frac{\chitau{k} -\chitau{k-1}}{\tau}  \right)
   + \tau A_p(\chitau{k})
\\
&
   + \tau \xitau{k} +\tau \gamma(\chitau{k})   + \nu\tau |\chitau{k}|^{\eta-2} \chitau{k}
 +\tau
b'(\chitau{k})\frac{\eps(\utau{k-1})\elm \eps(\utau{k-1})}2 - \tau
\mathcal{T}_M(\wtau{k})
\end{aligned}
\\
&
 =  \chitau{k-1}  + \sqrt \tau \chitau {k-1}
  \quad \aein\, \Omega\,.
  \end{aligned}
\end{align}

Denoting by $\mathcal{R}_{k-1}$ the operator acting on the unknown
$(\wtau k, \utau k, \chitau k)$ and by $H_{k-1}$ the vector of the
terms on the r.h.s.\ of the  above equations,  we can reformulate
system \eqref{sistemone1}--\eqref{sistemone3} in the abstract form
\begin{equation}
\label{abstract-sistemone} \mathcal{R}_{k-1}(\wtau k, \utau k,
\chitau k)= H_{k-1}.
\end{equation}

It can be checked that  $\mathcal{R}_{k-1}$ is a
  pseudo-monotone operator  (according to \cite[Chap.\ II, Def.\ 2.1]{roub-NPDE})  on
 $H^1(\Omega)\times W_0^{1,\eta} (\Omega;\R^d) \times H^1(\Omega)$.

    In order to check that  $\mathcal{R}_{k-1}$ is  coercive
on that space, it is sufficient to test \eqref{sistemone1} by
$\wtau{k}$, \eqref{sistemone2} by $\utau{k}$,  \eqref{sistemone3} by
$\chitau{k}$ and add the resulting equations. \beo  We will not
develop all the calculations in full detail, but rather point to the
most significant aspects.

Clearly, the term $- \nu \tau^2
\mathrm{div}(|\eps(\utau{k})|^{\eta-2}  \mathbb{I} \eps(\utau{k})) $
tested by $\utau{k}$ provides a bound for $\| \utau{k}
\|_{W^{1,\gamma} (\Omega;\R^d)}^{\gamma}$ via the Korn inequality.
Analogously we control $\| \chitau k\|_{H^1(\Omega)}^2$.  \eo
 To obtain a bound for
$\| \wtau{k}\|_{H^1(\Omega)}$  we use that $\mathcal{A}_M^k$ is
coercive (cf. \eqref{ellipticity-retained}). The additional terms
 $-\nu \mathrm{div}(|\eps(\utau{k})|^{\eta-2} \mathbb{I} \eps(\utau{k})) $
 and $ \nu |\chitau{k}|^{\eta-2}\chitau{k} $
 in  \eqref{sistemone2} and \eqref{sistemone3} enable us to control the quadratic terms on the right-hand side of \eqref{sistemone1}.
  More in detail, the test of \eqref{sistemone1} by $\wtau{k}$ gives rise, e.g., to the term
 $I_1 := \int_\Omega  a(\chitau{k-1}) \eps(\utau k) \vism   \eps(\utau k) \wtau{k} \dd x $, which can be estimated as follows
 \[
 \begin{array}{ll}
 |I_1| & \leq  C \|  a(\chitau{k-1})  \|_{L^\infty(\Omega)} \|  \eps(\utau k)  \|_{L^4(\Omega;\R^{d\times d})}^2 \| \wtau{k}\|_{L^2(\Omega)}
 \\ & \leq \frac14  \| \wtau{k}\|_{L^2(\Omega)}^2 + C  \|  \eps(\utau k)  \|_{L^4(\Omega;\R^{d\times d})}^4
\\ &  \leq \frac14  \| \wtau{k}\|_{L^2(\Omega)}^2 + \frac{\nu\tau^2}4  \|  \eps(\utau k)  \|_{L^{\eta}(\Omega;\R^{d\times d})}^{\eta} +C,
\end{array}
 \]
 where the first estimate follows from the H\"older inequality,  the second one from the fact that $ \|  a(\chitau{k-1})  \|_{L^\infty(\Omega)}\leq C$ since $\chitau{k-1} \in W^{1,p}(\Omega)$
 and $a \in \rmC^0 (\R)$, and the last one relies on
 the fact that
  $\eta>4$. Therefore,  for $\tau$ sufficiently small
   the right-hand side terms can be absorbed by the left-hand side ones, also resulting from the test of
 \eqref{sistemone2} by $\utau{k}$.
 With analogous calculations we estimate $I_2:= \int_\Omega (|\chitau k|^2+  \frac{\tau^{1/2}}2 |\chitau k|^2)  \wtau{k} \dd x$, exploiting the term $\nu\tau |\chitau{k}|^{\eta-2}\chitau{k}$ on the left-hand side
 of \eqref{sistemone3} \beo  which, tested by $\chitau{k}$, gives $\int_\Omega |\chitau{k}|^{\eta} \dd x
 $ on the left-hand side.

Since $\mathcal{R}_{k-1}$ is pseudo-monotone and coercive, \eo
the Leray-Lions type existence result of \cite[Chap.\ II, Thm.\
2.6]{roub-NPDE} applies, yielding the
existence of a solution $(\wtau k, \utau k, \chitau k)$ (whose dependence on  the parameters $M$ and $\nu$   
is not highlighted, for simplicity)
to \eqref{eq-discr-w-TRUNC}--\eqref{eq-discr-chi-TRUNC}.

\paragraph{\bf Step $3$: proof of the strict positivity \eqref{strict-pos-wk}.}
Observe first  that,  for $\wtau k$ solving \eqref{eq-discr-w-TRUNC}--\eqref{eq-discr-chi-TRUNC} the strict positivity
\eqref{strict-pos-wk} holds   for $k=0$  with $\underline\teta:=\teta_*$  due to \eqref{datoteta}.
 In order to prove that $\wtau{k} \geq \underline{\teta}>0$ a.e. in
$\Omega$,  for every $k\geq 1$, we proceed  in the same spirit of the proof of the strict positivity of $\teta$ in Sec.\ \ref{s:aprio} (cf. also \cite[Sec. 5.2]{kr-ro}). Namely,  we start by deducing  from \eqref{eq-discr-w}  that
\begin{equation}\label{ineqtetak}
\io \frac{\wtau{k} -\wtau{k-1}}{\tau} w \dd x +\io
 \condu_M(\wtau{k})\nabla\wtau{k}\nabla w \dd x \geq -C \io (\wtau{k})^2w  \dd x \quad \text{ for every  $w \in W^{1,2}_+(\Omega)$}
\end{equation}
 (cf.\ \eqref{label-added} for $W^{1,2}_+(\Omega)$), where
$C$ is independent of $k$. We now consider the
 decreasing sequence  $\{v_k\}\subseteq \RR$ defined recursively  as
\begin{equation}\label{eqvk}
\frac{v_k-v_{k-1}}{\tau}=-Cv_k^2, \quad v_0={\teta}_*>0\,,
\end{equation}
where $C$ is the same constant of \eqref{ineqtetak}.
We write now \eqref{eqvk},  adding the term $-\dive(\condu_M(\teta_\tau^k)\nabla v_k)=0$,  in the form
\[
\frac{1}{\tau}\int_\Omega (v_k-v_{k-1}) w\dd x+\int_\Omega
\condu_M(\teta_\tau^k)\nabla v_k\cdot \nabla w\dd x=-C\int_\Omega v_k^2 w\dd x\quad \text{ for every $w
 \in W^{1,2}_+(\Omega)$}.
\]
Subtracting \eqref{ineqtetak} from \eqref{eqvk} and testing the
difference by $w=H_\e(v_k-\teta_k)$, where
\[H_\e(v)=\begin{cases}
0&\quad \hbox{if }v\leq 0\\
v/\e&\quad \hbox{if }v\in (0,\e)\\
1&\quad \hbox{if }v\geq \e
\end{cases}
\]
we obtain, since $v_k<v_{k-1}$ that
\begin{equation}
\label{boh}
\io\left((v_k-v_{k-1})-(\teta^k_\tau-\teta^{k-1}_\tau)\right)H_\e(v_k-\teta^k_\tau) \dd x \leq
0\,.
\end{equation}
Assume now that $\teta_\tau^{k-1}\geq v_{k-1}$  a.e.\ in $\Omega$  (which is true for $k=1$). Taking $\e\searrow 0$, \eqref{boh} yields $\teta_\tau^k\geq v_k$
 a.e.\ in $\Omega$,
 and, by induction,  $\teta^k_\tau\geq v_k>v_{K_\tau}$   a.e.\ in $\Omega$   for every $k=1, \dots, K_{\tau}$.
  We now prove that
there exists $\underline\teta>0$ such  that
 $v_{K_\tau} \geq \underline \teta$ a.e.\ in $\Omega.$ To this aim, observe that $v_{K_\tau}$
 rewrites as
  $v_{K_\tau}=G^{-1}(G(v_{K_\tau}))$, where $G(z):=-\int_z^{v_0}\frac{1}{s^2}\dd s$ is \beruno monotonically \eruno increasing on $(0,v_0]$, $G(0+)=-\infty$, $G(v_0)=0$, hence, by the mean value theorem, for  every $k=1, \ldots, K_\tau$  there exists $s_k\in [v_{k}, v_{k-1}]$ such that
\[
\frac{G(v_k)-G(v_{k-1})}{v_k-v_{k-1}}=G'(s_k)=\frac{1}{s_k^2}\leq \frac{1}{v_k^2},
\]
from which we deduce, using \eqref{eqvk},
\[
\frac{G(v_k)-G(v_{k-1})}{-C\tau v_k^2}\leq \frac{1}{v_k^2}
\Longrightarrow G( v_{K_\tau} )\geq -C\tau K_{\tau}\,,
\]
\beo where the implication is also due to the fact that $G(v_0)=0$. \eo
 Hence, we get
\begin{equation}\label{low-bou-discr}
\teta_\tau^k>v_{K_\tau}=G^{-1}(G(v_{K_\tau}))\geq G^{-1}(-C\tau K_{\tau})=G^{-1}(-CT)=:\underline\teta(T).
\end{equation}
Thus, we conclude \eqref{strict-pos-wk} with $\underline\teta=G^{-1}(-CT)$.

\paragraph{\bf Step $4$: passage to the limit as $M \to \infty$.}
We now pass to the limit in
\eqref{eq-discr-w-TRUNC}--\eqref{eq-discr-chi-TRUNC}
 as $M \to \infty$, for
  $\nu>0$ fixed.
  In this framework, we will denote by $(\teta_M,\uu_M,\chi_M)$ the   solutions of \eqref{eq-discr-w-TRUNC}--\eqref{eq-discr-chi-TRUNC},
  with $(\wtau{k-1},\utau{k-1}, \chitau{k-1})$ given and $\nu>0$ fixed.
 First of all, we derive a bunch of estimates for $(\teta_M,\uu_M, \chi_M)_M, $ holding for constants independent of $M>0$
 (but possibly depending on $\tau>0$,  as well as  on norms of  $(\wtau{k-1},\utau{k-1}, \chitau{k-1})$).

We test \eqref{eq-discr-w-TRUNC} by $1$,  \eqref{eq-discr-u-TRUNC}
\beo by $\frac{\uu_M -\utau{k-1}}{\tau}$,
\eqref{eq-discr-chi-TRUNC} by $\frac{\chi_M - \chitau{k-1}}{\tau}$, \eo
and add the resulting relations. Taking into account all
cancellations, conditions \eqref{bulk-force}--\eqref{datochi}, as
well as the fact that the Yosida approximation
$\widehat{\alpha}_{\nu}$   of $\widehat{\alpha}=I_{(\infty,0]}$  is
a  positive function,
 we obtain that
\begin{equation}
\label{prelim-bound}
\exists\, C >0 \ \forall\, M>0\, : \quad
\| \teta_M\|_{L^1(\Omega)} + \| \uu_M\|_{H^1(\Omega;\R^d)}   +  \nu^{1/\eta}  \| \eps (\uu_M)\|_{L^{\eta}(\Omega;\R^{d\times d})}   + \| \chi_M\|_{W^{1,p}(\Omega)}\leq C.
\end{equation}
\beo We now introduce the notation
\[
\mathcal{S}_M:= \{ x \in \Omega\, : \  \teta_M (x) \leq M \}, \qquad
\mathcal{O}_M:= \Omega \setminus \mathcal{S}_M.
\]
In view of Markov's inequality and of estimate \eqref{prelim-bound},
we have that
\begin{equation}
\label{meas-converg}
|\mathcal{O}_M| \leq \int_{\mathcal{O}_M} \frac{\teta_M}{M} \dd x
\leq \frac1{M} \| \teta_M\|_{L^1(\Omega)} \leq \frac CM \to 0 \text{
as } M \to \infty.
\end{equation}
\eo

We now  test \eqref{eq-discr-w-TRUNC} by $\mathcal{T}_M (\teta_M)$. Observing that
\[
\left.
\begin{array}{ll}
 & \condu_M(\teta_M) \nabla \teta_M \nabla (\mathcal{T}_M (\teta_M)) = \condu (\mathcal{T}_M (\teta_M)) |\nabla (\mathcal{T}_M (\teta_M)))|^2
\\
&
\teta_M \mathcal{T}_M (\teta_M) \geq  |\mathcal{T}_M (\teta_M) |^2
\end{array}
\right\}
\quad \text{a.e.\ in } \Omega,
\]
we get
\begin{equation}
\label{calc-prelim1}
\begin{aligned}
\frac1\tau \int_\Omega  |\mathcal{T}_M (\teta_M) |^2\dd x + \int_\Omega   \condu (\mathcal{T}_M (\teta_M)) |\nabla (\mathcal{T}_M (\teta_M)))|^2
\dd x
 & \leq \int_\Omega |\gtau{k} + \wtau{k-1}| | \mathcal{T}_M (\teta_M) |  \dd x + \int_{\partial\Omega} \htau{k}  | \mathcal{T}_M (\teta_M) |  \dd S
\\ & \quad + \int_\Omega  |\elltaum{k}|  |\mathcal{T}_M (\teta_M)|^2    \dd x  + \int_\Omega | \jtaum{k}|  |\mathcal{T}_M (\teta_M)| \dd x
\end{aligned}
\end{equation}
with the place-holders
\begin{align}\no
\elltaum{k}:= &-  \frac{\chi_M -\chitau{k-1}}\tau -  \rho\dive\left(\frac{\uu_M-\utau{k-1}}{\tau}\right),\,\\
\no
\jtaum{k}:= &a( \chitau{k-1} )\eps\left(\frac{\uu_M-\utau{k-1}}{\tau}\right)\vism\eps\left(\frac{\uu_M-\utau{k-1}}{\tau}\right)+\left|\frac{\chi_M
-\chitau{k-1}}{\tau}\right|^2 +   \frac{\tau^{1/2}}{2} \left|\frac{\chi_M
-\chitau{k-1}}{\tau}\right|^2 .
\end{align}
We now deal with the second term on  the left-hand side of
\eqref{calc-prelim1}   in the same way as in the proof of
\cite[Thm.\ 2]{rocca-rossi-deg} (see also \cite[Rmk.\
2.10]{rocca-rossi-deg} and \cite{gmrs}). In fact, combining the
growth condition \eqref{hyp-K} on $\condu$ with  the Poincar\'e
inequality \eqref{poincare-type}, and taking into account estimate
\eqref{prelim-bound}, we deduce that
\begin{equation}
\label{2nd-label-added}
\begin{aligned}
  \exists\, c,\, C>0 \ \ \forall\, M>0 \, :  \ \
  &  \int_\Omega
\condu (\mathcal{T}_M (\teta_M)) |\nabla (\mathcal{T}_M
(\teta_M)))|^2 \dd x
 \\
 & \quad
 \geq c  \|\nabla (\mathcal{T}_M (\teta_M)))\|_{L^2(\Omega;\R^d)}^2  + \|   \mathcal{T}_M (\teta_M)\|_{L^{3\kappa +6}(\Omega)}^{\kappa +2}- C.
 \end{aligned}
\end{equation}
Let us now consider the terms on he right-hand side of
\eqref{calc-prelim1}. We have
\[
\begin{aligned}
\int_\Omega  |\elltaum{k}|  |\mathcal{T}_M (\teta_M)|^2    \dd x &\leq \| \elltaum{k}\|_{L^2(\Omega)} \|  \mathcal{T}_M (\teta_M)\|_{L^3(\Omega)}
 \|  \mathcal{T}_M (\teta_M)\|_{L^6(\Omega)}\\
  & \leq \frac{c}4  \|\nabla (\mathcal{T}_M (\teta_M)))\|_{L^2(\Omega;\R^d)}^2   + C  \|  \mathcal{T}_M (\teta_M)\|_{L^3(\Omega)}^2
 \\
 &
 \leq \frac{c}2  \|\nabla (\mathcal{T}_M (\teta_M)))\|_{L^2(\Omega;\R^d)}^2   + C'    \|  \mathcal{T}_M (\teta_M)\|_{L^1(\Omega)}^2,
\end{aligned}
\]
where  we have used that $\sup_M \| \elltaum{k}\|_{L^2(\Omega)} \leq
C$ thanks to \eqref{prelim-bound}. The last inequality   with
different constant $C'$
 follows from
the fact  that $H^1(\Omega) \Subset L^3(\Omega) \subset
L^1(\Omega)$, yielding that for all $\rho>0$ there exists $C_\rho>0$
such that $  \|  \mathcal{T}_M (\teta_M)\|_{L^3(\Omega)} \leq \rho
\|  \mathcal{T}_M (\teta_M)\|_{H^1(\Omega)} + C_\rho  \|
\mathcal{T}_M (\teta_M)\|_{L^1(\Omega)} $. In the same way, estimate
\eqref{prelim-bound} ensures that $\| \jtaum{k}
\|_{L^2(\Omega)} \leq C$, whence
\[
\int_\Omega | \jtaum{k}|  |\mathcal{T}_M (\teta_M)| \dd x \leq  C\|\mathcal{T}_M (\teta_M)\|_{L^2(\Omega)}.
\]
 All in all, from
\eqref{calc-prelim1}, taking into account \eqref{prelim-bound} and conditions \eqref{heat-source} and \eqref{dato-h} on $g$ and $h$,
we deduce that
\begin{equation}
\label{second-bound} \exists\, C>0 \ \ \forall\, M>0 \, : \qquad
\|\mathcal{T}_M (\teta_M)\|_{H^1(\Omega)} + \|   \mathcal{T}_M
(\teta_M)\|_{L^{3\kappa +6}(\Omega)}\leq C,
\end{equation}
where the bound for $\|   \mathcal{T}_M (\teta_M)\|_{L^{3\kappa
+6}(\Omega)}$ is due to \eqref{2nd-label-added}.

Let us finally  test \eqref{eq-discr-w-TRUNC} by $\teta_M$.  We rely
on the coercivity  \eqref{ellipticity-retained} of
$\mathcal{A}_M^k$, \beo  and on the previously obtained estimates
\eqref{prelim-bound} and \eqref{second-bound}, and
 we  use  essentially the same arguments as for
 treating \eqref{calc-prelim1}, estimating the terms $\elltaum{k}$ and $\jtaum{k}$ by means of
 \eqref{prelim-bound}. This leads to \eo
 \begin{equation}
\label{est-M-tetaM}
 \sup_{M>0}
 \left(
 \|
 \teta_M \|_{H^1(\Omega)} + \|\teta_M \|_{L^{3\kappa +6}(\mathcal{S}_M)} \right)\leq
 C,
 \end{equation}
\beo  cf.\ \eqref{2nd-label-added} for the bound on $ \|\teta_M \|_{L^{3\kappa
 +6}(\mathcal{S}_M)}$. \eo

In the end, it remains to estimate the terms  $\alpha_{\nu} ((\chi_M -\chitau{k-1})/\tau)$,  $A_p(\chi_M)$ and $\xi_M$ in
\eqref{eq-discr-chi}. First of all, we may suppose that the terms $A_p (\chitau{k-1})$, $\xitau{k-1}\in \beta(\chitau{k-1})$ from the previous step are
bounded in $L^2(\Omega)$ by a constant independent of $M$. Then, we test \eqref{eq-discr-chi} by $(A_p(\chi_M) -A_p(\chitau{k-1} )
+ (\xi_M -\xitau{k-1})) $, thus obtaining
\[
\begin{aligned}
&
\int_\Omega \lambda_{M} (A_p(\chi_M) -A_p(\chitau{k-1})
+ \xi_M -\xitau{k-1}) \dd x +
\| A_p(\chi_M) +\xi_M \|_{L^2(\Omega)}^2
\\&
 =
 \int_\Omega (A_p(\chi_M) +\xi_M) (A_p(\chitau{k-1}) + \xitau{k-1}) \dd x+
 \int_\Omega \mu_{M}  (A_p(\chi_M) -A_p(\chitau{k-1})
+ \xi_M -\xitau{k-1}) \dd x  \doteq I_1 +I_2.
\end{aligned}
\]
Here, we have used the place-holders $\lambda_M:= (\chi_M - \chitau{k-1})/\tau
 + \sqrt{\tau}   (\chi_M - \chitau{k-1})/\tau +  \alpha_{\nu}  ((\chi_M -\chitau{k-1})/\tau)$ and $\mu_M:= \teta_M  -
 b'(\chi_M)  \frac{\eps(\utau{k-1})\elm
\eps(\utau{k-1})}2 - \gamma(\chi_M)   -  \nu (\chi_M)^{\eta-2}\eta  $. With monotonicity arguments, we see that the first integral on the left-hand side is positive. We estimate
\[
I_1 \leq \frac12 \| A_p(\chi_M) +\xi_M \|_{L^2(\Omega)}^2  + \frac12\| A_p(\chitau{k-1}) + \xitau{k-1} \|_{L^2(\Omega)}^2\,.
\]
It follows from the estimates on $\utau{k-1},\, \chitau{k-1}$, from
\eqref{prelim-bound} for $\chi_M$,   and from \eqref{est-M-tetaM}
for $\teta_M$  that $\| \mu_M\|_{L^2(\Omega)} \leq C$ for a constant
independent of $M>0$. Therefore we have
\[
I_2 \leq \frac14  \| A_p(\chi_M)+\xi_M \|_{L^2(\Omega)}^2+ \frac14
\| A_p(\chitau{k-1}) + \xitau{k-1} \|_{L^2(\Omega)}^2  +C\,.
\]
  With this, we conclude that
  $ \| A_p(\chi_M) +\xi_M \|_{L^2(\Omega)} \leq C$ for a constant independent of $M$. By the monotonicity of the operator $\beta$
    \berdue (cf.\ \cite[Lemma 3.3]{akagi}), \erdue  we find
  $ \| A_p(\chi_M)\|_{L^2(\Omega)} \leq C$ and $\|\xi_M \|_{L^2(\Omega)} \leq C$. Then, a comparison argument in \eqref{eq-discr-chi} yields
 \begin{equation}
\label{comparison-estim-chim}
\mu \left\| \alpha_{\nu} \left( \frac{\chi_M -\chitau{k-1}}{\tau}  \right) \right\|_{L^2(\Omega)}
+\| A_p(\chi_M) \|_{L^2(\Omega)} + \|\xi_M\|_{L^2(\Omega)} \leq C.
\end{equation}

Standard compactness arguments together with \eqref{est-M-tetaM} imply that
there exists $\teta \in H^1(\Omega)$ such that,
up to a (not relabeled) subsequence,
\begin{equation}
\label{H1weak}
\teta_M \weakto \teta \ \text{ in } H^1(\Omega),
\qquad \teta_M \to \teta \ \text{ in } L^{q} (\Omega)  \text{ for all }
q< \begin{cases}
\infty & \text{if } d=2,
\\
6 & \text{if } d=3.
\end{cases}
\end{equation}
In particular, $\teta_M \to \teta$ in measure. Combine this with
\eqref{meas-converg} we infer that $ \mathcal{T}_M (\teta_M) \to
\teta$ in measure. Therefore,  in view of estimate
\eqref{second-bound} and of the Egorov theorem we ultimately have
that
\begin{equation}
\label{convergences-troncata}
\teta \in L^{3\kappa +6}(\Omega), \quad  \mathcal{T}_M (\teta_M) \weakto \teta \text{ in } H^1(\Omega) \cap L^{3\kappa +6}(\Omega),
 \quad \mathcal{T}_M (\teta_M) \to \teta  \text{ in } L^q(\Omega) \text{ for all } 1\leq q< 3\kappa +6.
\end{equation}
Therefore, taking into account the growth condition \eqref{hyp-K} for $\condu$, we have
\[
\condu_M (\teta_M) = \condu ( \mathcal{T}_M (\teta_M))  \to \condu(\teta)  \ \text{in } L^q(\Omega) \text{ for all } 1\leq q< 3 +\frac{6}\kappa.
\]
We combine this with the fact that $\nabla \teta_M \weakto \nabla
\teta $ in  $L^2(\Omega;\RR^d)$. On the one hand, we  infer
 that  for some sufficiently big $s>0$,
$\mathcal{A}_M^k (\teta_M)$  weakly converges
 in the space $W^{1,s}(\Omega)^*$
 to the operator  $\widetilde{\mathcal{A}}^k(\teta)$ defined by
 $\pairing{}{W^{1,s}(\Omega)}{\widetilde{\mathcal{A}}^k(\teta)}{v}: = \int_\Omega \condu (\teta) \nabla \teta \nabla v \dd x - \int_{\partial\Omega} h_\tau^k v \dd x $
 for all $v \in W^{1,s}(\Omega)$.
 On the other hand,  a comparison in
 \eqref{eq-discr-w-TRUNC} shows that
 $(\mathcal{A}_M^k (\teta_M) )_M$ is bounded in $H^1(\Omega)^*$.
 Therefore, it is not difficult to infer that  the operator  $\widetilde{\mathcal{A}}^k(\teta)$  extends to $H^1(\Omega)$
and coincides with the operator $\mathcal{A}^k$ from \eqref{spaces},
and that
 \begin{equation}
\label{convergences-condu} \mathcal{A}_M^k (\teta_M)\weakto
\mathcal{A}^k (\teta) \quad \text{in } H^1(\Omega)^* \quad \text{as
} M \to \infty.
\end{equation}
This allows us to pass to the limit in the elliptic operator in
\eqref{eq-discr-w-TRUNC}. Let us now comment the limit passage in
the other nonlinear terms featuring in
\eqref{eq-discr-w-TRUNC}--\eqref{eq-discr-chi-TRUNC}.

From estimates \eqref{prelim-bound} and \eqref{comparison-estim-chim} we also deduce that
there exist $\uu,\, \chi, \, \xi $ and, if $\mu=1$, $ \zeta $  such that, up to a subsequence,
$\uu_M \weakto \uu$ in  $W_0^{1,\eta} (\Omega;\R^d)$,   $\chi_M\to \chi $ in $W^{1,p}(\Omega)$ (this follows from the fact that
$(\chi_M)_M$ is bounded in $W^{1+\sigma,p }(\Omega)$ for all $0 <\sigma < \frac1p$ by  \cite[Thm.\ 2, Rmk.\ 2.5]{savare98}),
$\xi_M \weakto \xi $ in $L^2(\Omega)$, and, if $\mu=1$,
$   \alpha_{\nu} ( (\chi_M -\chitau{k-1})/\tau )  \weakto  \zeta $ in $L^2(\Omega)$.
  By the strong-weak closedness in the sense of graphs of $\alpha_\nu$ (viewed as a maximal monotone graph in
$L^2(\Omega ) \times L^2(\Omega ) $), we infer, in the case $\mu=1$,  that $\zeta =\alpha_\nu  ( (\chitau{k} -\chitau{k-1})/\tau )$
a.e.\ in $\Omega$. Analogously,  the  strong-weak closedness property of  $\beta$  yields that $\xi \in \beta(\chi)$.

Combining  the above convergences with
\eqref{convergences-troncata}--\eqref{convergences-condu} we
conclude that the functions $\teta, \, \uu,\, \chi, \, \xi, \zeta $
fulfill  a.e. in $\Omega$
\[
 \frac{\chi -\chitau{k-1}}{\tau}   + \sqrt{\tau} \frac{\chi -\chitau{k-1}}{\tau} + \mu \alpha_\nu  ( (\chi -\chitau{k-1})/\tau )  + A_p(\chi) + \xi+\gamma(\chi)   + \nu |\chi|^{\eta-2}\chi   = -
 b'(\chi) \frac{\eps(\utau{k-1})\elm
\eps(\utau{k-1})}2 + \teta
\]
as well as
\begin{align}
& \label{eq-discr-w-TRUNC-half} \frac{\teta -\wtau{k-1}}{\tau} +
\frac{\chi -\chitau{k-1}}{\tau}
\teta  +
\rho\dive\left(\frac{\uu-\utau{k-1}}{\tau}\right) \teta
+  \mathcal{A}^k(\teta)
\\
& \qquad \qquad
   = \gtau{k}
+ a( \chitau{k-1}  )\Lambda_k+\left|\frac{\chi
-\chitau{k-1}}{\tau}\right|^2 + \frac{\tau^{1/2}}{2} \left|\frac{\chi
-\chitau{k-1}}{\tau}\right|^2 \quad \text{in $H^1(\Omega)^*$,}
\\
& \label{eq-discr-u-TRUNC-half} \frac{\uu -2\utau{k-1} +
\utau{k-2}}{\tau^2} + \opj{ a(\chitau{k-1} )}{\frac{\uu-\utau{k-1}}{\tau}} +
\oph{b(\chi)}{\uu}  + \ciro(\teta ) - \nu  \mathrm{div}(\Gamma_k)
= \ftau{k}
  \quad \text{in }  W_0^{1,\eta}(\Omega;\R^d)^*,
\end{align}
where $\Lambda_k$ denotes the weak limit of $\eps\left(\frac{\uu_M-\utau{k-1}}{\tau}\right)\vism\eps\left(\frac{\uu_M-\utau{k-1}}{\tau}\right)$ in
$L^2(\Omega)$, and $\Gamma_k$  stands for the weak limit of $|\eps(\uu_M)|^{\eta-2} \mathbb{I} \eps(\uu_M )$ in $L^{\eta/(\eta-1)}(\Omega;\R^d)$. In order to identify them, it is sufficient to test
\eqref{eq-discr-u-TRUNC} by $\uu_M$ and show that
\begin{equation}
\label{limsup-arg} \begin{aligned} \limsup_{M\to \infty}
\pairing{}{W^{1,\eta}( \Omega;\R^d)}{-\mathrm{div}
(|\eps(\uu_M)|^{\eta-2} \mathbb{I} \eps(\uu_M ))}{\uu_M} &  =
\limsup_{M\to \infty} \int_\Omega |\eps(\uu_M)|^\eta \dd x \\ & \leq
\pairing{}{W^{1,\eta}( \Omega;\R^d)}{-\mathrm{div} (\Gamma_k)}{\uu},
\end{aligned}
\end{equation}
which we  can do, exploiting that $\uu$ solves
\eqref{eq-discr-u-TRUNC-half}. This enables us to conclude that
$\Gamma_k = -\mathrm{div} (|\eps(\uu)|^{\eta-2} \mathbb{I} \eps(\uu
))$ and that $\uu_M \to \uu$ \emph{strongly} in
$W^{1,\eta}(\Omega;\R^d)$. The latter convergence clearly allows us
to  conclude that $ \Lambda_k  =
\eps\left(\frac{\uu-\utau{k-1}}{\tau}\right)\vism\eps\left(\frac{\uu-\utau{k-1}}{\tau}\right)$.
All in all,  the triple
 $(\teta,\uu,\chi)$ solves the  system
\begin{align}
& \label{eq-discr-w-TRUNC-approx} \frac{\teta -\wtau{k-1}}{\tau} +
\frac{\chi -\chitau{k-1}}{\tau}
\teta  +
\rho\dive\left(\frac{\uu-\utau{k-1}}{\tau}\right)\teta
+  \mathcal{A}^k(\teta)   = \gtau{k} \\
\no
&\qquad\qquad\qquad\qquad+
a( \chitau{k-1}  )\eps\left(\frac{\uu-\utau{k-1}}{\tau}\right)\vism\eps\left(\frac{\uu-\utau{k-1}}{\tau}\right)+\left( 1+ \frac{\tau^{1/2}}{2}\right)\left|\frac{\chi
-\chitau{k-1}}{\tau}\right|^2\quad \text{in $H^1(\Omega)^*$,}
\\
& \label{eq-discr-u-TRUNC-approx}
\begin{aligned}
&
\frac{\uu -2\utau{k-1} +
\utau{k-2}}{\tau^2} + \opj{ a(\chitau{k-1} )}{\frac{\uu-\utau{k-1}}{\tau}} +
\oph{b(\chi)}{\uu}  + \ciro(\teta ) - \nu  \mathrm{div}(|\eps(\uu)|^{\eta-2} \mathbb{I} \eps(\uu))
= \ftau{k} \\
& \hspace{12cm} \text{in } W_0^{1,\eta}(\Omega;\R^d)^*,
\end{aligned}
\\
& \label{eq-discr-chi-TRUNC-approx}
\begin{aligned}
 & (1+\sqrt \tau)\frac{\chi -\chitau{k-1}}{\tau}  + \mu \alpha_{\nu}   \left( \frac{\chi -\chitau{k-1}}{\tau}  \right)
+ A_p(\chi) + \xi +\gamma(\chi)   +  \nu | \chi|^{\eta-2}\chi \ni -
b'(\chi  )\frac{\eps(\utau{k-1})\elm
\eps(\utau{k-1})}2 + \teta \\
&\hspace{12cm}\aein\, \Omega\,,
  \end{aligned}
\end{align}
with  $\xi \in \beta(\chi)$  a.e.\ in $\Omega$.
 It follows from Step $3$ and convergences
\eqref{H1weak} that
$\teta$ also fulfills the strict positivity property  \eqref{strict-pos-wk}.
 \paragraph{\bf Step $5$: passage to the  limit as $\nu \to 0$.}
We now pass to the limit in
\eqref{eq-discr-w-TRUNC-approx}--\eqref{eq-discr-chi-TRUNC-approx}
 as $\nu \to 0$.  We denote
   by $(\teta_\nu,\uu_\nu,\chi_\nu)$ the   solutions of \eqref{eq-discr-w-TRUNC-approx}--\eqref{eq-discr-chi-TRUNC-approx}
   and, as before, obtain a series of estimates independent of the parameter $\nu$.

First, we test \eqref{eq-discr-w-TRUNC-approx} by $1$,
\eqref{eq-discr-u-TRUNC-approx}  \beo by $\frac{\uu_\nu
-\utau{k-1}}{\tau}$, \eqref{eq-discr-chi-TRUNC-approx} by
$\frac{\chi_\nu - \chitau{k-1}}{\tau}$, \eo   and add the resulting
relations. We thus conclude that
\begin{equation}
\label{prelim-bound-nu}
\exists\, C >0 \ \forall\, \nu>0\, : \quad
\| \teta_\nu\|_{L^1(\Omega)} + \| \uu_\nu\|_{H^1(\Omega;\R^d)}   +  \nu^{1/\eta}  \| \eps (\uu_\nu)\|_{L^{\eta}(\Omega;\R^d)}   + \| \chi_\nu\|_{W^{1,p}(\Omega)}\leq C.
\end{equation}

Second, we test \eqref{eq-discr-w-TRUNC-approx} by  $\teta_\nu^{\alpha-1}$, with $\alpha \in (0,1)$.
With the very same calculations as for the \emph{Second a priori estimates}, cf.\ also the proof of Prop.\
\ref{prop:discrete-aprio} ahead, we conclude that
(cf.\ \eqref{calc2.1}) that
\[
c\int_\Omega \condu(\teta_\nu) |\nabla \teta_\nu^{\alpha/2}|^2 \dd x + c\int_\Omega \left| \eps\left(  \frac{\uu_\nu-\utau{k-1}}{\tau} \right) \right|^2
\teta_\nu^{\alpha-1}\dd x + c \int_\Omega \left| \frac{\chi_\nu - \chitau{k-1}}\tau \right|^2
\teta_\nu^{\alpha-1}\dd x \leq C + C\int_\Omega \teta_\nu^{\alpha+1} \dd x
\]
whence, with the same arguments as throughout
\eqref{stima-aux1}--\eqref{additional-info}, we arrive at
$\int_\Omega |\nabla \teta_\nu^{(\kappa+\alpha)/2}|^2 \dd x \leq C$
for a constant independent of $\nu$. Then, choosing $\alpha \in
 (1/2, 1)$ such that     $\kappa+\alpha \geq 2$,
 we conclude that
\begin{equation}
\label{H1-discre-teta}
\| \teta_\nu\|_{H^1(\Omega)} \leq C
\end{equation}
and, again arguing via the nonlinear Poincar\'e inequality
\eqref{poincare-type}, we also have that
\begin{equation}
\label{K-tetanu}
\| \teta_\nu^{(\kappa+\alpha)/2} \|_{H^1(\Omega)} \leq C\,.
\end{equation}

We then test \eqref{eq-discr-chi-TRUNC-approx} by $(A_p(\chi_\nu ) - A_p (\chitau{k-1}) + \xi_\nu - \xitau{k-1})$ and, arguing in the very same way as in Step $4$, conclude that
 \begin{equation}
\label{comparison-estim-chinu}
\mu \left\| \alpha_{\nu} \left( \frac{\chi_\nu -\chitau{k-1}}{\tau}  \right) \right\|_{L^2(\Omega)}
+\| A_p(\chi_\nu) \|_{L^2(\Omega)} + \|\xi_\nu\|_{L^2(\Omega)} \leq C.
\end{equation}

We can now pass to the limit in system
\eqref{eq-discr-w-TRUNC-approx} --\eqref{eq-discr-chi-TRUNC-approx}
as $\nu \down 0$.  It follows from the previously proved a priori
estimates  and from the same arguments as in Step $4$  that,
along a (not relabeled) subsequence, $\uu_\nu \weakto \uu$ in
$H_0^1(\Omega;\R^d)$, $\chi_\nu \to \chi$ in $W^{1,p}(\Omega)$, and
$\teta_\nu \weakto \teta$ in $H^1(\Omega)$. Using these
convergences, it is not difficult to pass to the limit in
\eqref{eq-discr-u-TRUNC-approx} and conclude that $\uu$ fulfills
\eqref{eq-discr-u},  \beo with  test functions in
$W_0^{1,\eta}(\Omega;\R^d)$. We  then conclude \eqref{eq-discr-u}
with test functions in $H_0^1(\Omega;\R^d)$ by a density argument. \eo

 With the same
argument as in Step $4$  (cf.\ \eqref{limsup-arg}), testing
\eqref{eq-discr-u-TRUNC-approx} by $\uu_\nu$ we conclude that
\[
\limsup_{\nu \to 0} \int_\Omega \eps(\uu_\nu) \elm \eps(\uu_\nu)  \dd x \leq \int_\Omega \eps(\uu) \elm \eps(\uu)  \dd x,
\]
yielding that $\uu_\nu \to \uu$ \emph{strongly} in $H^1 (\Omega;\R^d)$. Therefore,
\begin{equation}
\label{strong-L1-unu}
 a(\chitau{k-1}  )\eps\left(\frac{\uu_\nu-\utau{k-1}}{\tau}\right)\vism\eps\left(\frac{\uu_\nu-\utau{k-1}}{\tau}\right) \to
  a(\chitau{k-1}  )\eps\left(\frac{\uu-\utau{k-1}}{\tau}\right)\vism\eps\left(\frac{\uu-\utau{k-1}}{\tau}\right) \qquad \qquad \text{in } L^1(\Omega).
\end{equation}
We use this information to pass to the limit in  the heat
equation
 \eqref{eq-discr-w-TRUNC-approx}. Moreover,
 estimate \eqref{K-tetanu} allows us to conclude that, up to a subsequence,
$\teta_\nu^{(\kappa+\alpha)/2} \weakto \teta^{(\kappa+\alpha)/2} $ in $H^1(\Omega)$, hence
$\teta_\nu^{(\kappa+\alpha)/2} \to \teta^{(\kappa+\alpha)/2}$ in $L^{6-\epsilon}(\Omega)$ for all $\epsilon>0$, whence, taking into account the growth condition on $\condu$, that
\[
\condu(\teta_\nu) \to \condu(\teta) \qquad \text{in } 
 L^\gamma(\Omega)  \quad \text{with } \gamma =
 \frac{(6-\epsilon)(\kappa+\alpha)}{2\kappa}  \quad \text{for all }
\epsilon>0.
\]
This allows us to pass to the limit in the term $\condu(\teta_\nu)
\nabla \teta_\nu$, tested  against   $v \in W^{1,s}(\Omega)$ for
some sufficiently big $s>0$.  All in all, we infer that
$(\teta,\uu,\chi)$ satisfies  \eqref{eq-discr-w} in some dual space
$W^{1,s}(\Omega)^*$, such that, also, $W^{1,s}(\Omega) \subset \
 L^\infty(\Omega) $ in accord with the $L^1$-convergence
\eqref{strong-L1-unu}. Finally, we pass to the limit in  the
flow rule  \eqref{eq-discr-chi-TRUNC-approx}. Due to estimate
\eqref{comparison-estim-chinu}, we have that there exist $\xi \in
L^2(\Omega)$ and, if $\mu=1$, $\zeta \in L^2(\Omega)$ such that
\[
 \alpha_{\nu} \left( \frac{\chi_\nu -\chitau{k-1}}{\tau}  \right)  \weakto \zeta, \qquad \xi_\nu \weakto \xi \qquad \text{in } L^2(\Omega).
\]
The strong-weak closedness of $\beta$ yields that $\xi \in
\beta(\chi)$ a.e.\ in $\Omega$. In order to conclude that,
in the case $\mu=1$, $\zeta \in \alpha ((\chi-\chitau{k-1})/\tau)$
a.e.\ in $\Omega$, we show that
\[
\limsup_{\nu \down 0} \int_\Omega  \alpha_{\nu} \left( \frac{\chi_\nu -\chitau{k-1}}{\tau}  \right) \left( \frac{\chi_\nu -\chitau{k-1}}{\tau} \right) \dd x
\leq \int_\Omega \zeta \left(  \frac{\chi -\chitau{k-1}}{\tau}\right) \dd x
\]
and invoke well-knows results from the theory of maximal monotone operators.

All in all, we infer that $(\teta,\uu,\chi)$ solves system
\eqref{eq-discr-w}--\eqref{eq-discr-chi}, where  the heat
equation \eqref{eq-discr-w} is to be understood in
$W^{1,s}(\Omega)^*$.
\paragraph{\bf Step $6$: $H^2(\Omega;\R^d)$-regularity for $\utau{k}$ and conclusion.} \beo We
follows the steps of the regularity argument in the proof of
\cite[Lemma 4.1]{HR}
 proceed
by induction and suppose that $\uu_\tau^{k-1} \in H_{\mathrm{Dir}}^2
(\Omega;\R^d)$.   First of all, we rewrite  the discrete momentum
equation  \eqref{eq-discr-u} in the following form
            \begin{align*}
                \int_\Omega \big(\tau a(\chi_\tau^{k-1})\vism+\tau^2 b(\chi_\tau^k)\elm\big)\varepsilon(\uu_\tau^k):\varepsilon(\zeta)\dd x
                =\int_\Omega {\bf h}_\tau^k\cdot\zeta\dd x\,,
            \end{align*}
            where $\zeta\in H_0^1(\Omega;\R^d)$ and the right-hand side
            (note that  $\uu_\tau^{k-1}\in  H_{\mathrm{Dir}}^2(\Omega;\R^d)$) is defined as
            $$
                {\bf h}_\tau^k:=-\uu_\tau^k-\tau a(\chi_\tau^{k-1})\vism\varepsilon(\uu_\tau^{k-1})-
                 \tau^2
                \ciro(\teta_\tau^k)+2\uu_\tau^{k-1}-\uu_\tau^{k-2} +  \tau^2 \mathbf{f}_\tau^k
                \in L^2(\Omega;\R^d).
            $$
            Condition \eqref{eqn:visc} shows
            \begin{align}
            \label{eqn:ellipticEq}
                \int_\Omega \big(\tau a(\chi_\tau^{k-1})\gamma+\tau^2 b(\chi_\tau^k)\big)\elm\varepsilon(\uu_\tau^k):\varepsilon(\zeta)\dd x
                =\int_\Omega {\bf h}_\tau^k\cdot\zeta\dd x.
            \end{align}
            Since the coefficient function $\tau a(\chi_\tau^{k-1})\gamma+\tau^2 b(\chi_\tau^k)\in W^{1,p}(\Omega)$ in \eqref{eqn:ellipticEq}
            is scalar-valued and bounded from below by a positive constant (see \eqref{data-a}), we get
            $
                \left(\tau a(\chi_\tau^{k-1})\gamma+\tau^2 b(\chi_\tau^k)\right)^{-1}\in W^{1,p}(\Omega).
            $
            Testing \eqref{eqn:ellipticEq} with $\zeta=\left(\tau a(\chi_\tau^{k-1})\gamma+\tau^2 b(\chi_\tau^k)\right)^{-1}\varphi$ where $\varphi\in H_0^1(\Omega;\R^d)$ is another test-function yields
            \begin{align}
            \label{eqn:ellipticEqTrans}
                \int_\Omega \elm\varepsilon(\uu_\tau^k):\varepsilon(\varphi)\dd x
                =\int_\Omega \widehat {\bf h}_\tau^k\cdot\varphi\dd x
            \end{align}
            with the new right-hand side
            \begin{align}
            \label{eqn:hatG}
                \widehat {\bf h}_\tau^k:=\frac{1}{\tau a(\chi_\tau^{k-1})\gamma+\tau^2 b(\chi_\tau^k)}{\bf h}_\tau^k
                    +\elm\varepsilon(\uu_\tau^k)\cdot\frac{\tau a'(\chi_\tau^{k-1})\gamma+\tau^2 b'(\chi_\tau^k)}{\tau a(\chi_\tau^{k-1})\gamma+\tau^2 b(\chi_\tau^k)}\nabla\chi_\tau^k\,.
            \end{align}
            Since $\nabla\chi_\tau^k\in L^p(\Omega;\R^d)$ and $\varepsilon(\uu_\tau^k)\in L^2(\Omega;\R^{d\times d})$, we get
            $\widehat {\bf h}_\tau^k\in L^{2p/(2+p)}(\Omega;\R^d)$.
            Now, we can refer to   the proof of \cite[Lemma 4.1]{HR}
              where an iteration argument
             leads to
$\widehat {\bf h}_\tau^k\in L^2(\Omega;\R^d)$. Then, the regularity
result \cite[Lemma 3.2]{necas} yields   $\uu_\tau^k\in
H_{\mathrm{Dir}}^2(\Omega;\R^d)$  as desired. \eo

 In the end, exploiting that $\utau{k} \in H^2(\Omega;\R^d)$,   a comparison argument in the heat equation allows us to conclude that
$\int_\Omega \condu(\teta) \nabla \teta \cdot \nabla v \dd x$ is well defined for all test functions $v \in H^1(\Omega)$, hence \eqref{eq-discr-w}
is solved in $H^1(\Omega)^*$.
\end{proof}
\begin{remark}
\label{rmk:still-pos}
\upshape
In the case $\mu=1$, as mentioned  in Remark \ref{rmk:discrete-features}, the discrete $\chi$-equation could be decoupled from the discrete equations for $\teta$ and $\uu$, cf.\ \eqref{chi-discre-irrv}. This would lead to having the term $\frac{\chitau k - \chitau{k-1}}{\tau} \wtau{k-1}$. The argument for the strict positivity of $\wtau k$ in Step $3$ in this case would not go through. Nonetheless, it would be possible to prove that $\wtau k \geq 0$ a.e.\ in $\Omega$, by testing the discrete heat equation by $-(\wtau k)^-$, and using that
$
\int_\Omega \frac{\chitau k - \chitau{k-1}}{\tau} \wtau{k-1} (-(\wtau k)^-) \dd x \geq 0
$
since $\chitau k \leq \chitau{k-1}$ a.e.\ in $\Omega$.
\end{remark}
\begin{remark}
\label{rmk:2parameters}
\upshape
 We briefly comment on
the reason why we need to perform two distinct passages to the limit in the proof of Lemma \ref{lemma:ex-discr}.
As the above proof shows, in the passage to limit as $\nu \to 0$ we \beruno loose \eruno the information that the right-hand side of the
equation for
$\teta$ is  estimated in  $L^2(\Omega)$. Hence, we need to carry out refined estimates on the $\teta$-equation (i.e., testing it by $\teta^{\alpha-1}$), where we fully exploit the growth of $\condu$ to carry out the related calculations.
Clearly, to do so we first have to pass to the limit with the truncation parameter.
\end{remark}

\subsection{Approximate entropy and total energy inequalities}
\label{ss:3.1bis}
Preliminarily, we establish the
\begin{notation}[Interpolants and discrete integration-by-parts formula]
\upshape
 Hereafter, for a given Banach space $B$ and a
$K_\tau$-tuple $( \mathfrak{h}_\tau^k )_{\beruno k=0\eruno}^{K_\tau}
\subset B$, we shall use the short-hand notation
\[
\dtau{k}{\mathfrak{h}}:= \frac{\mathfrak{h}_\tau^k-\mathfrak{h}_\tau^{k-1}}{\tau}, \qquad
\duetau{k}{\mathfrak{h}}:= \dtau{k}{\dtau{k}{\mathfrak{h}}}= \frac{\mathfrak{h}_\tau^k -2
\mathfrak{h}_\tau^{k-1} + \mathfrak{h}_\tau^{k-2}}{\tau^2}.
\]
We recall the well-known \emph{discrete by-part integration} formula
for all  \beruno $\{\mathfrak{h}_\tau^k \}_{\beruno
k=0\eruno}^{K_\tau} \subset B,\, \{ v_\tau^k \}_{\beruno
k=0\eruno}^{K_\tau} \subset B^*$ \eruno \beruno
\begin{equation}
\label{discr-by-part} \sum_{k=1}^{K_\tau} \tau
\pairing{}{B}{\vtau{k}}{\dtau{k}{\mathfrak{h}}} =
\pairing{}{B}{\vtau{K_\tau}}{\btau{K_\tau}}
-\pairing{}{B}{\vtau{\beruno0\eruno}}{\btau{0}} -\sum_{\beruno
k=1\eruno}^{K_\tau}\tau\pairing{}{B}{\dtau{k}{v} }{ \btau{k-1}}\,.
\end{equation}
\eruno

 We  introduce
 the left-continuous and  right-continuous piecewise constant, and the piecewise linear interpolants
 of the values  $\{ \mathfrak{h}_\tau^k
\}_{k=1}^{K_\tau}$ by
\[
\left.
\begin{array}{llll}
& \pwc  {\mathfrak{h}}{\tau}: (0,T) \to B  & \text{defined by}  &
\pwc {\mathfrak{h}}{\tau}(t): = \mathfrak{h}_\tau^k,
\\
& \upwc  {\mathfrak{h}}{\tau}: (0,T) \to B  & \text{defined by}  &
\upwc {\mathfrak{h}}{\tau}(t) := \mathfrak{h}_\tau^{k-1},
\\
 &
\pwl  {\mathfrak{h}}{\tau}: (0,T) \to B  & \text{defined by} &
 \pwl {\mathfrak{h}}{\tau}(t):
=\frac{t-t_\tau^{k-1}}{\tau} \mathfrak{h}_\tau^k +
\frac{t_\tau^k-t}{\tau}\mathfrak{h}_\tau^{k-1}
\end{array}
\right\}
 \qquad \text{for $t \in
(t_\tau^{k-1}, t_\tau^k]$.}
\]
We also introduce the piecewise linear interpolant    of the values
$\{ (\mathfrak{h}_\tau^k - \mathfrak{h}_{\tau}^{k-1})/\tau\}_{k=1}^{K_\tau}$ (namely, the
values taken by  the -piecewise constant- function $\pwl
{\mathfrak{h}}{\tau}'$), viz.
\[
\pwwll  {\mathfrak{h}}{\tau}: (0,T) \to B \qquad \pwwll
{\mathfrak{h}}{\tau}(t) :=\frac{(t-t_\tau^{k-1})}{\tau}
\frac{\mathfrak{h}_\tau^k - \mathfrak{h}_{\tau}^{k-1}}{\tau} +
\frac{(t_\tau^k-t)}{\tau} \frac{\mathfrak{h}_\tau^{k-1} -
\mathfrak{h}_{\tau}^{k-2}}{\tau} \qquad \text{for $t \in
(t_\tau^{k-1}, t_\tau^k]$.}
\]
Note that $ {\pwwll  {\mathfrak{h}}{\tau}}'(t) =\duetau{k}{\mathfrak{h}}$ for $t \in
(t_\tau^{k-1}, t_\tau^k]$.

Furthermore, we   denote by  $\pwc{\mathsf{t}}{\tau}$ and by
$\upwc{\mathsf{t}}{\tau}$ the left-continuous and right-continuous
piecewise constant interpolants associated with the partition, i.e.
 $\pwc{\mathsf{t}}{\tau}(t) := t_\tau^k$ if $t_\tau^{k-1}<t \leq t_\tau^k $
and $\upwc{\mathsf{t}}{\tau}(t):= t_\tau^{k-1}$ if $t_\tau^{k-1}
\leq t < t_\tau^k $. Clearly, for every $t \in [0,T]$ we have
$\pwc{\mathsf{t}}{\tau}(t) \downarrow t$ and
$\upwc{\mathsf{t}}{\tau}(t) \uparrow t$ as $\tau\to 0$.
\end{notation}

In view of \eqref{bulk-force}, \eqref{heat-source}, and
\eqref{dato-h}, it is easy to check that the piecewise constant
interpolants $ (\pwc {\mathbf{f}}{\tau} )_\tau$, $(\pwc  g{\tau}
)_{\tau}$, $(\pwc  h{\tau} )_{\tau}$  of the values $\ftau{k}$,
$\gtau{k}$,  $\htau{k}$ \eqref{local-means} fulfill as $\tau \down
0$
\begin{align}
 \label{converg-interp-f}  & \pwc {\mathbf{f}}{\tau}  \to \mathbf{f}
  \text{ in $L^2(0,T;L^2(\Omega;\R^d))$,}
 \\
\label{converg-interp-g}  & \pwc g{\tau}  \to g
  \text{ in $L^1(0,T;L^1(\Omega))\cap L^2(0,T;H^1(\Omega)')$.}
  \\
  \label{converg-interp-h}  & \pwc h{\tau}  \to h
  \text{ in $L^1(0,T;L^2(\partial\Omega))$.}
\end{align}

We now rewrite the discrete equations
\eqref{eq-discr-w}--\eqref{eq-discr-chi} in terms of the interpolants
 $\pwc\teta\tau,$  $\pwl\teta\tau,$   $\pwc\uu\tau,$
$\upwc\uu\tau,$ $\pwl\uu\tau,$ $\pwwll\uu\tau,$ $\pwc\chi\tau,$
$\upwc\chi\tau,$  $\pwl\chi\tau,$
   $\pwc \xi\tau$, and $\pwc \zeta\tau$
of the elements $(\wtau{k},\utau{k},\chitau{k},\xitau{k},\zetau{k})_{k=1}^{K_\tau}$.
Indeed, we have for almost all $t\in (0,T)$
\begin{align}
&
\label{eq-w-interp}
\begin{aligned}
\partial_t \pwl \teta{\tau}(t) +\partial_t \pwl \chi{\tau}(t) \pwc{\teta}\tau (t)
 & +
\rho\dive(\partial_t \pwl \uu{\tau}(t) ) \pwc{\teta}\tau (t)
+   \mathcal{A}^{\frac{\bar{\mathsf{t}}_\tau(t)}{\tau}}( \pwc{\teta}\tau (t) )   = \pwc g{\tau}(t)+
\\
& \qquad + a( \upwc \chi\tau(t) )\eps\left(\partial_t \pwl
\uu{\tau}(t) \right)\vism\eps\left(\partial_t \pwl \uu{\tau}(t)
\right)+\left(1+\frac{\tau^{1/2}}2 \right)\left|\partial_t \pwl
\chi{\tau}(t)\right|^2\quad \text{in $B^*$,}
\end{aligned}
\\
&
 \label{eq-u-interp}
 \begin{aligned}
 \partial_t\pwwll {\uu}{\tau}(t)+
\opj{ a(\upwc\chi\tau (t)) }{\partial_t \pwl \uu\tau(t)} +
\oph{b(\pwc\chi\tau (t))}{\pwc \uu{\tau}(t)}   +
\ciro(\pwc\teta{\tau}) 
  & = \pwc
{\mathbf{f}}{\tau}(t) \\
&\quad  \aein\, \Omega,
\end{aligned}\\
& \label{eq-chi-interp}
\begin{aligned}
(1+\sqrt{\tau}) \partial_t \pwl
{\chi}{\tau}(t) +\mu \pwc\zeta\tau(t) + A_p\pwc\chi\tau(t) +
\pwc\xi\tau(t)+&\gamma(\pwc\chi\tau(t))  
\\ &  = -
b'(\pwc\chi\tau(t) )\frac{\eps(\upwc\uu\tau(t))\elm
\eps(\upwc\uu\tau(t))}2 + \pwc\teta\tau(t)\\
&\qquad\aein\, \Omega, 
  \end{aligned}
\end{align}
with  $\pwc \xi\tau \in \beta(\pwc \chi \tau)$ and
$\pwc\zeta\tau \in \partial I_{(-\infty,0]}(\partial_t \pwl
{\chi}{\tau})$ a.e.\ in $\Omega \times (0,T)$.

Our next result states that  the  interpolants of suitable discrete solutions to system \eqref{eq-discr-w}--\eqref{eq-discr-chi}
also satisfy the approximate versions of the
entropy inequality
\eqref{entropy-ineq}
and of the
 total \beruno energy \eruno inequality \eqref{total-enid}. 

 For stating the discrete entropy inequality  \eqref{entropy-ineq-discr} below, we need to introduce  \emph{discrete} test functions.
 Namely, with every 
 test function $\varphi \in \mathrm{C}^0 ([0,T]; W^{1,d+\epsilon}(\Omega)) \cap H^1 (0,T; L^{6/5}(\Omega)) $
 we associate
 \begin{equation}
 \label{discrete-tests-phi}
\text{for } k=1, \ldots, K_\tau \quad  \varphi_\tau^k:= \varphi(t_\tau^k) 
 \end{equation}
and consider the piecewise constant and  linear interpolants
$\pwc \varphi\tau$ and $\pwl \varphi\tau$ of the values
$(\varphi_\tau^k)_{k=1}^{K_\tau}$.
It can be shown that  the following convergences hold as $\tau \to 0$
\begin{equation}
\label{convergences-test-interpolants}
\pwc \varphi\tau  \to \varphi \quad
\text{ in } L^\infty (0,T; W^{1,d+\epsilon}(\Omega))
\text{ and }  \quad \partial_t \pwl \varphi\tau  \to \partial_t \varphi \quad \text{ in } L^2 (0,T; L^{6/5}(\Omega)).
\end{equation}
Then, \eqref{entropy-ineq-discr} is obtained by testing
\eqref{eq-discr-w} by $\frac{\varphi_\tau^k}{\wtau{k}}$, for
$k=1,\ldots,K_\tau$.

As for the total energy inequality \eqref{total-enid-discr} below, let us mention that it results  from our carefully designed time-discretization scheme,
 observing in addition that \eqref{eq-discr-chi} is indeed the Euler-Lagrange equation for a suitable minimum problem, cf.\ \eqref{min-prob-chi} below,
where the additional term
\[
\frac{\tau^{3/2}}{2}\int_\Omega \left| \frac{\chitau k - \chitau{k-1}}{\tau}\right|^2 \dd x
\]
has the role to ``compensate'' for the possible non-convexity of $\int_\Omega \widehat{\gamma}(\chi) \dd x $.
Therefore, to get the discrete total energy inequality \eqref{total-enid-discr}
 we have  added the term $\frac{\tau^{1/2}}2 \left| \frac{\chitau k - \chitau{k-1}}{\tau}\right|^2 $
 to the right-hand side of \eqref{eq-discr-w}. This will lead to the necessary cancellations,
 cf.\ \eqref{cancellation} below.
 \begin{proposition}[Discrete entropy and total energy inequalities, $\mu \in \{0,1\}$]
\label{prop:discr-enid}
 Under \textbf{Hypotheses (I)--(III)},
 for $\tau>0$ sufficiently small,
  the  discrete solutions
 $(\wtau{k},  \utau{k}, \chitau{k} )_{k=1}^{K_\tau}$ to Problem \ref{prob:rhoneq0}
 fulfill
 \begin{itemize}
 \item[-]
  the \emph{discrete} entropy inequality \beruno
 \begin{equation}\label{entropy-ineq-discr}
\begin{aligned}
& \int_{\pwc{\mathsf{t}}{\tau}(s)}^{\pwc{\mathsf{t}}{\tau}(t)}
\int_\Omega (\log(\upwc\teta\tau (r)) + \upwc\chi\tau (r))
\partial_t \pwl \varphi\tau (\tau) \dd x \dd r + \rho
\int_{\pwc{\mathsf{t}}{\tau}(s)}^{\pwc{\mathsf{t}}{\tau}(t)}
\int_\Omega \dive(\partial_t \pwl\uu\tau(r)) \pwc\varphi\tau(r)  \dd
x \dd r
 \\ & \qquad - \int_{\pwc{\mathsf{t}}{\tau}(s)}^{\pwc{\mathsf{t}}{\tau}(t)}  \int_\Omega  \condu(\pwc\teta\tau (r)) \nabla \log(\pwc\teta\tau(r)) \cdot \nabla \pwc\varphi\tau(r)  \dd x \dd r
\\ &
\leq \int_\Omega (\log(\pwc\teta\tau(t))
+ \pwc\chi\tau (t)) 
\pwc\varphi\tau(t) \dd x - \int_\Omega
(\log(\pwc\teta\tau(s)) 
+ \pwc\chi\tau(s)) 
\pwc\varphi\tau(s) \dd x
\\
& \qquad
 -  \int_{\pwc{\mathsf{t}}{\tau}(s)}^{\pwc{\mathsf{t}}{\tau}(t)}  \int_\Omega \condu(\pwc \teta\tau(r)) \frac{\pwc\varphi\tau(r)}{\pwc \teta\tau(r)}
\nabla \log(\pwc \teta\tau(r)) \cdot \nabla \pwc \teta\tau (r) \dd x \dd r
\\ & \qquad
 -
  \int_{\pwc{\mathsf{t}}{\tau}(s)}^{\pwc{\mathsf{t}}{\tau}(t)}   \int_\Omega \left( \pwc g{\tau} (r) +a(\upwc\chi\tau(r)) \eps(\partial_t \pwl \uu\tau(r))\vism \eps(\partial_t \pwl \uu\tau(r))
  + |\partial_t \pwl \chi\tau(r)|^2 + \frac{\tau^{3/2}}{2}  |\partial_t \pwl \chi\tau(r)|^2  \right)
 \frac{\pwc\varphi\tau(r)}{\pwc\teta\tau(r)} \dd x \dd r \\
 &\qquad -  \int_{\pwc{\mathsf{t}}{\tau}(s)}^{\pwc{\mathsf{t}}{\tau}(t)}   \int_{\partial\Omega} \pwc h{\tau} (r)   \frac{\pwc\varphi\tau(r)}{\pwc\teta\tau(r)} \dd S \dd r
\end{aligned}
\end{equation}
for all $0 \leq s \leq t \leq T$ and
for all $\varphi \in \mathrm{C}^0 ([0,T]; W^{1,d+\epsilon}(\Omega)) \cap H^1 (0,T; L^{6/5}(\Omega)) $ with
$\varphi \geq 0$;
\item[-] the \emph{discrete} total energy inequality for all $ 0 \leq s \leq t
\leq  T$, viz.
\begin{equation}
\label{total-enid-discr}
\begin{aligned}
 \mathscr{E}(\pwc\teta\tau(t),\pwc\uu\tau(t), \partial_t \pwl\uu\tau(t), \pwc\chi\tau(t))\leq
 \mathscr{E}(\pwc\teta\tau(s),\pwc\uu\tau(s), \partial_t \pwl\uu\tau (s),\pwc \chi\tau(s))  &  + \int_{\pwc{\mathsf{t}}{\tau}(s)}^{\pwc{\mathsf{t}}{\tau}(t)}
\int_\Omega  (\pwc g\tau + \pwc{\mathbf{f}}\tau \cdot
\partial_t \pwl\uu\tau) \dd x \dd r
\\
& \quad
+ \int_{\pwc{\mathsf{t}}{\tau}(s)}^{\pwc{\mathsf{t}}{\tau}(t)}
 \int_{\partial\Omega} \pwc h\tau  \dd S \dd r \,,
 \end{aligned}
\end{equation}
 with $\mathscr{E}$ from \eqref{total-energy}.\eruno
 \end{itemize}
 \end{proposition}

 For the proof  of the discrete entropy inequality,  we will  rely on a crucial inequality satisfied by any \underline{concave}
 (differentiable)
  function $\psi: \mathrm{dom}(\psi) \to \R,$ i.e.
 \begin{equation}
 \label{inequality-concave-functions}
 \psi(x) - \psi(y) \leq \psi'(y) (x-y) \qquad \text{for all } x,\, y \in \mathrm{dom}(\psi).
 \end{equation}
\begin{proof}
We split the proof in two steps.
 \paragraph{\bf Step $1$: proof of the total energy inequality.}
 Let us consider the minimum problem
  \begin{equation}
 \label{min-prob-chi}
 \begin{aligned}
 \min_{\chi \in W^{1,p}(\Omega)} \Big\{
 \int_\Omega \Big(
 \frac{\tau^{3/2}}2   \left|\frac{\chi - \chitau{k-1}}\tau\right|^2
 + \left( \frac{\chitau k -\chitau {k-1}}{\tau}\right)\chi &   + \mu   \widehat{\alpha} \left( \frac{\chi - \chitau{k-1}}{\tau} \right)+ \frac{|\nabla \chi|^p}p
 +\widehat{\beta}(\chi)
 \\ &
 +\widehat{\gamma}(\chi) 
 + b(\chi) \frac{\eps(\utau{k-1}) \elm \eps(\utau{k-1}) }2 - \wtau{k} \chi
 \Big) \dd x
\Big \}\,
\end{aligned}
\end{equation}
where $\chitau k$ and $\wtau{k}$ are the discrete solutions
from Lemma \ref{lemma:ex-discr} and $\utau{k-1}, \, \chitau{k-1}$
are given from the previous step,
 and let $\lambda>0$ such that $\widehat{\gamma}'' \geq -\lambda$ as in \eqref{lambda-convex}. Then, the function
 \begin{equation}
 \label{strict-convex}
  r \mapsto \widehat{\gamma}(r) + \lambda |r|^2 \qquad \text{is strictly convex.}
 \end{equation}
 Let $\bar\tau>0$ such that $\tfrac{1}{2\tau} >\lambda$ for all $0<\tau \leq \bar \tau$.
Adding and subtracting $\int_\Omega\lambda
|\chi-\chitau{k-1}|^2 \dd x$,
 we may rewrite the minimum problem \eqref{min-prob-chi} as
  \begin{equation}
 \label{min-prob-chi-2}
 \begin{aligned}
 \min_{\chi \in W^{1,p}(\Omega)} \Big\{
 \int_\Omega \Big(    &\left(\frac{1}{2\sqrt{\tau}}-\lambda\right)  |\chi - \chitau{k-1}|^2
 +   \left( \frac{\chitau k -\chitau {k-1}}{\tau}\right)\chi +  \mu \widehat{\alpha} \left( \frac{\chi - \chitau{k-1}}{\tau} \right)+ \frac{|\nabla \chi|^p}p
 +\widehat{\beta}(\chi)   +\widehat{\gamma}(\chi)
 \\
 &
+ \lambda |\chi|^2
 + b(\chi) \frac{\eps(\utau{k-1}) \elm \eps(\utau{k-1}) }{2} - \wtau{k} \chi  + \lambda |\chitau{k-1}|^2 + 2\lambda \chi \chitau{k-1}
 \Big) \dd x
\Big \}.
\end{aligned}
\end{equation}
 Observe that the Euler-Lagrange equation for \eqref{min-prob-chi-2} is exactly \eqref{eq-discr-chi}.
Using  the convexity of $\widehat \alpha$, $\widehat \beta$, $b$, and the $\lambda$-convexity of $\widehat\gamma$
(whence \eqref{strict-convex}),
 it is not difficult to check that  \eqref{eq-discr-chi} has a unique solution.
 We may thus conclude that the minimum problem \eqref{min-prob-chi-2} has a unique solution,
 which coincides with  the discrete solution $\chitau{k}$ from Lemma \ref{lemma:ex-discr}.


Now, choosing $\chitau{k-1}$ as a competitor for $\chitau k$ in the minimum problem \eqref{min-prob-chi} yields
 \begin{equation}
\label{to-be-quoted-later}
\begin{aligned}
   &\tau \int_\Omega \left| \frac{\chitau k - \chitau{k-1}}{\tau}\right|^2 \dd x
   +   \int_\Omega \frac{\tau^{3/2}}2   \left|\frac{\chitau k - \chitau{k-1}}\tau\right|^2\dd x
    + \mu \int_\Omega
 \widehat{\alpha} \left(\frac{\chitau k - \chitau{k-1}}{\tau} \right) \dd x
  + \io\frac{|\nabla\chitau{k}|^p}{p} \dd x + \int_{\Omega}  \widehat{\beta}(\chitau{k}) \dd
x
\\ & \hspace{5cm}
+ \int_{\Omega}  \widehat{\gamma}(\chitau{k}) \dd
x  + \int_\Omega b(\chitau{k}) \frac{\eps(\utau{k-1}) \elm \eps(\utau{k-1}) }2 \dd x  -\int_\Omega \wtau{k} \chitau k  \dd x
\\
 & \leq \io\frac{|\nabla\chitau{k-1}|^p}{p}\dd x + \int_{\Omega}
\widehat{\beta}(\chitau{k-1}) \dd x
+ \int_{\Omega}  \widehat{\gamma}(\chitau{k-1}) \dd
x
\\
&
\hspace{6cm}
+ \int_\Omega b(\chitau{k-1}) \frac{\eps(\utau{k-1}) \elm \eps(\utau{k-1}) }2 \dd x  -\int_\Omega \wtau{k} \chitau {k-1}  \dd x.
\end{aligned}
\end{equation}

 Hence, we test   \eqref{eq-discr-u}  by
$\utau{k}-\utau{k-1}$
and observe that , for all $k=1,\ldots,K_\tau$,
\begin{align}
\label{est-1.1}
& \tau\int_{\Omega} \duetau{k}{\uu} \cdot
\dtau{k}{\uu}\dd x \geq \frac1{2} \|\dtau{k}{\uu}\|_{L^2(\Omega)}^2
-\frac1{2} \|\dtau{k-1}{\uu}\|_{L^2(\Omega;\R^d)}^2
\\
\label{est-1.2}
&
\pairing{}{H^1(\Omega)}{\opj{a(\chitau{k-1})}{\dtau{k}{\uu}}}{\utau{k}-\utau{k-1}} =
\tau \int_\Omega a(\chitau{k-1})\eps\left(\frac{\utau{k}-\utau{k-1}}{\tau}\right)\vism\eps\left(\frac{\utau{k}-\utau{k-1}}{\tau}\right) \dd x.
\end{align}
Furthermore, we have
\begin{equation}
\label{est-1.3}
\begin{aligned}
\pairing{}{H^1(\Omega;\R^d)}{ \oph{b(\chitau{k})
}{\utau{k}} } {\frac{\utau{k}-\utau{k-1}}\tau  }&\geq
\frac12 \int_\Omega  b(\chitau{k})\eps(\utau{k})
\elm \eps(\utau{k})\dd x
-\frac12 \int_\Omega  b(\chitau{k})\eps(\utau{k-1})
\elm \eps(\utau{k-1})\dd x
\\
&
= \frac12 \int_\Omega  b(\chitau{k})\eps(\utau{k})
\elm \eps(\utau{k})\dd x
-\frac12 \int_\Omega  b(\chitau{k-1})\eps(\utau{k-1})
\elm \eps(\utau{k-1})\dd x\\
&\quad
-\frac12\int_\Omega (b(\chitau{k}) - b(\chitau{k-1})) \eps(\utau{k-1})
\elm \eps(\utau{k-1})\dd x\,.
\end{aligned}
\end{equation}
Finally,
\begin{equation}
\label{est-1.4}
\tau \pairing{}{H^1(\Omega;\R^d)}{\mathcal{C}_\rho (\wtau{k})}{\frac{\utau{k}-\utau{k-1}}\tau  }
=-\rho\int_\Omega \wtau{k} \mathrm{div}\left( \frac{\utau{k}-\utau{k-1}}\tau \right) \dd x\,.
\end{equation}

 Next, we multiply
  \eqref{eq-discr-w} by $\tau$ and integrate over $\Omega$.
We add the
resulting relation to   the equation obtained testing
\eqref{eq-discr-u-TRUNC}  by $\utau{k}-\utau{k-1}$ and to
\eqref{to-be-quoted-later}.
 The terms
 \begin{equation}
 \label{cancellation}
 \begin{array}{lll}
 & \displaystyle
 \tau\int_{\Omega}\dtau{k}{\chi}
\wtau{k}\dd x, \qquad    & \displaystyle  \rho \tau \int_\Omega
\wtau{k} \dive(\dtau{k}{\uu})\dd x,
\\
 & \displaystyle
\tau \int_\Omega a(\chitau{k-1} )\eps\left(\frac{\utau{k}-\utau{k-1}}{\tau}\right)\vism\eps\left(\frac{\utau{k}-\utau{k-1}}{\tau}\right) \dd x,
 & \displaystyle
 \tau \int_\Omega \left| \frac{\chitau k -\chitau {k-1}}{\tau}\right|^2 \dd x
 \\
 & \displaystyle
 \frac{\tau^{3/2}}2 \int_\Omega \left| \frac{\chitau k -\chitau {k-1}}{\tau}\right|^2 \dd x,
&\displaystyle  \frac12\int_\Omega (b(\chitau{k}) - b(\chitau{k-1})) \eps(\utau{k-1})
\elm \eps(\utau{k-1})\dd x
 \end{array}
 \end{equation}
 cancel out.


We sum over the index
$k=m,\ldots,j$,
for any couple of indexes $1\leq m <j\leq K_\tau$.
 Taking into account \eqref{to-be-quoted-later}--\eqref{est-1.4}, we  ultimately obtain
\begin{equation}
\label{est-2.1}
\begin{aligned}
& \int_\Omega \left(  \wtau{j}
 + \frac1{2} |\dtau{j}{\uu}|^2
+ \frac12 b(\chitau{j}) \eps(\utau{j}) \elm \eps(\utau{j})
 + \frac{|\nabla\chitau{j}|^p}{p} +\widehat{\beta}(\chitau{j})
+  \widehat{\gamma} (\chitau{j}) \right)
 \dd x
 \\
&  \leq \int_\Omega \left(  \wtau{m}   + \frac1{2} |\dtau{m}{\uu}|^2
+ \frac12
b(\chitau m)\eps(\utau m)
\elm   \eps(\utau m)
 +\frac{|\nabla\chitau m|^p}{p}
+\widehat{\beta}(\chitau m) +  \widehat{\gamma}(\chitau m) \right)
\dd x
\\
& \qquad
+\sum_{k=m}^j \tau
  \left( \int_\Omega\left(   \gtau{k} + \ftau{k} \cdot   \dtau{k}{\uu}  \right) \dd x + \int_{\partial \Omega}  \htau{k}  \dd S\right),
\end{aligned}
\end{equation}
which yields \eqref{total-enid-discr}.
\paragraph{\bf Step $2$: proof of the entropy inequality.}
Let us fix  an arbitrary \emph{\beruno positive \eruno }  test  function
\[
\varphi \in \mathrm{C}^0 ([0,T]; W^{1,d+\epsilon}(\Omega)) \cap H^1 (0,T; L^{6/5}(\Omega))
\]
 with $(\phitau{k})_{k=1}^{K_\tau}$ defined by \eqref{discrete-tests-phi}.
We multiply \eqref{eq-discr-w} by $\frac{\phitau k}{\wtau k} \in H^1(\Omega)$
(hence, an admissible test function for  \eqref{eq-discr-w})
and integrate over $\Omega$.
We obtain
\begin{equation}
\label{w-concavita}
\begin{aligned}
&
\int_\Omega  \left( \gtau{k} +
a(\chitau{k})\eps\left(\frac{\utau{k}-\utau{k-1}}{\tau}\right)\vism\eps\left(\frac{\utau{k}-\utau{k-1}}{\tau}\right)+\left|\frac{\chitau{k}
-\chitau{k-1}}{\tau}\right|^2  +  \frac{\tau^{1/2}}{2}  \left|\frac{\chitau{k}
-\chitau{k-1}}{\tau}\right|^2 \right) \frac{\phitau k}{\wtau k} \dd  x \\
&\qquad+ \int_{\partial \Omega} \htau{k}\, \frac{\phitau k}{\wtau k} \dd S
\\ & =
\int_\Omega  \left( \frac{\wtau{k} -\wtau{k-1}}{\tau}
+\frac{\chitau{k} -\chitau{k-1}}{\tau}\wtau{k} +
\rho\dive\left(\frac{\utau{k}-\utau{k-1}}{\tau}\right)\wtau{k}
\right) \frac{\varphi_\tau^k}{\wtau{k}}\dd x +\int_\Omega
\condu(\wtau{k}) \nabla \wtau{k} \cdot \nabla \left( \frac{\phitau
k}{\wtau k} \right) \dd x
\\
& \leq \int_\Omega
\left(
\frac{\log(\wtau k) - \log(\wtau {k-1})}{\tau}   +\frac{\chitau{k} -\chitau{k-1}}{\tau}
+
\rho\dive\left(\frac{\utau{k}-\utau{k-1}}{\tau}\right) \right)
 {\phitau k }  \dd x
 \\
 & \qquad
  +\int_\Omega \left( \frac{\condu(\wtau{k}) }{\wtau k}\nabla \wtau{k} \cdot \nabla {\phitau k}   -  \frac{\condu(\wtau{k}) }{|\wtau k|^2} |\nabla \wtau{k}|^2 \phitau {k}
  \right) \dd x
\end{aligned}
\end{equation}
where we have used that (cf.\ \eqref{inequality-concave-functions})
\[
\frac{\wtau k- \wtau{k-1}}{\wtau k} \leq \log(\wtau k) - \log(\wtau {k-1})\qquad \text{ a.e.\ in } \Omega.
\]
Note that this inequality is preserved by the positivity of the
discrete test function $\phitau k$. We now sum \eqref{w-concavita},
multiplied by $\tau$, over $k=m, \ldots, j$, for any couple of
indexes $1\leq m <j\leq K_\tau$. We use the discrete
by-part-integration  formula \eqref{discr-by-part}, yielding \beruno
\[
\begin{aligned}
&
 \sum_{k=m}^{j} \tau \int_\Omega \dtau{k}{\log(\wtau k)} \phitau{k} \dd x =
\int_\Omega \log(\wtau{j})\phitau{j} \dd x-
\int_\Omega\log(\wtau{m})\phitau{m}\dd x -\sum_{k=m}^{j}\tau
\int_\Omega  \log(\wtau{k-1}) \dtau{k}{\varphi} \dd x
\\
& \sum_{k=m}^{j} \tau \int_\Omega \dtau{k}{\chitau k} \phitau{k} \dd
x = \int_\Omega \chitau{j}\phitau{j} \dd x-
\int_\Omega\chitau{m}\phitau{m}\dd x -\sum_{k=m}^{j}\tau \int_\Omega
\chitau{k-1} \dtau{k}{\varphi} \dd x.
\end{aligned}
\]
\eruno
Inserting the two above inequalities in \eqref{w-concavita}
(summed  up over $k=m, \ldots, j$), rearranging terms,
 we conclude
\eqref{entropy-ineq-discr}.
\end{proof}
 \begin{remark}
\label{b-l-convex}
\upshape
A close perusal of the proof of Proposition
\ref{prop:discr-enid} reveals that,
 $b$ is only $\lambda$-convex, in place of convex, it is still possible to prove that the
 discrete equation for $\chi$ \eqref{eq-discr-chi} admits a unique solution, and therefore conclude that $\chitau k$
 is the unique minimizer for \eqref{min-prob-chi}. This, provided we replace the $p$-Laplacian operator
 in \eqref{eq-discr-chi} with its \emph{non-degenerate} version, cf.\ Remark \ref{rmk:uni-rev}.
\end{remark}
\subsection{A priori estimates revisited}
\label{ss:3.2}
The following result collects all the a priori estimates for the
approximate solutions constructed via time discretization.
  In particular, the proof renders on the discrete level the \emph{Second} estimate, which has a \emph{nonlinear}
character and  thus  needs to be suitably translated 
within the frame of the discrete system
\eqref{eq-discr-w}--\eqref{eq-discr-chi}.

\beo
 Instead,  we are not able to render
 the \emph{Sixth  estimate}. The ultimates reason for this is that  this estimate is  based on a comparison argument  in the
 heat equation divided by $\teta$. On the time-discrete level, the  analogue   of the latter
 relation is  somehow
 represented
 by the entropy inequality \eqref{entropy-ineq-discr}. Essentially, since \eqref{entropy-ineq-discr} holds as an \emph{inequality},
 only, we are not able to recover from it the full information provided by the rescaled heat equation.
Nonetheless,  with
careful calculations  we will deduce from   the entropy inequality \eqref{entropy-ineq-discr} the following \emph{weaker version} of  estimate
\eqref{bv-esti-temp}, namely
\begin{equation}
\label{weaker-sixth}
\exists\, S>0  \ \ \forall\, \tau>0 \, : \qquad
\sup_{\varphi \in W^{1,d+\epsilon}(\Omega), \ \| \varphi \|_{W^{1,d+\epsilon}(\Omega)}\leq 1}
\mathrm{Var}(\pairing{}{W^{1,d+\epsilon}(\Omega)}{\log(\pwc\teta\tau)}{\varphi}; [0,T]) \leq S
\end{equation}
for every $\epsilon>0$, where  we have used the notation
\begin{equation}
\label{definition-Var}
\begin{aligned}
&
\mathrm{Var}(\pairing{}{W^{1,d+\epsilon}(\Omega)}{\log(\pwc\teta\tau)}{\varphi}; [0,T]):=
\\
& \qquad
 \sup_{0 =\sigma_0 < \sigma_1 < \ldots < \sigma_J =T }
\sum_{i=1}^J \left | \pairing{}{W^{1,d+\epsilon}(\Omega)} {\log(\pwc\teta\tau (\sigma_{i})) }{\varphi} -
 \pairing{}{W^{1,d+\epsilon}(\Omega)} {\log(\pwc\teta\tau (\sigma_{i-1})) }{\varphi} \right|.
 \end{aligned}
\end{equation}
Thanks to a suitable abstract compactness result proved in the Appendix, Theorem \ref{th:mie-theil}, estimate \eqref{weaker-sixth} turns out to be sufficient to develop the compactness arguments that will allow us to pass to the time-continuous  limit and thus prove Theorems \ref{teor3} and
\ref{teor1}. We postpone to Remark \ref{rmk:after6} some comments on how the  $\mathrm{BV}([0,T];W^{1,d+\epsilon}(\Omega)^*)$-estimate for
$\partial_t \log(\teta)$,  on the time-continuous level, might be recovered.
\eo
\begin{proposition}
\label{prop:discrete-aprio}
Assume  \beo \textbf{Hypotheses (0)--(III)} \eo and \eqref{bulk-force}--\eqref{datochi}. Let $\mu \in \{0,1\}$.
Then,
there exists a constant $S>0$ such that for all $\tau>0$  the following estimates
\begin{subequations}
\label{a-priori}
\begin{align}
& \label{aprio1bis}
\|\pwc \uu{\tau}\|_{L^\infty(0,T;\boY)}
 \leq S,
\\
& \label{aprio1-discr}
\|\pwl \uu{\tau}\|_{H^1(0,T;\boY) \cap
W^{1,\infty}(0,T;\boZ)}
 \leq S,
\\
& \label{aprio2-discr} \|\pwwll \uu{\tau}
\|_{H^1(0,T;L^{2}(\Omega;\R^d))} \leq S,
\\
& \label{aprio3-discr}
\|\pwc
\chi{\tau}\|_{L^\infty(0,T;W^{1,p}(\Omega))} \leq S,
\\
& \label{aprio4-discr} \|\pwl
\chi{\tau}\|_{L^\infty(0,T;W^{1,p}(\Omega)) \cap H^1
(0,T;L^2(\Omega))} \leq S,
\\
&
\label{log-added}
\| \log(\pwc \teta{\tau}) \|_{L^2 (0,T; H^1(\Omega))}
 \leq S,
\\
& \label{aprio6-discr}
 \|\pwc
\teta{\tau}\|_{L^2(0,T; H^1(\Omega)) \cap L^{\infty}(0,T;L^{1}(\Omega))} \leq S,
\end{align}
hold, as well as \underline{estimate \eqref{weaker-sixth}}.  Furthermore, under \textbf{Hypothesis (V)} (i.e.\ if
$\kappa$ from \eqref{hyp-K} fulfills
 $1<\kappa <5/3$), we have in addition
\begin{align}
\label{aprio7-discr}
&
\sup_{\tau>0} \| \pwl \teta{\tau} \|_{\mathrm{BV} ([0,T]; W^{2,d+\epsilon} (\Omega)^*)} \leq S  \quad \text{for all }\epsilon>0.
\end{align}
Finally, if $\mu=0$ we also have
\begin{align}
\label{aprio8-discr}
&
\sup_{\tau>0}\left(  \| \pwc \chi \tau\|_{L^2 (0,T; W^{1+\sigma, p} (\Omega))} +\| \pwc \xi\tau \|_{L^2(0,T; L^2(\Omega))} \right)
  \leq S \quad \text{for all } 1\leq \sigma<\frac1p.
\end{align}
\end{subequations}
\end{proposition}
We now sketch the proof, showing how the formal a priori estimates in Section \ref{s:aprio}
 can be translated in the framework of the time-discretization scheme;  we shall only detail the argument    leading to estimate \eqref{weaker-sixth}.
\begin{proof}
From the discrete total energy
inequality \eqref{total-enid-discr}, arguing in the very same way as for the \textbf{First a priori estimate}, we deduce
\begin{equation}
\label{crucial-first} \| \pwc\teta\tau \|_{L^\infty
(0,T;L^1(\Omega))} + \| \pwl\uu\tau\|_{W^{1,\infty}
(0,T;L^2(\Omega;\R^d))} + \beo \| \nabla\pwc\chi\tau \|_{L^\infty
(0,T;L^{p}(\Omega))}  \eo \leq C.
\end{equation}
We also infer
 that
 $\|b(\pwc\chi\tau)^{1/2} \eps(\pwc \uu\tau)
\|_{L^\infty (0,T;L^2(\Omega;\R^{d\times d}))} \leq C$ which gives, via \eqref{data-a} and Korn's inequality, that
\[
\| \pwc \uu\tau \|_{L^\infty (0,T;H_0^1(\Omega; \R^{d}))} \leq C.
\]

Next,
along the lines of the  \textbf{Second a priori estimate},
we test
\eqref{eq-discr-w} by $F'(\wtau{k}) = (\wtau{k})^{\alpha-1}$, with $\alpha \in (0,1)$.
Since $F(\teta)=\teta^\alpha/\alpha$ is concave, by \eqref{inequality-concave-functions} we have
\[
(\wtau k- \wtau{k-1}) F'({\wtau k} )\leq F(\wtau k) - F(\wtau {k-1})\qquad \text{ a.e.\ in } \Omega,
\]
therefore we obtain
\begin{equation}
\label{F-concavita}
\begin{aligned}
&
\int_\Omega  \left( \gtau{k} +
a(\chitau{k})\eps\left(\frac{\utau{k}-\utau{k-1}}{\tau}\right)\vism\eps\left(\frac{\utau{k}-\utau{k-1}}{\tau}\right)+\left( 1+ \frac{\tau^{1/2}}{2} \right)  \left|\frac{\chitau{k}
-\chitau{k-1}}{\tau}\right|^2
\right) F'(\wtau k) \dd  x
+\int_{\partial \Omega} \htau{k} F'(\wtau k) \dd  S
\\
& \leq \int_\Omega
\left(
\frac{F(\wtau k) - F(\wtau {k-1})}{\tau}   +\frac{\chitau{k} -\chitau{k-1}}{\tau} \wtau{k} F'(\wtau k)
+\rho\dive\left(\frac{\utau{k}-\utau{k-1}}{\tau}\right)  \wtau{k} F'(\wtau k)  +
\condu(\wtau{k}) \nabla \wtau{k} \nabla(F'(\wtau k))
 \right)
  \dd x\,.
\end{aligned}
\end{equation}
Then, we multiply \eqref{F-concavita}
 by $\tau$. Summing over the index $k$ and recalling that $g\geq 0$ and $h\geq 0$, we obtain  for all $t \in (0,T]$
 \[
 \begin{aligned}
 &
 \frac{4(1-\alpha)}{\alpha^2} \int_0^{\pwc{\mathsf{t}}{\tau}(t)} \int_\Omega \condu(\pwc \teta{\tau}) |\nabla ((\pwc \teta{\tau})^{\alpha/2}) |^2 \dd x \dd s\\
&\qquad
+ \int_0^{\pwc{\mathsf{t}}{\tau}(t)}  \int_\Omega   \left( c_2  |\eps (\partial_t \pwl {\uu}{\tau})|^2 F'(\pwc \teta \tau) + \left(1+ \frac{\tau^{1/2}}{2} \right)|\partial_t\pwl {\chi}{\tau}|^2 F'(\pwc \teta \tau)
\right) \dd x \dd s
\\
&
\leq \int_\Omega F(\pwc\teta\tau(t)) \dd x - \int_\Omega F(\teta_0) \dd x  +   \int_0^{\pwc{\mathsf{t}}{\tau}(t)}  \int_\Omega
\left(    \partial_t \pwl {\chi}{\tau} \pwc\teta\tau F'(\pwc\teta\tau)  +\rho
\dive (\partial_t \pwl {\uu}{\tau}) \pwc\teta\tau F'(\pwc\teta\tau) \right)\dd x\dd s\,.
 \end{aligned}
 \]
Starting from this inequality, we develop calculations completely analogous to the ones in Section \ref{s:aprio}  {\bf for the Second a priori estimate}.
In particular,
we conclude that
\begin{equation}
\label{crucial-third}
\int_0^{\pwc{\mathsf{t}}{\tau}(t)} \int_\Omega \condu(\pwc \teta{\tau}) |\nabla ((\pwc \teta{\tau})^{\alpha/2}) |^2 \dd x \dd s  \leq C\,.
\end{equation}
The same calculations as for the \textbf{Third estimate} allow us
then to deduce from \eqref{crucial-third} and \eqref{crucial-first}
estimate \eqref{aprio6-discr}. As a byproduct of these calculations,
we again  have for all  $\alpha \in (1/2,1)$
\begin{equation}
\label{crucial-third-bis}
\| (\pwc \teta{\tau})^{(\kappa-\alpha)/2}\|_{L^2 (0,T; H^1(\Omega))},  \,  \| (\pwc \teta{\tau})^{(\kappa+\alpha)/2}\|_{L^2 (0,T; H^1(\Omega))}  \leq C\,.
\end{equation}
Moreover, since
\begin{equation}
\label{positiv-interp}
 \pwc\teta \tau(t) \geq \underline{\teta}\qquad \text{a.e.\ in } \Omega \quad \text{ for all $t \in [0,T]$,}
 \end{equation}
 (with $\underline{\teta}$ from \eqref{low-bou-discr}),
  we also have  \eqref{log-added}.

As for the \textbf{Fourth estimate}, we subtract from the discrete total energy inequality \eqref{total-enid-discr} the discrete heat equation
\eqref{eq-discr-w} multiplied by $\tau$ and summed over the index $k$. Therefore, we obtain  for all $ t \in [0, T]$
\begin{align}\no
& \frac 12\io |\partial_t \pwl {\uu}{\tau} (\pwc {\mathsf{t}}{\tau}(t))|^2\dd x
+\int_0^{\pwc {\mathsf{t}}{\tau}(t)} \bilj{a(\upwc\chi\tau )}{\partial_t\pwl {\uu}{\tau}}{\partial_t\pwl {\uu}{\tau}}\dd s
+\frac12
\bilh{b(\pwc {\chi}{\tau}(\pwc{\mathsf{t}}{\tau}(t)))}{\pwc{\uu}{\tau}(\pwc{\mathsf{t}}{\tau}(t))}{\pwc{\uu}{\tau}(\pwc{\mathsf{t}}{\tau}(t))}
\\
& \quad
+ \left(1+ \frac{\tau^{1/2}}{2}  \right) \int_0^{\pwc {\mathsf{t}}{\tau}(t)}
\io |\partial_t\pwl \chi\tau|^2\dd x\dd s
+\io \frac1p |\nabla\pwc\chi\tau(\pwc {\mathsf{t}}{\tau}(t))|^p + W(\pwc\chi\tau(\pwc {\mathsf{t}}{\tau}(t)))\dd x
\no
\\
& = I_0 +  \int_0^{\pwc {\mathsf{t}}{\tau}(t)}  \io \pwc\teta\tau  \left(\rho \dive (\partial_t\pwl\uu\tau)+\partial_t\pwl\chi\tau\right)\dd x \dd
s+ \int_0^{\pwc {\mathsf{t}}{\tau}(t)}  \io \pwc{\mathbf{f}}\tau \cdot  \partial_t\pwl\uu\tau \dd x \dd s\,,
\no
\end{align}
where we have used the place-holder   $I_0=
\frac12\bilh{b(\chi_0)}{\uu_0}{\uu_0} +
\io (\frac12 |{\bf v}_0|^2 +
 \frac1p|\nabla\chi_0|^p+ W(\chi_0))\dd x$.
Exploiting
\eqref{bulk-force} and
 estimate  \eqref{aprio6-discr}, we control  the second term on the right-hand side with
 $\int_0^t\int_\Omega |\partial_t\pwl \chi\tau|^2\,\dd x  \dd s $ and
  the second term  on the left-hand side, which bounds  $ \int_0^{\pwc {\mathsf{t}}{\tau}(t)}
  \| \partial_t\pwl {\uu}{\tau}\|_{H^1(\Omega;\R^d)}^2\, \dd s $ thanks to
  \eqref{korn}. Therefore, we conclude that $\|  \partial_t\pwl {\uu}{\tau}\|_{L^2 (0,T;H^1(\Omega;\R^d))} \leq C$, as well as
  estimates \eqref{aprio3-discr}--\eqref{aprio4-discr}.

  The  \textbf{Fifth estimate}  is performed on the time-discretization scheme by testing \eqref{eq-discr-u}
  by $-\dive (\beo\vism\eo\eps (\utau{k} - \utau{k-1}))$. For all the calculations, we refer to \cite[(3.61)--(3.67)]{rocca-rossi-deg}: therein,
  the equation for $\uu$ was the same  as our  own \eqref{eqI}, but the elasticity and viscosity tensors $\elm$ and $\vism$  were assumed
  to be independent of  the space variable $x$. Nonetheless, the computations from \cite{rocca-rossi-deg} carry over to the present setting, cf.\ also the formal calculations for the  {\bf Fourth a priori estimate}  in Sec.\ \ref{s:aprio}. Therefore, we conclude estimates \eqref{aprio1bis} and \eqref{aprio1-discr}.
  A comparison argument in \eqref{eq-discr-u}, joint with \eqref{reg-pavel-b},  yields \eqref{aprio2-discr}.

We will now  render the \emph{weaker version}
\eqref{weaker-sixth}
 of the  \textbf{Sixth estimate} in the time-discrete setting. To do so,
let us fix a partition $0 =\sigma_0 < \sigma_1 < \ldots < \sigma_J =T$ of the interval $[0,T]$. Preliminarily,  from the discrete  entropy inequality
\eqref{entropy-ineq-discr}, written on the  interval $[\sigma_{i-1},\sigma_i]$
 and for a \emph{constant-in-time} test function $\varphi \in W^{1,d+\epsilon}(\Omega)$ for some $\epsilon>0$, we deduce that
 \begin{align}
 &
 \label{ell-i-pos}
 \int_\Omega (\frakh_{i,\tau} - \frakh_{i-1,\tau}) \varphi \dd x + \Lambda_{i,\tau} (\varphi) \geq 0 && \text{for all } \varphi \in W_+^{1,d+\epsilon}(\Omega),
 \\
   &
 \label{ell-i-neg}
 \int_\Omega (\frakh_{i-1,\tau} - \frakh_{i,\tau}) \varphi \dd x - \Lambda_{i,\tau} (\varphi) \geq 0 && \text{for all } \varphi \in W_-^{1,d+\epsilon}(\Omega),
 \end{align}
 where we have used the place-holders
 \[
\begin{aligned}
&
\frakh_{i,\tau}= \log(\pwc\teta\tau (\sigma_{i})) + \pwc\chi\tau (\sigma_{i}),
\\
&
\begin{aligned}
\Lambda_{i,\tau} (\varphi) =
 &   \int_{\pwc{\mathsf{t}}{\tau}(\sigma_{i-1})}^{\pwc{\mathsf{t}}{\tau}(\sigma_i)}  \int_\Omega  \condu(\pwc\teta\tau) \nabla \log(\pwc\teta\tau) \cdot \nabla \varphi \dd x \dd r - \rho  \int_{\pwc{\mathsf{t}}{\tau}(\sigma_{i-1})}^{\pwc{\mathsf{t}}{\tau}(\sigma_i)} \int_\Omega \dive(\partial_t \pwl\uu\tau) \varphi  \dd x \dd r
\\
& \quad -  \int_{\pwc{\mathsf{t}}{\tau}(\sigma_{i-1})}^{\pwc{\mathsf{t}}{\tau}(\sigma_i)} \int_\Omega  \condu(\pwc\teta\tau) \frac{\varphi}{\pwc\teta\tau} \nabla (\log(\pwc\teta\tau)) \nabla \pwc\teta\tau \dd x \dd r -  \int_{\pwc{\mathsf{t}}{\tau}(\sigma_{i-1})}^{\pwc{\mathsf{t}}{\tau}(\sigma_i)}   \int_{\partial\Omega} \pwc h{\tau}  \frac{\varphi}{\pwc\teta\tau} \dd S \dd r
 \\
 &\quad - \int_{\pwc{\mathsf{t}}{\tau}(\sigma_{i-1})}^{\pwc{\mathsf{t}}{\tau}(\sigma_i)} \int_\Omega \left(\pwc g\tau+ a(\upwc \chi\tau)  \eps(\partial_t \pwl \uu\tau)\vism \eps(\partial_t \pwl \uu\tau)
  + \left( 1 +   \frac{\tau^{1/2}}{2}   \right) |\partial_t \pwl \chi\tau|^2  \right)
 \frac{\varphi}{\pwc\teta\tau} \dd x \dd r.
\end{aligned}
\end{aligned}
\]
For later use,  we also introduce the place-holder
\[
\mathcal R_{\tau}:=
    \rho\dive(\partial_t \pwl\uu\tau)   + \condu(\pwc\teta\tau)|\nabla (\log(\pwc\teta\tau)) |^2
+\left(\pwc g\tau +a(\upwc \chi\tau)  \eps(\partial_t \pwl \uu\tau)\vism \eps(\partial_t \pwl \uu\tau)
  + \left( 1 +  \frac{\tau^{1/2}}{2}   \right) |\partial_t \pwl \chi\tau|^2 \right) \frac1{\pwc \teta\tau},
  \]
  so that $\Lambda_{i,\tau}(\varphi)$ rewrites as
  \begin{equation}
\label{resto-Lambda_i}
\Lambda_{i,\tau}(\varphi) = \int_{\pwc{\mathsf{t}}{\tau}(\sigma_{i-1})}^{\pwc{\mathsf{t}}{\tau}(\sigma_i)}  \int_\Omega \left(  \condu(\pwc\teta\tau) \nabla \log(\pwc\teta\tau) \cdot \nabla \varphi  - \mathcal{R}_\tau \varphi \right) \dd x \dd r -  \int_{\pwc{\mathsf{t}}{\tau}(\sigma_{i-1})}^{\pwc{\mathsf{t}}{\tau}(\sigma_i)}   \int_{\partial\Omega} \pwc h{\tau}  \frac{\varphi}{\pwc\teta\tau} \dd S \dd r\,.
\end{equation}

\beo
 We now estimate   the total variation $\mathrm{Var}(\pairing{}{W^{1,d+\epsilon}(\Omega)}{\log(\pwc\teta\tau) + \pwc\chi\tau}{\varphi}; [0,T])$
 cf.\ \eqref{definition-Var}, for
 $\varphi \in W^{1,d+\epsilon} (\Omega)$ with $\| \varphi\|_{W^{1,d+\epsilon} (\Omega)}\leq 1$,
by proceeding as follows. We observe that  for every fixed $\varphi \in W^{1,d+\epsilon}(\Omega)$
there holds
\begin{equation}
\label{neg-pos-part}
\begin{aligned}
  \left| \pairing{}{W^{1,d+\epsilon}(\Omega)}{\frakh_{i,\tau} - \frakh_{i-1,\tau}}{\varphi} \right|
 &  \leq
\left| \int_\Omega (\frakh_{i,\tau} - \frakh_{i-1,\tau}) \varphi^+\dd x + \Lambda_{i,\tau}(\varphi^+)  \right| +   |\Lambda_{i,\tau} (\varphi^+)|
 \\
 &  \quad  +  \left| \int_\Omega ( \frakh_{i-1,\tau} - \frakh_{i,\tau} ) (-\varphi^-) \dd x - \Lambda_{i,\tau} (-\varphi^-)  \right| + | \Lambda_{i,\tau} (\varphi^-) |
 \\
 &
= \int_\Omega (\frakh_{i,\tau} - \frakh_{i-1,\tau}) |\varphi| \dd x + \Lambda_{i,\tau} (|\varphi|) +   |\Lambda_{i,\tau} (\varphi^+)|  + | \Lambda_{i,\tau} (\varphi^-) |,
\end{aligned}
\end{equation}
where
$\varphi^+$ ($\varphi^-$, resp.) denotes the positive (negative) part of $\varphi$. The last equality ensues from   \eqref{ell-i-pos}--\eqref{ell-i-neg}, allowing us  to remove the absolute values in the first and second lines of \eqref{neg-pos-part}, and from the linearity of the map $\varphi \mapsto \Lambda_{i,\tau}(\varphi)$, yielding
$ \Lambda_{i,\tau}(\varphi^+) - \Lambda_{i,\tau} (-\varphi^-)   = \Lambda_{i,\tau} (|\varphi|)  $.
Therefore,
\begin{equation}
\label{chain-1}
\begin{aligned}
&
\sum_{i=1}^J  
\left| \pairing{}{W^{1,d+\epsilon}(\Omega)}{\frakh_{i,\tau} - \frakh_{i-1,\tau}}{\varphi}\right|
\\
& \stackrel{(1)}\leq  \sum_{i=1}^J   
\int_\Omega (\frakh_{i,\tau} - \frakh_{i-1,\tau}) |\varphi| \dd x
+ \Lambda_{i,\tau} (|\varphi|) + |\Lambda_{i,\tau} (\varphi^+)| + |\Lambda_{i,\tau} (\varphi^-)|.
\end{aligned}
\end{equation}
Next,   rewriting $\Lambda_i (|\varphi|)$ by means of
 \eqref{resto-Lambda_i} we
 find that
\begin{equation}
\label{chain-1.1}
\begin{aligned}
\Lambda_i (|\varphi|) =
 & \sum_{i=1}^J  \int_{\pwc{\mathsf{t}}{\tau}(\sigma_{i-1})}^{\pwc{\mathsf{t}}{\tau}(\sigma_i)}  \int_\Omega
\condu(\pwc\teta\tau) \nabla \log(\pwc\teta\tau) \cdot \nabla (|\varphi|) \dd x \dd r
\\
&\quad
-
\sum_{i=1}^J
\int_{\pwc{\mathsf{t}}{\tau}(\sigma_{i-1})}^{\pwc{\mathsf{t}}{\tau}(\sigma_i)}   \int_{\partial\Omega} \pwc h{\tau}  \frac{|\varphi|}{\pwc\teta\tau} \dd S \dd r
-
\sum_{i=1}^J
\int_{\pwc{\mathsf{t}}{\tau}(\sigma_{i-1})}^{\pwc{\mathsf{t}}{\tau}(\sigma_i)}
\int_\Omega  \mathcal{R}_\tau  |\varphi|\dd x   \dd r
\doteq I_1 -I_2 -I_3.
\end{aligned}
\end{equation}
We observe that (due to {\bf Hypothesis (I)})
 \begin{equation}
 \label{chain-4}
 \begin{aligned}
 &
\left| I_1 \right| \leq  \sum_{i=1}^J    \sup_{\|\varphi \|_{ W^{1,d+\epsilon}(\Omega)}  \leq     1  }  \left|  \int_{\pwc{\mathsf{t}}{\tau}(\sigma_{i-1})}^{\pwc{\mathsf{t}}{\tau}(\sigma_i)}  \int_\Omega
\condu(\pwc\teta\tau) \nabla \log(\pwc\teta\tau) \cdot \nabla (|\varphi|) \dd x \dd r \right|
\\
 & \leq   \sum_{i=1}^J    \sup_{\|\varphi \|_{ W^{1,d+\epsilon}(\Omega)}  \leq     1  } \| \varphi\|_{W^{1,3}(\Omega)} \int_{\pwc{\mathsf{t}}{\tau}(\sigma_{i-1})}^{\pwc{\mathsf{t}}{\tau}(\sigma_i)}   \| (\pwc\teta\tau)^{(\kappa
+\alpha-2)/2} \nabla \pwc\teta\tau \|_{L^2(\Omega;\R^d)}  \| (\pwc\teta\tau)^{(\kappa
-\alpha)/2} \|_{L^{6}(\Omega)} \dd s
 \\ & \leq C \int_0^T    \| (\pwc\teta\tau)^{(\kappa
+\alpha-2)/2} \nabla \pwc\teta\tau \|_{L^2(\Omega;\R^d)}   \| (\pwc\teta\tau)^{(\kappa
-\alpha)/2} \|_{L^{6}(\Omega)} \dd s,
\end{aligned}
\end{equation}
while
we note that $-I_2 \leq 0$ by the positivity of $h$.
Moreover, taking into account the definition of $\mathcal{R}_\tau$,
the fact that $|1/\pwc\teta\tau|\leq C$  a.e.\ in $\Omega \times (0,T)$  by  \eqref{positiv-interp}, and the continuous embedding $W^{1,d+\epsilon}(\Omega)\subset L^\infty(\Omega)$,
 and arguing as for  \eqref{i32},
  we find
  \begin{equation}
 \label{chain-6}
 \begin{aligned}
\left|  I_3  \right|& \leq  \sum_{i=1}^J    \sup_{\|\varphi \|_{ W^{1,d+\epsilon}(\Omega)}  \leq     1  }
  \left|  \int_{\pwc{\mathsf{t}}{\tau}(\sigma_{i-1})}^{\pwc{\mathsf{t}}{\tau}(\sigma_i)}   \int_\Omega \mathcal{R}_\tau |\varphi|\dd x \dd r  \right|
  \\
  &
  \begin{aligned}
  \leq &C  \int_0^T \Big(  \| \partial_t \pwl\uu\tau \|_{H^1(\Omega;\R^d)}  + \int_\Omega |\pwc\teta\tau|^{\kappa +\alpha-2} |\nabla \pwc\teta\tau|^2 \dd  x
+ \int_\Omega|\nabla \pwc\teta\tau|^2 \dd  x
\\
&
 \qquad +  \| \pwc g \tau \|_{L^1(\Omega)} + \| \eps( \partial_t \pwl\uu\tau) \|_{L^2 (\Omega;\R^{d\times d})}^2
+ \|  \partial_t \pwl\chi\tau \|_{L^2 (\Omega)}^2  \Big)
  \dd s\,.
  \end{aligned}
  \end{aligned}
\end{equation}
 With the same calculations as throughout \eqref{chain-1.1}--\eqref{chain-6}
we also estimate the terms $|\Lambda_i (\varphi^+)|$ and $|\Lambda_i(\varphi^-)|$. Inserting
the above estimates into
\eqref{chain-1},  we find  for every $\varphi \in W^{1,d+\epsilon}(\Omega)$ with $\|\varphi \|_{W^{1,d+\epsilon}(\Omega)} \leq 1$,
\[
\begin{aligned}
  \sum_{i=1}^J  
\left| \pairing{}{W^{1,d+\epsilon}(\Omega)}{\frakh_{i,\tau} - \frakh_{i-1,\tau}}{\varphi}\right|
 & \stackrel{(1)}{\leq}    
\int_\Omega   \sum_{i=1}^J (\frakh_{i,\tau} - \frakh_{i-1,\tau}) |\varphi| \dd x
+ \overline{C}
\\
& =
  \int_\Omega \left( \log(\pwc\teta\tau(T)) + \pwc\chi\tau(T) - \log(\teta_0) - \chi_0 \right)|\varphi| \dd x
+ \overline{C}
\stackrel{(2)}{\leq} C.
\end{aligned}
\]
Here, $(1)$ with a positive constant $\bar{C}$,  uniform with respect to $\varphi $,  follows from the previously proved estimates
\eqref{aprio1-discr},  \eqref{aprio3-discr}, \eqref{aprio4-discr}, \eqref{aprio6-discr}, and
\eqref{crucial-third}. Finally, (2)
is due to
 \eqref{aprio3-discr} and to
 the fact
 that  $|\log(\pwc\teta\tau (t))| \leq  C \left(|\pwc\teta \tau(t)| + \frac{1}{|\pwc\teta\tau(t)|}\right) \leq   C \left(|\pwc\teta \tau(t)| +\frac1{\underline{\teta}(T)}\right)$
 a.e.\ in $\Omega$  for all $t \in [0,T]$ thanks  to \eqref{positiv-interp}. \eo
 Using  that
 $(\pwc\teta\tau)_\tau$ is bounded in $L^\infty (0,T; L^1(\Omega))$ by
 \eqref{crucial-first}, \beo we ultimately conclude that
 \[
 \exists\, C>0  \ \ \forall\, \tau>0 \, : \qquad
\sup_{\varphi \in W^{1,d+\epsilon}(\Omega), \ \| \varphi \|_{W^{1,d+\epsilon}(\Omega)}\leq 1}
\mathrm{Var}(\pairing{}{W^{1,d+\epsilon}(\Omega)}{\log(\pwc\teta\tau) + \pwc\chi\tau}{\varphi}; [0,T]) \leq C.
\]
Therefore, \eqref{weaker-sixth} follows, taking into account estimate \eqref{aprio4-discr}, which in particular yields a bound for
$(\pwc\chi\tau)_\tau$
in $\mathrm{BV}([0,T];L^2(\Omega))$. \eo

Under the additional Hypothesis (V), the same comparison argument in \eqref{eq-discr-w} as for the  \textbf{Seventh estimate}  yields \eqref{aprio7-discr}.

For the  \textbf{Eighth estimate},   in the case $\mu=0$ we perform a comparison in \eqref{eq-discr-chi}. Based on
 \eqref{aprio1bis}, \eqref{aprio3-discr}, \eqref{aprio4-discr}, and \eqref{aprio6-discr}
 we conclude
 \[
 \sup_{\tau>0}\left(  \| A_p (\pwc \chi \tau)\|_{L^2 (0,T; L^2 (\Omega))} +\| \pwc \xi\tau \|_{L^2(0,T; L^2(\Omega))} \right)
 \leq C
 \]
 whence \eqref{aprio8-discr} by the aforementioned regularity results from \cite{savare98}.
\end{proof}
\beo
\begin{remark}
\upshape
\label{rmk:after6}
Since we are not able to obtain an estimate in $\mathrm{BV}([0,T];W^{1,d+\epsilon}(\Omega)^*)$
for the family $(\log(\pwc\teta\tau))_\tau$,
in the time continuous limit
 the (albeit poor) regularity information
 \begin{equation}
 \label{log-teta-lost}
 \log(\teta) \in \mathrm{BV}([0,T];W^{1,d+\epsilon}(\Omega)^*)
 \end{equation}
  will be lost.
 Observe that it cannot be recovered from a comparison argument in the rescaled heat equation, since we will only obtain the entropic
 formulation of
 \eqref{eq0}.

 Still, the  formal calculations from the  \emph{Sixth estimate} in Section $3$ suggest that it should be possible to recover \eqref{log-teta-lost}. Possibly, this could be done via a double approximation procedure, where one first passes to the limit in a suitable modified version of the time-discrete scheme  \eqref{eq-discr-w}--\eqref{eq-discr-chi} and obtains in the time-continuous limit a \emph{regularized} version of system
 \eqref{eq0}--\eqref{eqII}, allowing for a \emph{rigorous} test of the heat equation by $\frac1{\teta}$. Thus, in the frame of this approximation of
  \eqref{eq0}--\eqref{eqII} it would be possible to prove the \emph{Sixth estimate}, and hence to conclude  \eqref{log-teta-lost} by a further limit passage.
\end{remark}
\eo
%
\section{Passage to the limit}
\label{s:pass-limit}
Let
$(\pwc\teta\tau,\pwl\teta\tau, \pwc\uu\tau, \upwc\uu\tau, \pwl\uu\tau, \pwwll\uu\tau, \pwc\chi\tau,
\upwc\chi\tau,  \pwl\chi\tau)_\tau$ be a family of approximate solutions,
fulfilling the discrete version of \eqref{eq-w-interp}--\eqref{eq-chi-interp}
of system \eqref{eq0}--\eqref{eqII},
the discrete entropy inequality \eqref{entropy-ineq-discr} and the discrete total energy inequality \eqref{total-enid-discr}: its existence  is ensured by Proposition~\ref{prop:discr-enid}. We derive a preliminary compactness result,
relying on the a priori estimates from Prop.\ \ref{prop:discrete-aprio} \beo  and on an auxiliary compactness result, Theorem \ref{th:mie-theil},
proved in the Appendix. \eo
\begin{lemma}[Compactness, $\mu \in \{0,1\}$]
\label{l:compactness}
Under \beo \textbf{Hypotheses (0)--(III)} \eo and conditions \eqref{bulk-force}--\eqref{datochi} on the data $\mathbf{f},  g,  h,  \teta_0,  \uu_0,  \vv_0,  \chi_0$, for any sequence $(\tau_k )_k \subset (0,+\infty)$ with $\tau_k \downarrow 0$ as
$k \to \infty$,  there exist a (not relabeled) subsequence, and  a triple $(\teta,\uu, \chi)$ such that the following
convergences hold
\begin{align}
&
\label{cnv-1}
 \pwl \uu{\tau_k} \weaksto \uu  &&  \text{ in $H^1(0,T;\boY) \cap
W^{1,\infty}(0,T;\boZ)$,}
\\
&
\label{cnv-2}
  \pwc \uu{\tau_k},\, \upwc \uu{\tau_k} \to \uu &&
 \text{ in $L^\infty(0,T;H^{2-\epsilon}(\Omega;\R^d)) $ for all $\epsilon \in (0,1]$,}
 \\
 &
 \label{cnv-2-added}
 \pwl \uu{\tau_k} \to \uu  && \text{ in $\mathrm{C}^0([0,T];H^{2-\epsilon}(\Omega;\R^d)) $ for all $\epsilon \in (0,1]$,}
\\
&
\label{cnv-3}
\partial_t \pwwll {\uu}{\tau_k} \weakto \uu_{tt} &&
 \text{ in $L^2(0,T;L^{2}(\Omega;\R^d)) $,}
 \\
 &
 \label{cnv-4}
 \partial_t \pwl{\uu}{\tau_k}\to \uu_t &&\text{ in $L^2(0,T;H^1(\Omega;\R^d)) $,}
\\
&
\label{cnv-5}
\pwc\chi{\tau_k}, \, \upwc\chi{\tau_k}, \, \pwl \chi{\tau_k} \weakto \chi && \text{ in $L^\infty (0,T; W^{1,p}(\Omega)) \cap H^1 (0,T;L^2(\Omega))$,}
\\
&
\label{cnv-6}
 \pwl \chi{\tau_k} \to \chi && \text{ in $\mathrm{C}^0 ([0,T]; X)$ for all $X$ such that $W^{1,p}(\Omega) \Subset X\subset L^2(\Omega)$},
 \\
 &
\label{cnv-5-added}
\pwc\chi{\tau_k}, \, \upwc\chi{\tau_k}\to \chi && \text{ in $L^\infty (0,T; X)$ for all $X$ such that $W^{1,p}(\Omega) \Subset X\subset L^2(\Omega)$,}
\\
 &
 \label{cnv-7}
 \pwc \teta{\tau_k}\weakto \teta && \text{ in $L^2 (0,T; H^1(\Omega))$},
 \\
&
\label{cnv-9}
\log(\pwc\teta{\tau_k})  \weaksto  \log(\teta)  &&
\text{ in } L^2 (0,T; H^1(\Omega))  \cap \beo L^\infty (0,T; W^{1,d+\epsilon}(\Omega))  \eo \quad \text{for every } \epsilon>0,
\\
&
\label{cnv-10}
\log(\pwc\teta{\tau_k})  \to  \log(\teta)  &&
\text{ in
}L^2(0,T;L^s(\Omega)) \text{ for all $ s \in [1,6)$ if $d=3$, and all $s\in [1,\infty) $ if $d=2$,}
\\
&
\label{cnv-mi-th}
 \log(\pwc\teta{\tau_k}(t)) \weakto \log(\teta(t))   && \beo  \text{ in $H^1(\Omega)$ for almost all $t \in (0,T)$,}  \eo
\\
&
\label{cnv-8}
\pwc\teta{\tau_k}\to \teta &&  \text{  in $L^h(\Omega\times (0,T))$
for all $h\in [1,8/3) $ for $d=3$ and all $h\in [1, 3)$ if $d=2$,}
\end{align}
and $\teta$  also fulfills
\begin{equation}
\label{additional-teta}
\teta \in L^\infty (0,T; L^1(\Omega)),
\quad
 \teta \geq \underline{\teta} \text{ a.e.\ in $\Omega \times (0,T)$}
 \end{equation}
  (with $\underline{\teta}$  from \eqref{strict-pos-wk}).

Under the additional \textbf{Hypothesis (V)}, we also have $\teta \in \mathrm{BV} ([0,T]; W^{2,d+\epsilon} (\Omega)^*)$ for all $\epsilon>0$, and
\begin{align}
& \label{cnv-11}
\pwc\teta{\tau_k} \to \teta && \text{ in } L^2 (0,T; Y) \text{ for all $Y$ such that $H^1(\Omega) \Subset Y \subset W^{2,d+\epsilon} (\Omega)^*$},
\\
&
\label{cnv-12}
\pwc\teta{\tau_k}(t) \to \teta(t) && \text{ in } W^{2,d+\epsilon} (\Omega)^* \text{ for all } t \in [0,T].
\end{align}
\end{lemma}
\begin{proof}
Due to  due to estimates \eqref{aprio1-discr} and \eqref{aprio2-discr}, there holds
 \begin{equation}
\label{e:stability-estimate}
\begin{array}{ll}
 &  \| \pwl \uu\tau - \pwc \uu\tau\|_{L^\infty (0,T;\boY)} \leq \tau^{1/2}  \| \partial_t \pwl
 {\uu}\tau \|_{L^2 (0,T;\boY)} \leq S \tau^{1/2},
 \\
&  \| \pwwll \uu\tau - \partial_t\pwl {\uu}\tau\|_{L^\infty
(0,T;L^2(\Omega;\R^d))}\leq \tau^{1/2}  \| \partial_t \pwwll
 {\uu}\tau \|_{L^2 (0,T;L^2(\Omega;\R^d))} \leq S \tau^{1/2}.
\end{array}
 \end{equation}
 Taking into account   estimates
 \eqref{aprio1bis}, \eqref{aprio1-discr}, \eqref{aprio2-discr},  applying well-known weak and strong compactness results (for the latter, cf.\ e.g.\ \cite{simon}),
 also relying on \eqref{e:stability-estimate}
  we conclude convergences \eqref{cnv-1}--\eqref{cnv-4}. The same kind of arguments yields  \eqref{cnv-5}--\eqref{cnv-5-added}
  on account of estimates
  \eqref{aprio3-discr} and \eqref{aprio4-discr}.

\beo Concerning the convergence of the   temperature variables, observe that the forthcoming Theorem \ref{th:mie-theil}
applies to the family $(\log(\pwc\teta{\tau}))_\tau$, with the choices $V= H^1(\Omega)$, $p=2$, $Y= W^{1,d+\epsilon}(\Omega)$.
Hence we conclude that, up to a subsequence  the functions $\log(\pwc\teta{\tau_k})$ weakly$^*$  converge
to some $\lambda \in L^2 (0,T; H^1(\Omega)) \cap L^\infty(0,T; W^{1,d+\epsilon}(\Omega)^*)$ for all $\epsilon>0$, and that
$\log(\pwc\teta{\tau_k}(t)) \weakto \lambda(t)$ in $H^1(\Omega)$ for almost all $t\in (0,T)$. Therefore, up to a further subsequence
we have  $\log(\pwc\teta{\tau_k}(\cdot, t)) \to \lambda(\cdot, t)$ almost everywhere in $\Omega$.  Thus,
\begin{equation}
\label{che-bello}
\pwc\teta{\tau_k} \to \teta:= e^{\lambda} \qquad \foraa\, (t,x) \in \Omega \times (0,T).
\end{equation}

Writing $\lambda = \log(\teta)$, we immediately deduce \eqref{cnv-9} and \eqref{cnv-mi-th}. Convergence \eqref{cnv-10} follows from  this argument: from \eqref{cnv-mi-th} we gather that
for almost all $t\in (0,T)$
 $\log(\pwc\teta{\tau_k}(t)) \to \log(\teta(t)) $ in every Banach space $Z$ such that $H^1(\Omega) \Subset Z$, in particular in $L^s(\Omega)$
 with $s$ as in  \eqref{cnv-10}. From the bound of
$(\log(\pwc\teta{\tau_k}) )_k$ in $L^2 (0,T; H^1(\Omega)) \cap L^\infty (0,T; W^{1,d+\epsilon}(\Omega)^*)$, combined with the interpolation inequality
(cf.\ e.g.\ \cite[Lemma 8]{simon})
\[
\forall\, \eta>0 \ \ \exists\, C_\eta>0 \ \ \forall\, \theta \in H^1(\Omega)\, : \ \ \|\theta \|_{L^s(\Omega)}\leq \eta  \|\theta \|_{H^1(\Omega)}
+C_\eta  \|\theta \|_{W^{1,d+\epsilon}(\Omega)^*},
\]
we also infer that the sequence $(\log(\pwc\teta{\tau_k}))_k$ is uniformly integrable in $L^2(0,T; L^s(\Omega))$. Then, by e.g.\
\cite[Thm.\ III.3.6]{edwards} the desired  \eqref{cnv-10} ensues.

Furthermore, from the bound \eqref{aprio6-discr} for $(\pwc\teta{\tau_k})_k$ we deduce by interpolation (cf.\
\eqref{estetainterp})
that the sequence  $(\pwc\teta{\tau_k})_k$ is uniformly integrable in    $L^h(\Omega\times (0,T))$
for all $h\in [1,8/3) $ for $d=3$ and all $h\in [1, 3)$ if $d=2$. Combining this with \eqref{che-bello} we deduce  convergence \eqref{cnv-8}.
By weak compactness arguments,   \eqref{aprio6-discr}  gives the weak convergence \eqref{cnv-7}.
\eo
 With a lower semicontinuity argument one also has that $\teta \in L^\infty (0,T; L^1(\Omega))$.
  Relying on \eqref{cnv-8} and on the approximate positivity property \eqref{positiv-interp},
we  also conclude the last of \eqref{additional-teta}.

Finally, under the additional Hypothesis (V), we also dispose of the $\mathrm{BV}$-estimate \eqref{aprio7-discr} for $(\pwc\teta\tau)_\tau$. Combining this with \eqref{aprio6-discr} and applying 
 an Aubin-Lions type compactness result for $\BV$-functions (see, for instance, \cite[Cor.\ 4]{simon} or \cite[Chap.\ 7, Cor.\ 4.9]{roub-NPDE})
 we conclude   \eqref{cnv-11}.  The pointwise convergence \eqref{cnv-12}  ensues from, e.g., \cite[Thm.\ 6.1]{MieTh04}.
\end{proof}

\noindent We are now in the position to develop the
\underline{\bf Proof of Theorem \ref{teor3}}, by passing to the limit
in the time-discrete scheme set up in Sec.\ \ref{s:time-discrete},  in the case $\mu=1$.
Let $(\tau_k)_k$ be a vanishing sequence of time-steps, and  let
$$(\pwc\teta{\tau_k},\pwl\teta{\tau_k}, \pwc\uu{\tau_k}, \upwc\uu{\tau_k}, \pwl\uu{\tau_k}, \pwwll\uu{\tau_k}, \pwc\chi{\tau_k},
\upwc\chi{\tau_k},  \pwl\chi{\tau_k})_k$$ be a sequence of approximate solutions.
 We can exploit the
 compactness results from Lemma \ref{l:compactness}.
We split the limit passage
 in the following steps.

\paragraph{\textbf{Ad the weak momentum equation \eqref{weak-momentum}}}
  Relying on  convergences \eqref{cnv-1}, \eqref{cnv-3}--\eqref{cnv-4},  on  \eqref{cnv-5-added}
  which yields that $a(\upwc\chi{\tau_k}) \to a(\chi)$  and $b(\pwc\chi{\tau_k}) \to b(\chi)$
  in
  $L^p( \Omega\times (0,T))$ for every $1\leq p <\infty$,
  and on \eqref{cnv-7}, as well as on \eqref{converg-interp-f} for $(\pwc {\mathbf{f}}{\tau_k})_k$, we pass to the limit in
  the discrete momentum equation
  \eqref{eq-u-interp} and  conclude that the triple $(\teta,\uu,\chi)$ fulfills
\eqref{weak-momentum}.

\paragraph{\textbf{Ad the weak formulation \eqref{constraint-chit}--\eqref{energ-ineq-chi} of the equation for $\chi$, $\mu=1$}}
The argument for obtaining \eqref{constraint-chit}--\eqref{energ-ineq-chi} in the limit follows exactly the same lines as the proof of \cite[Thms.\ 4.4, 4.6]{hk1} (see also \cite[Thm.\ 3]{rocca-rossi-deg}). Therefore we only recapitulate it, referring to the latter papers for all details.

First of all,  as we have pointed out in the proof of Proposition \ref{prop:discr-enid},
 the discrete flow rule
 \eqref{eq-discr-chi}  for $\chi$  can be interpreted as the Euler-Lagrange equation for the minimum problem  \eqref{min-prob-chi}, i.e.\ (recall that here $\mu=1$ and that  $\widehat\alpha=I_{(-\infty, 0]}$ and $\widehat\beta = I_{[0,+\infty)}$)
\begin{equation}
\label{minimum-prob}
\begin{aligned}
 \min_{\chi \in W^{1,p}(\Omega)} \Big\{
 \int_\Omega \Big(
 \frac{\tau^{3/2}}2   \left|\frac{\chi - \chitau{k-1}}\tau\right|^2
 + \left( \frac{\chitau k -\chitau {k-1}}{\tau}\right)\chi     & + I_{(-\infty, 0]}\left( \frac{\chi - \chitau{k-1}}{\tau} \right)+ \frac{|\nabla \chi|^p}p
 + I_{[0,+\infty)}(\chi)
 \\ &
 +\widehat{\gamma}(\chi) 
 + b(\chi) \frac{\eps(\utau{k-1}) \elm \eps(\utau{k-1}) }2 - \wtau{k} \chi
 \Big) \dd x
\Big \}
\end{aligned}
\end{equation}
  Writing necessary optimality conditions for the minimum problem \eqref{minimum-prob},
  with the very same calculations as  in the proof of  \cite[Thm.\ 3]{rocca-rossi-deg},
   we arrive
at
\begin{equation}
\label{eq-chi-ineq-better}
\begin{aligned}
&
  \int_\Omega  \Big( \partial_t \pwl
\chi{\tau}(t)  \psi + \sqrt{\tau}  \partial_t \pwl \chi{\tau}(t)
\psi + |\nabla{\pwc\chi\tau}(t)|^{p-2} \nabla{\pwc\chi\tau}(t) \cdot
\nabla \psi +     \gamma(\pwc \chi\tau (t) )   \psi
+ \pwc j\tau (t) \psi
 \Big) \, \mathrm{d}x   \geq 0
\\
&    \ \text{ for all } t \in [0,T] \text{ and all }
 \psi \in W^{1,p}(\Omega) \text{ s.t.\ there exists $\nu>0$ with  } 0
\leq \nu \psi + \pwc\chi\tau(t) \leq \upwc\chi\tau(t) \ \aein \,
\Omega,
\end{aligned}
\end{equation}
where 
where we have used the place-holder
\begin{equation}
\label{plc-holder}
 \pwc j\tau  :=
b'(\pwc\chi\tau) \frac{\eps(\upwc \uu\tau)\elm \eps(\upwc \uu\tau)}2
- \pwc\teta\tau.
 \end{equation}
Choosing $\psi= - \partial_t \pwl \chi{\tau}(t) $ in
\eqref{eq-chi-ineq-better} and summing over the index $k$ we deduce
the \emph{discrete version} of the \emph{\berdue
energy-dissipation  \erdue}  inequality \eqref{energ-ineq-chi} for
$\chi$, holding for all  $0 \leq s \leq t \leq T$, viz.
\begin{equation}
\label{energ-ineq-discrete}
\begin{aligned}
 &   \int_{\pwc{\mathsf{t}}{\tau}(s) }^{\pwc{\mathsf{t}}{\tau}(t) }
   \int_{\Omega}(1+ \tau^{1/2} ) |\partial_t \pwl \chi\tau|^2 \dd x \dd r
    +  \int_{\Omega} \left(
\frac1p |\nabla \pwc\chi\tau(\pwc{\mathsf{t}}{\tau}(t))) |^p +
  W(\pwl\chi\tau(\pwl{\mathsf{t}}{\tau}(t)))
   \right) \dd x\\
 & \leq
   \int_{\Omega} \left(
\frac1p |\nabla \pwc\chi\tau(\pwc{\mathsf{t}}{\tau}(s)) |^p +
  W(\pwl\chi\tau(\pwl{\mathsf{t}}{\tau}(s)))
   \right) \dd x\
  \\
& \qquad \qquad +  \int_{\pwc{\mathsf{t}}{\tau}(s) }^{\pwc{\mathsf{t}}{\tau}(t) } \int_\Omega \partial_t \pwl \chi\tau \left(- b'(\pwc\chi\tau)
  \frac{\varepsilon(\upwc\uu\tau)\mathrm{\elm}\varepsilon(\upwc\uu\tau)}2
+\pwc\w\tau\right)\dd x \dd r + C \tau \|\partial_t \pwl \chi\tau \|_{L^2 (0,T; L^2(\Omega))}^2\,,
\end{aligned}
\medskip
\end{equation}
where
we have used that
\[
\int_{\pwc{\mathsf{t}}{\tau}(s) }^{\pwc{\mathsf{t}}{\tau}(t) }\gamma(\pwc \chi\tau)  \partial_t \pwl \chi\tau  \dd x \dd r
=  \int_{\pwc{\mathsf{t}}{\tau}(s) }^{\pwc{\mathsf{t}}{\tau}(t) }\gamma(\pwl \chi\tau )  \partial_t \pwl \chi\tau  \dd x \dd r
+
\int_{\pwc{\mathsf{t}}{\tau}(s) }^{\pwc{\mathsf{t}}{\tau}(t) }\left(\gamma(\pwc \chi\tau  ) - \gamma(\pwl \chi\tau  ) \right)  \partial_t \pwl \chi\tau  \dd x \dd r \doteq I_1+I_2
\]
and that
\[
I_1 \stackrel{(1)}{=}  \int_\Omega \widehat \gamma(\pwl \chi\tau
(\pwc{\mathsf{t}}{\tau}(t)  )) \dd x -  \int_\Omega \widehat
\gamma(\pwl \chi\tau( \pwc{\mathsf{t}}{\tau}(s)  )) \dd x \stackrel{(2)}{=}
\int_\Omega W(\pwl \chi\tau (\pwc{\mathsf{t}}{\tau}(t) ) ) \dd x -
\int_\Omega W(\pwl \chi\tau \pwc{\mathsf{t}}{\tau}(s)  )) \dd x
\]
where $(1)$ follows from the chain rule and $(2)$ from the fact that $W= \widehat{\beta} + \widehat{\gamma}$ with
$\widehat{\beta} = I_{[0,+\infty)}$.  Finally,
\[
I_2
\leq \| \partial_t \pwl \chi\tau  \|_{L^2 (0,T;L^2(\Omega))}
\| \gamma(\pwl \chi\tau ) - \gamma(\pwc \chi\tau )\|_{L^2 (0,T;L^2(\Omega))} \leq C\tau \|\partial_t \pwl \chi\tau \|_{L^2 (0,T; L^2(\Omega))}^2
\]
thanks to the Lipschitz continuity of $\gamma$.

Second, repeating the ``recovery sequence'' argument from \cite[proof of Thm.\ 4.4]{hk1}, we
improve the weak convergence
\eqref{cnv-5}
to
\begin{equation}
\label{strong-w1p-chi}
\pwc\chi{\tau_k} \to \chi \quad \text{ in } L^p (0,T; W^{1,p}(\Omega)).
\end{equation}
We refer to \cite{hk1} and \cite{rocca-rossi-deg} for all the related calculations.

We are now in the position to take the  limit  as $\tau_k \down 0$
in the approximate  \berunodue energy-dissipation \erunodue energy inequality \eqref{energ-ineq-discrete}.
We pass to the limit on the left-hand side by lower semicontinuity,
\beo relying on     convergences
\eqref{cnv-5}--\eqref{cnv-6} and on the fact that
$\pwc\chi{\tau_k}(t) \to \chi(t)$ in $W^{1,p}(\Omega)$ for all $t\in [0,T]$.\eo

For the right-hand side, we exploit the  strong convergence \eqref{strong-w1p-chi},
yielding that $\pwc\chi{\tau_k} (s)\to
\chi(s) $  in $W^{1,p}(\Omega)$, whence $\pwc\chi{\tau_k}
(s)\to \chi(s) $  in $\mathrm{C}^0 (\overline\Omega)$, for almost
all $s \in (0,T)$. It follows from $\widehat \gamma \in
\mathrm{C}^2(\R)$ that $\widehat \gamma $ has at most quadratic
growth on bounded subsets of $\R$. We combine this with the uniform
convergence of $(\pwc\chi{\tau_k} (s))_k$ to conclude that
$\int_\Omega \widehat{\gamma} (\pwc\chi{\tau_k} (s)) \dd x \to
\int_\Omega \widehat{\gamma}(\chi(s)) \dd x $ for almost all $s \in
(0,T).$ Since $\widehat \beta = I_{[0,+\infty)}$, we have $
\int_\Omega W(\pwc\chi{\tau_k}(s)) \dd x  \to\int_\Omega W(\chi(s))
\dd x $
 for almost all $s \in (0,T)$.
 Since $(\pwl\chi\tau)_\tau$ is bounded in $H^1(0,T; L^2(\Omega))$, we also have
 \begin{equation}
 \label{citata-molto-dopo}
 \sqrt{\tau_k} \partial_t \pwl\chi{\tau_k} \to 0 \text{ in } L^2(0,T;
L^2(\Omega)).
\end{equation}
 Combining   the weak convergence \eqref{cnv-5} with the strong ones \eqref{cnv-2},   \eqref{cnv-5-added}
\beo (yielding that $b'(\pwc\chi{\tau_k}) \to b'(\chi)$ in $L^p( \Omega\times (0,T))$ for all $1\leq p <\infty$),  \eo
 and \eqref{cnv-8}, we also pass to the limit in the second integral term on the right-hand side of \eqref{energ-ineq-discrete}. The last summand obviously tends to zero. Therefore,
 we conclude the \berdue energy-dissipation \erdue inequality \eqref{energ-ineq-chi}.

 Clearly,  convergence   \eqref{cnv-5}  and the fact that $\partial_t{\pwl{\chi}{\tau}} \leq 0$ a.e.\ in $\Omega \times (0,T)$ ensure that $\chi_t \leq 0$
 .e.\ in $\Omega \times (0,T)$, i.e.\   \eqref{constraint-chit}. To obtain the variational inequality
 \eqref{ineq-chi}, together with
 \eqref{xi-def}, we proceed exactly as in \cite{hk1, rocca-rossi-deg}. The main steps are as follows:
passing to the limit in
\eqref{eq-chi-ineq-better}
as $\tau_k \down 0$ with suitable test functions from \cite[Lemma 5.2]{hk1},
also relying on \eqref{citata-molto-dopo},
we prove that for
almost all $t \in (0,T)$
\[
\begin{aligned}
&
 \int_\Omega  \Big( \chi_t(t) \tilde\psi 
  +
|\nabla\chi(t))|^{p-2}\nabla \chi(t) \cdot \nabla \tilde\psi +
\gamma(\chi(t)) \tilde\varphi + b'(\chi(t))\frac{\varepsilon(\ub(t))
\elm  \eps(\ub(t)) }{2}\tilde\psi  -\teta(t) \tilde\psi
\Big) \, \mathrm{d}x  \geq 0
\\ & \qquad \qquad \qquad \qquad\qquad \qquad \qquad \qquad\qquad \qquad
\text{for all } \tilde\psi \in W_-^{1,p}(\Omega) \text{ with }
\{\tilde\psi=0\} \supset \{\chi(t)=0\},
\end{aligned}
\]
where we have used the short-hand notation $\{ f=0\}$ for
$\{ x \in \Omega\, : \ f(x) =0\}$.
From this, arguing as in the proof of \cite[Thm. 4.4]{hk1} we deduce
that for almost all $t \in (0,T)$
\begin{equation}
\label{ineq-ssy2}
\begin{aligned}
&
 \int_\Omega  \Big( \chi_t(t) \psi  +
|\nabla\chi(t))|^{p-2}\nabla \chi(t) \cdot \nabla \psi +
\gamma(\chi(t)) \varphi + b'(\chi(t))\frac{\varepsilon(\ub(t))
\elm \eps(\ub(t))}{2}\psi  -\teta(t) \psi \Big) \,
\mathrm{d}x
\\ &
  \geq \int_{\{\chi(t)=0\}} \left( \gamma(\chi(t))  +
b'(\chi(t))\frac{\varepsilon(\ub(t))
\elm\varepsilon(\ub(t))}{2}  -\teta(t)  \right)^+ \psi
\, \mathrm{d}x
 \ \ \text{for all } \psi \in
W_-^{1,p}(\Omega).
\end{aligned}
\end{equation}
Relying on \eqref{ineq-ssy2},
 it is possible to check that  the function $\xi$ from
\eqref{xi-specific} complies with   \eqref{ineq-chi} and
 \eqref{xi-def}, cf.\ \cite{hk1}
for all details.

 \paragraph{\bf Ad the entropy inequality \eqref{entropy-ineq}}
 Let us fix a test function
 $\varphi \in \mathrm{C}^0 ([0,T]; W^{1,d+\epsilon} (\Omega)) \cap H^1 (0,T; L^{6/5}(\Omega))$ (for some $\epsilon>0$),
 for the entropy inequality \eqref{entropy-ineq}. We pass to the limit as $\tau_k \down 0$
 in the discrete entropy inequality \eqref{entropy-ineq-discr}, with the discrete test functions
 constructed from $\varphi$ in \eqref{discrete-tests-phi}. In order to pass to the limit in the
  first two integral terms on the left-hand side of \eqref{entropy-ineq-discr}, we combine convergences
 \eqref{cnv-1},  \eqref{cnv-5-added}, and \eqref{cnv-mi-th},   with
 the convergence \eqref{convergences-test-interpolants} for the
test functions.
 In order to deal with the last integral on the left-hand side, we observe that
 the family
 \begin{equation}
 \label{fist-bound}
\text{ $ (\condu (\pwc \teta\tau) \nabla \log(\pwc\teta\tau))_\tau$ is bounded
 in $L^{1+\delta} (Q;\R^d)$ for some $\delta >0$. }
\end{equation}
 Indeed,
the growth condition \eqref{hyp-K}
 implies that
 \[
 |\condu (\pwc \teta\tau) \nabla \log(\pwc\teta\tau)| \leq
 C \left(|\pwc\teta\tau|^{\kappa-1} + \frac1{\pwc\teta\tau} \right) |\nabla \pwc \teta\tau|\leq  | \leq
 C\left (|\pwc\teta\tau|^{\kappa-1} + \frac1{\underline{\teta}(T)} \right) |\nabla \pwc \teta\tau| \qquad  \aein \, \Omega \times (0,T)
\]
 (also due to   the strict positivity  \eqref{positiv-interp}).
 Thus, it remains to bound the term $|\pwc\teta\tau|^{\kappa-1}
|\nabla \pwc \teta\tau|$. To do so, we observe
\begin{equation}
\label{no-bound}
\begin{aligned}
 \iint_Q \left( |\pwc\teta\tau|^{\kappa-1} |\nabla
\pwc \teta\tau| \right)^{r} \dd x \dd t &  \leq \| (
|\pwc\teta\tau|^{(\kappa-\alpha)/2})^{r} \|_{L^{2/(2-r)}(Q)}  \| (
|\pwc\teta\tau|^{(\kappa+\alpha-2)/2} |\nabla \pwc
\teta{\tau_k})^{r} \|_{L^{2/r}(Q; \RR^d)} \\ & \leq C \| (
|\pwc\teta\tau|^{(\kappa-\alpha)/2})^{r} \|_{L^{2/(2-r)}(Q)}
\end{aligned}
\end{equation}
for some $r>0$ (to be chosen below), where we   have  exploited that
 $(
|\pwc\teta\tau|^{(\kappa+\alpha-2)/2} \nabla \pwc \teta{\tau})_\tau$
is bounded in $L^2(Q;\R^d)$ thanks to \eqref{necessary-added} (cf.\
also \eqref{additional-info}). Indeed the latter estimate yields
that $((\pwc\teta{\tau})^{(\kappa+\alpha)/2})_\tau$ is bounded in
$L^2(Q)$,  hence that $((\pwc\teta{\tau})^{(\kappa-\alpha)/2})_\tau$
is bounded in $L^{2(\kappa+\alpha)/(\kappa-\alpha)}(Q)$. Therefore,
it is sufficient to choose in \eqref{no-bound} $r$ such that
$2r/(2-r) = 2(\kappa+\alpha)/(\kappa-\alpha)$, i.e.\ $r= (\kappa
+\alpha)/\kappa$, which is strictly bigger than $1$.
 Hence,  up to some subsequence
$\condu (\pwc \teta{\tau_k}) \nabla \log(\pwc\teta{\tau_k}) $ weakly
converges to some $ \eta$   in $L^{1+\delta} (Q;\R^d)$.
 In
order to identify
 $\eta $ as  $\condu(\teta) \nabla \log(\teta)$, we use these facts.
 We first show that
 \begin{equation}
 \label{matamoros}
 |\pwc\teta{\tau_k}|^{(\kappa+\alpha-2)/2}  \nabla \pwc \teta{\tau_k} \weakto |\teta|^{(\kappa+\alpha-2)/2}  \nabla\teta \qquad \text{in }  L^2(Q;\R^d).
 \end{equation}
 Indeed, on the one hand, \eqref{cnv-7} gives
 \begin{equation}
 \label{weak-nabla-teta}
 \nabla\pwc\teta{\tau_k}  \weakto \nabla \teta \qquad \text{ in
  $L^2 (0,T; L^2(\Omega;\R^d))$.}
  \end{equation}
   On the other hand,
 the pointwise convergence $\pwc \teta{\tau_k} \to \teta$ a.e.\ in $\Omega \times (0,T)$
 combined with the fact that $(\pwc\teta{\tau_k})_k$ is bounded in $L^{\kappa +\alpha}(\Omega)$
 yields that $\pwc\teta{\tau_k} \to \teta$ in  $L^{\kappa +\alpha-\epsilon}(\Omega)$ for all $\epsilon>0$.
\beo Therefore $|\pwc\teta\tau|^{(\kappa+\alpha-2)/2} \to |\teta|^{(\kappa+\alpha-2)/2}$ in
  $L^{\eta_\epsilon}(\Omega)$, with $\eta_\epsilon:= \frac{2(\kappa+\alpha)}{\kappa+\alpha-2}-\epsilon $,
  for all $\epsilon>0$. \eo
  We may then choose $\epsilon>0$ such that $\eta_\epsilon >2$ and combine this with \eqref{weak-nabla-teta} to conclude
  \eqref{matamoros}, taking into account that
   $( |\pwc\teta{\tau_k}|^{(\kappa+\alpha-2)/2}  \nabla \pwc \teta{\tau_k})_k$ is bounded in $L^2(Q;\R^d)$.
 Second, we have that
 \begin{equation}
 \label{juramento}
 |\pwc\teta{\tau_k}|^{(\kappa-\alpha)/2} \to \teta^{(\kappa-\alpha)/2} \text{ in }  L^{2(\kappa+\alpha)/(\kappa-\alpha) - \epsilon}(\Omega) \quad \text{for all } \epsilon>0,
 \end{equation}
 again due to the pointwise convergence of $\pwc\teta{\tau_k}$ and to the fact  $(\pwc\teta{\tau_k})_k$ is bounded in $L^{\kappa +\alpha}(\Omega)$.
 It follows from \eqref{matamoros}, \eqref{juramento},  the growth condition on $\condu$,  and the Lebesgue Theorem,
 that
  \begin{equation}
  \label{evvai}
\condu (\pwc \teta{\tau_k}) \nabla \log(\pwc\teta{\tau_k})  \weakto
  \condu(\teta) \nabla \log(\teta) \qquad \text{in } L^{1+\delta} (Q;\R^d).
  \end{equation}
 This  and convergence \eqref{convergences-test-interpolants}
  for the discrete test functions
 enables us to take the limit in third term on the left-hand side of \eqref{entropy-ineq-discr}.
 The passage to the limit in the first  two
 integrals on the right-hand side results from convergences  \eqref{cnv-6},  \eqref{cnv-mi-th}
  and again
  \eqref{convergences-test-interpolants}. For the third term, we use that
 \[
\begin{aligned}
&
 \limsup_{k \to \infty}
 \left( - \int_{\pwc{\mathsf{t}}{\tau_k}(s)}^{\pwc{\mathsf{t}}{\tau_k}(t)}  \int_\Omega \condu(\pwc \teta{\tau_k}(r)) \frac{\pwc\varphi{\tau_k}(r)}{\pwc \teta{\tau_k}(r)}
\nabla \log(\pwc \teta{\tau_k}(r)) \cdot \nabla \pwc \teta{\tau_k} (r) \dd x \dd r
\right)
\\
 & =-
 \liminf_{k \to \infty}
  \int_{\pwc{\mathsf{t}}{\tau_k}(s)}^{\pwc{\mathsf{t}}{\tau_k}(t)}
    \int_\Omega \condu(\pwc \teta{\tau_k}(r)) \pwc\varphi{\tau_k}(r)
    \left|
    \nabla  \log(\pwc \teta{\tau_k}(r)) \right|^2
    \dd x \dd r
\leq
 - \int_{s}^{t}
    \int_\Omega \condu( \teta(r)) \varphi(r)
    \left|
    \nabla  \log(\teta(r)) \right|^2
   \dd x \dd r
   \end{aligned}
  \]
 which results from the weak convergence \eqref{cnv-9}, combined with the pointwise convergence $\pwc\teta{\tau_k} \to \teta$ a.e.\ in $\Omega \times (0,T)$,  \eqref{convergences-test-interpolants} for the discrete test functions,
  applying the Ioffe theorem \cite{ioffe}.
 With analogous  lower semicontinuity arguments  we pass to the limit in the last two integrals on the right-hand side of
 \eqref{entropy-ineq-discr}.

 \paragraph{\bf Ad the total energy inequality \eqref{total-enid}} It follows from passing to the limit as $\tau_k \down 0$ in
 the discrete total energy inequality
 \eqref{total-enid-discr}, based on convergences \eqref{converg-interp-f}--\eqref{converg-interp-h} for $\pwc {\mathbf{f}}{\tau_k}, \, \pwc g{\tau_k}, \, \pwc h{\tau_k}$, and on
 \eqref{cnv-2},
  \eqref{cnv-4},
  \eqref{cnv-6}, and  on the pointwise convergence
 \eqref{cnv-8}.
  Observe that convergences  \eqref{cnv-2},  \eqref{cnv-4},  and  \eqref{cnv-6} are sufficient to pass to the limit  on the left-hand side of
    \eqref{total-enid-discr}, by lower semicontinuity, \emph{for all}
  $t \in [0,T]$.  However, \eqref{cnv-8} only guarantees that $\pwc \teta{\tau_k}(t) \to \teta(t)$ in
  $L^1 (\Omega)$ \emph{for almost all} $t \in (0,T)$.

\paragraph{\bf Enhanced regularity and improved total energy inequality under Hypothesis (V)}
If in addition  Hyp.\ (V)
 holds, in view of Lemma \ref{l:compactness}
 $\teta$  is in $\BV ([0,T]; W^{2,d+\epsilon}(\Omega)^*)$ for every $\epsilon>0$, and the enhanced convergences \eqref{cnv-11} and \eqref{cnv-12} hold. The latter pointwise convergence allows  us to pass to the limit on the left-hand side of   \eqref{total-enid-discr}
\emph{for all}
  $t \in [0,T]$. This ends the proof.
 \QED

\noindent We conclude this section with the \underline{\bf Proof of Theorem \ref{teor1}},  in the case $\mu=0$.
Let $(\tau_k)_k$ be a vanishing sequence of time-steps, and
$(\pwc\teta{\tau_k},\pwl\teta{\tau_k}, \pwc\uu{\tau_k}, \upwc\uu{\tau_k}, \pwl\uu{\tau_k}, \pwwll\uu{\tau_k}, \pwc\chi{\tau_k},
\upwc\chi{\tau_k},  \pwl\chi{\tau_k})_k$ be a sequence of approximate solutions;
let $(\pwc\xi{\tau_k})_k$ be a sequence of selections in  $\beta (\pwc \chi{\tau_k})$, such that $(\pwc\chi{\tau_k}, \pwc \xi{\tau_k})$
satisfy for all $k \in \N$
 the approximate equation \eqref{eq-chi-interp}.

In the case $\mu=0$,
in addition to convergences \eqref{cnv-1}--\eqref{cnv-12},
 estimates
\eqref{aprio8-discr} yield, up to a subsequence, the further convergences
\begin{equation}
\label{further-chi}
\pwc \chi{\tau_k} \to  \chi  \quad \text{in $L^2 (0,T; W^{1+\sigma,p}(\Omega))$ for all } 1 \leq \sigma <\frac1p, \qquad \pwc \chi{\tau_k} \to \chi \quad  \text{ in } L^q  (0,T; W^{1,p}(\Omega)) \text{ for all } 1 \leq q <\infty.
\end{equation}
Furthermore,
there exists $\xi \in L^2 (0,T; L^2(\Omega))$ such that
\begin{equation}
\label{further-xi}
\pwc\xi{\tau_k} \weakto \xi \quad \text{in } L^2 (0,T; L^2(\Omega)).
\end{equation}
The strong convergence \eqref{further-chi} and the strong-weak closedness of $\beta$ (as a maximal monotone operator from
$L^2 (0,T;L^2(\Omega))$ to
$L^2 (0,T;L^2(\Omega))$) immediately yield that $\xi \in \beta(\chi)$ a.e.\ in $\Omega\times (0,T)$.

Therefore, also exploiting convergences \eqref{cnv-1}--\eqref{cnv-7}  we pass to the limit in the discrete equation for $\chi$
\eqref{eq-chi-interp} and immediately conclude that the quadruple $(\teta,\uu,\chi,\xi)$
fulfills  the pointwise formulation \eqref{weak-phase}--\eqref{xi-def-reversible} of the \underline{internal parameter equation  \eqref{eqII}}.

 The proof of the \underline{entropy inequality}, of the \underline{total energy inequality}, and of the \underline{momentum equation} is clearly the same as for Theorem \ref{teor3}.

Under the additional Hypothesis (V),  as previously seen $\teta$ is in $\mathrm{BV}([0,T]; W^{2,d+\eps}(\Omega)^*)$. We prove  the \underline{weak form  \eqref{eq-teta} of the heat equation}  by passing to the limit  as $\tau_k  \down 0$
 in the approximate heat equation \eqref{eq-w-interp}, tested by an arbitrary $\varphi \in \mathrm{C}^0 ([0,T]; W^{2,d+\epsilon}(\Omega)) \cap H^1 (0,T; L^{6/5}(\Omega))$.
 The passage to the limit in the first three terms on the left-hand side,
  and on the first two terms on the right-hand side, results from convergences \eqref{converg-interp-g},
   \eqref{converg-interp-h} for  $(\pwc g{\tau_k})_k$ and $(\pwc h{\tau_k})_k$, and from \eqref{cnv-1}--\eqref{cnv-2},
   \eqref{cnv-4}--\eqref{cnv-7}: in particular, we exploit that
  $\eps (\partial_t \pwl \uu{\tau_k}) \elm  \eps (\partial_t \pwl \uu{\tau_k})\to \eps(\uu_t) \elm   \eps(\uu_t) $    \emph{strongly} in $L^1(Q)$
 thanks to the strong convergence \eqref{cnv-4}.

In order to pass to the limit with the fourth term on the left-hand
side of \eqref{eq-w-interp}, we need to derive a finer estimate for
$(\condu(\pwc \teta{\tau_k}) \nabla \pwc \teta{\tau_k})_k$.
 Arguing as for  \eqref{citata-dopo-ehsi} we use that
 \begin{equation}
 \label{newK}
 | \condu(\pwc \teta{\tau_k}) \nabla \pwc \teta{\tau_k}  | \leq C |\pwc \teta{\tau_k}|^{(\kappa-\alpha+2)/2}\,
 |\pwc \teta{\tau_k}|^{(\kappa+\alpha-2)/2}\, |\nabla \pwc \teta{\tau_k}| + C|\nabla \pwc \teta{\tau_k}| .
 \end{equation}
Now, $ (\pwc \teta{\tau_k})^{(\kappa+\alpha-2)/2} \nabla \pwc
\teta{\tau_k} $ is bounded in   $L^2 (0,T; L^2(\Omega;\R^d))$   (thanks to
\eqref{crucial-third}). On the other hand, $(\pwc\teta{\tau_k})_k$
is bounded in $L^p (Q)$  for all  $1\leq p<8/3$, in the case $d=3$
(to which we confine this discussion). Therefore,
  choosing   $\alpha \in (1/2, 1)$   such that
$\alpha > \kappa - \frac23$ (this can be done since $\kappa <5/3$ by
assumption), we conclude that $( (\pwc
\teta{\tau_k})^{(\kappa-\alpha+2)/2} )_k$ is bounded in
$L^{2+\delta}(Q)$ for some $\delta>0$. Ultimately,
 in view of \eqref{newK}
 we conclude that
$(\condu(\pwc \teta{\tau_k}) \nabla \pwc \teta{\tau_k})_k$ is
bounded in $L^{1+\bar\delta} (0,T; L^{1+\bar\delta}(\Omega))$ for
some $\bar\delta>0$, hence
\begin{equation}
\label{dopo-yes} \exists\, \eta \in L^{1+\bar\delta} (0,T;
L^{1+\bar\delta}(\Omega))\,: \qquad \condu(\pwc \teta{\tau_k})
\nabla \pwc \teta{\tau_k} \weakto \eta \text{ in } L^{1+\bar\delta}
(0,T; L^{1+\bar\delta}(\Omega))\,.
\end{equation}
In order to identify the weak limit $\eta$, it is sufficient to
observe that   (cf.\  \cite{Marcus-Mizel})  $\condu(\pwc \teta{\tau_k}) \nabla
\pwc \teta{\tau_k} = \nabla \widehat{\condu}(\pwc \teta{\tau_k})$  a.e.\ in $\Omega \times (0,T)$.
Combining the growth property \eqref{hyp-K} of $\condu$ (where
$1\leq \kappa <5/3$), with the strong convergence \eqref{cnv-8} of $\pwc
\teta{\tau_k}$ in $L^p(Q)$ for all $1\leq p<8/3$, we ultimately
conclude that $(\widehat{\condu}(\pwc \teta{\tau_k}))_k$ strongly
converges to $\widehat{\condu} (\teta) $ in $L^{1+\tilde\delta}(Q)$
for some $\tilde \delta >0$. A standard argument then yields
\begin{equation}
\label{evvai-bis} \eta = \nabla \widehat{\condu} (\teta) =
\condu(\teta) \nabla \teta \qquad \aein\, \Omega \times (0,T).
\end{equation}
 Combining
\eqref{dopo-yes} and \eqref{evvai-bis} leads to
\[
\int_0^T\int_\Omega \condu(\pwc \teta{\tau_k}) \nabla \pwc \teta{\tau_k}\cdot \nabla \varphi \dd x \dd t \to  \int_0^T\int_\Omega \condu(\teta) \nabla\teta\cdot \nabla \varphi \dd x \dd t
\]
for every test function $\varphi \in \mathrm{C}^0 ([0,T]; W^{2,d+\epsilon}(\Omega)) $.

To  complete  the passage to the limit on the right-hand side of
 \eqref{eq-w-interp}, it remains to show that
 \begin{equation}
 \label{strong-chi-serve}
 \partial_t\pwl \chi{\tau_k} \to \chi_t \qquad \text{ in } L^2 (0,T; L^2(\Omega)).
 \end{equation}
 This follows from testing  the discrete equation for $\chi$ \eqref{eq-chi-interp}
by $\partial_t\pwl \chi{\tau_k} $, integrating in time, and passing
to the limit as $k \to \infty$. Indeed, exploiting convergences
\eqref{cnv-2} and  \eqref{cnv-5}--\eqref{cnv-7} we deduce that
\[
\limsup_{k \to \infty}\int_0^T \int_\Omega |\partial_t\pwl
\chi{\tau_k}|^2 \dd x \dd t \leq \int_0^T \int_\Omega |\chi_t|^2 \dd
x \dd t,
\]
whence \eqref{strong-chi-serve}.

\beo In this way, we conclude that the limit triple $(\teta,\uu,\chi)$ fulfills for all $t\in [0,T]$
\begin{align}
&
\begin{aligned}
\label{eq-teta-inter}   &
\pairing{}{W^{2,d+\epsilon}(\Omega)}{\teta(t)}{\varphi(t)}
-\int_0^t\io \teta \varphi_t \dd x \dd s   +\itt \io \chi_t
\teta\varphi \dd x \dd s + \rho\itt \io
\hbox{\rm div}(\ub_t) \teta\varphi \dd x \dd s \\
&\quad +\itt \io \condu(\teta) \nabla \teta\nabla\varphi \dd
x \dd s   =\itt \io \left(g+\frac{\eps(\uu_t) \vism \eps(\uu_t)}2
+|\chi_t|^2\right) \varphi  \dd x \dd s + \int_0^t \int_{\partial\Omega} h \varphi   \dd S \dd s
 +\io \ental_0\varphi(0)\dd x\\
 &\quad \text{for all }
\varphi\in C^0([0,T]; W^{2,d+\epsilon}(\Omega))\cap
H^1(0,T;L^{6/5}(\Omega)) \text{ for some $\epsilon>0$}, 
\end{aligned}
\end{align}
whence  for every $\bar{\varphi} \in  W^{2,d+\epsilon}(\Omega)$ and for every $0 \leq s \leq t \leq T$
\begin{equation}
\label{towards-ac}
\begin{aligned}
\pairing{}{W^{2,d+\epsilon}(\Omega)}{\teta(t)-\teta(s)}{\bar\varphi} & =
-\int_s^t \io \chi_t
\teta\bar\varphi \dd x \dd r - \rho\int_s^t \io
\hbox{\rm div}(\ub_t) \teta\bar\varphi \dd x \dd r
 -\int_s^t \io \condu(\teta) \nabla \teta\nabla\bar\varphi \dd
x \dd r
\\
& \quad +
\int_s^t \io \left(g+\frac{\eps(\uu_t) \vism \eps(\uu_t)}2
+|\chi_t|^2\right) \bar\varphi  \dd x \dd r + \int_s^t \int_{\partial\Omega} h \bar\varphi   \dd S \dd r\,.
\end{aligned}
 \end{equation}
Then, we deduce from
 \eqref{towards-ac}  that
 $\teta$ is absolutely continuous
 with values in $W^{2,d+\epsilon}(\Omega)^*$.
 Thus, we recover the improved regularity \eqref{better-4-w}, and the improved formulation of the heat  equation
\begin{align}
&
\begin{aligned}
\label{eq-teta-proof}   &
\itt
\pairing{}{W^{2,d+\epsilon}(\Omega)}{\partial_t\teta}{\varphi} \dd s
  +\itt \io \chi_t
\teta\varphi \dd x \dd s + \rho\itt \io
\hbox{\rm div}(\ub_t) \teta\varphi \dd x \dd s \\
&\quad +\itt \io \condu(\teta) \nabla \teta\nabla\varphi \dd
x \dd s   =\itt \io \left(g+\frac{\eps(\uu_t) \vism \eps(\uu_t)}2
+|\chi_t|^2\right) \varphi  \dd x \dd s + \int_0^t \int_{\partial\Omega} h \varphi   \dd S \dd s
\\
 &\quad \text{for all }
\varphi\in C^0([0,T]; W^{2,d+\epsilon}(\Omega))  \text{ for some $\epsilon>0$}
\quad \text{and for all } t\in [0,T].
\end{aligned}
\end{align}
 Clearly, from \eqref{eq-teta-proof}  we obtain \eqref{eq-teta}  by differentiating in time.
\eo

The total energy \emph{equality} \eqref{total-enid-asequality},
holding \underline{for every} $0 \leq s \leq t \leq T$,
ensues from testing  \eqref{eq-teta} by $\varphi=1$, the momentum
balance \eqref{weak-momentum} by $\uu_t$, and the (pointwise)
$\chi$-equation \eqref{weak-phase} by $\chi_t$, adding the resulting
relations, and integrating in time. \QED


\section{\bf From the $p$-Laplacian to the Laplacian}
\label{s:final}

In this Section we prove a global-in-time existence result for a suitable \emph{entropic formulation} of   the initial-boundary value problem for system
\eqref{eq0}--\eqref{eqII},  in the case the $p$-Laplacian operator $-\dive(|\nabla\chi|^{p-1}\nabla\chi)$
is replaced by the 
 Laplacian $-\Delta\chi$, i.e.\  for  $p=2$,
 keeping the evolution \emph{unidirectional} (i.e., $\mu=1$).
  Hence, \eqref{eqII} rewrites as
\begin{equation}\label{eqII2}
\chi_t +\partial I_{(-\infty,0]}(\chi_t) -\Delta\chi  +W'(\chi) \ni - b'(\chi)\frac{\tensore
\elm \tensore}2 + \teta \quad\hbox{in }\Omega \times
(0,T).
\end{equation}
We restrict, apparently for technical reasons (which however we cannot bypass), to the irreversible case $\mu=1$.
The main idea of the technique consists in passing to the limit as $\delta\searrow0$ in the following approximation of \eqref{eqII2}
\begin{equation}\label{eqIIdelta}
\chi_t +\partial I_{(-\infty,0]}(\chi_t) -\Delta\chi -\delta\dive(|\nabla\chi|^{p-1}\nabla\chi) +W'(\chi) \ni - b'(\chi)\frac{\tensore
\elm \tensore}2 + \teta \quad\hbox{in }\Omega \times
(0,T).
\end{equation}
Indeed, \RRRNEW under suitable conditions \EEE
  the existence result in   Thm.\  \ref{teor3} applies to the initial-boundary value problem for
 system  \eqref{eq0}--\eqref{eqI}, \eqref{eqIIdelta}, %
  with $p>d$ (supplemented with
 the boundary conditions \eqref{intro-b.c.}), yielding
  the existence of global-in-time entropic solutions for fixed
  $\delta>0$.
   In this entropic formulation we will  then
  pass to the limit as $\delta\searrow0$,  
  recovering  an  existence result for the case $p=2$. \RRRNEW In what
  follows, we will  in fact work under a set of assumptions
 suited to the limit passage as $\delta\searrow0$, but
  slightly
  weaker than the ones necessary to apply  the existence Theorem
  \ref{teor3}, cf.\ e.g.\ Remark \ref{omegaLip}. \EEE

  \par
   Let us  now  state  the notion of entropic solution for
   the limit
   system as $\delta \to 0$.  We mention in advance that the solution concept introduced below 
   is weaker
   than the one we have obtained in the case $p>d$ (cf.\ Definition \ref{prob-rev1}). In fact,
   the total energy inequality holds true  only on $(0,t)$ (cf.\  \eqref{total-enidSec6} below), and not on a generic interval
   $(s,t)$, and so does the \berdue energy-dissipation \erdue energy inequality in the weak formulation of the equation for $\chi$. Moreover, the momentum equation is no longer formulated pointwise a.e.\ in $\Omega \times (0,T)$, but   in $H^{-1}(\Omega;\R^d)$, a.e.\ in time, only. Let us also anticipate that we will confine to initial data $\chi_0\in H^1(\Omega)$
   such that $\chi_0 \geq 0$ a.e.\ in $\Omega$ (which gives $\widehat{\beta}(\chi_0) \in L^1(\Omega)$ as in
   \eqref{datochi}) and, at the same time,  $\chi_0 \leq 1$ a.e.\ in $\Omega$. This
   and the irreversible character of the evolution
   will ensure that $\chi \in [0,1]$ a.e.\ in $\Omega \times (0,T)$, in accord with  the  physical meaning of $\chi$.

\begin{definition}[Entropic solutions to the irreversible system with $p=2$]
\label{prob-rev1Sec6}
{\sl
 Given initial data    \RRRNEW $(\teta_0,\uu_0,\vv_0,\chi_0)$
 such that $\teta_0$ fulfills \eqref{datoteta}, $(\uu_0,\vv_0) \in H_0^1(\Omega;\R^d) \times L^2(\Omega;\R^d)$, \EEE  and
$\chi_0$ such that
\begin{equation}
\label{new-dato-chi}
\chi_0 \in H^1(\Omega), \quad 0 \leq \chi_0 \leq 1 \text{ a.e.\ in } \Omega,
\end{equation}
we call a triple $(\teta,\uu,\chi)$ an \emph{entropic solution} to the  Cauchy  problem
for system \eqref{eq0}--\eqref{eqI}, \eqref{eqII2}  with the
boundary conditions \eqref{intro-b.c.},
if
\begin{align}
\label{reg-tetaSec6}  & \teta \in L^2(0,T; H^1(\Omega))\cap L^\infty(0,T;L^1(\Omega))
\,,
\\
& \label{reg-uSec6} \uu \in    H^1(0,T;\boZ) \cap W^{1,\infty}
(0,T;L^2(\Omega;\RR^d))  \cap H^2 (0,T;H^{-1}(\Omega; \RR^d))\,,
\\
& \label{reg-chiSec6} \chi \in L^\infty (0,T;H^1 (\Omega)) \cap H^1
(0,T;L^2 (\Omega)),
\end{align}
$(\teta,\uu,\chi)$ complies with
 the initial conditions   \eqref{iniu}--\eqref{inichi},  
and with  the {\em entropic formulation} of \eqref{eq0}--\eqref{eqI}, \eqref{eqII2} consisting of
\begin{itemize}
\item[-]
the \emph{entropy}  inequality  \eqref{entropy-ineq};
\item[-] the \emph{total energy inequality}  for almost all $t \in (0,T]$:
\begin{equation}
\label{total-enidSec6} \mathscr{E}(\teta(t),\uu(t), \uu_t(t), \chi(t))
\leq
 \mathscr{E}(\teta_0,\uu_0, \vv_0, \chi_0)  + \int_0^t
\int_\Omega g \dd x \dd r
+ \int_0^t
\int_{\partial\Omega} h \dd S \dd r
  + \int_0^t \int_\Omega \mathbf{f} \cdot
\mathbf{u}_t \dd x \dd r \,,
\end{equation}
where
\begin{equation}
 \label{total-energySec6}
 \mathscr{E}(\teta,\uu,\uu_t,\chi):= \int_\Omega \teta \dd x +
  \frac12 \int_\Omega |\uu_t |^2 \dd x + \frac12\bilh{b(\chi(t))}{\uu(t)}{\uu(t)}
+\frac12 \int_\Omega |\nabla \chi|^2 \dd x + \int_\Omega W(\chi) \dd
x\,;
 \end{equation}
\item[-] the momentum equation
\begin{equation}
\label{weak-momentumSec6}
\ub_{tt}+\opj{a(\chi)}{\ub_t}+\oph{b(\chi)}{\uu} +
\ciro( \teta) =\mathbf{f} \quad
\text{  in }  H^{-1}(\Omega;\R^d)  \quad \aein \, (0,T),
\end{equation}
\item[-] the weak formulation of \eqref{eqII2}, viz.\
\begin{align}
\label{constraint-chitSec6}
&
\chi_t(x,t) \leq 0 \qquad \foraa\, (x,t) \in \Omega \times (0,T),
\\
 \label{ineq-chiSec6}
 &
\begin{aligned}
  \int_\Omega  \Big( \chi_t (t) \psi     +
\nabla\chi(t)\cdot \nabla \psi  + \xi(t) \psi +
\gamma(\chi(t)) \psi 
 & + b'(\chi(t))\frac{\varepsilon(\ub(t))
\elm\varepsilon(\ub(t))}{2}\psi  -\teta(t) \psi \Big)
\,
\mathrm{d}x 
  \geq 0 \\ &  \text{for all }  \psi
\in 
W_-^{1,2}(\Omega)\cap L^\infty(\Omega), \quad \foraa\, t \in (0,T),
\end{aligned}
\\
&
\nonumber
\text{where $\xi \in
\partial I_{[0,+\infty)}(\chi)$ in the  sense that}
\\
&
\label{xi-defSec6} \xi \in L^1(0,T;L^1(\Omega)) \qquad\text{and}\qquad
\pairing{}{W^{1,2}(\Omega)}{\xi(t)}{\psi-\chi(t)} \leq 0 \  \
\forall\, \psi \in W_+^{1,2}(\Omega)\cap L^\infty(\Omega), \ \foraa\, t \in (0,T),
\\
&
 \nonumber \text{ as well as  the energy
inequality for all $t \in (0,T]$:}
\\
&
\label{energ-ineq-chiSec6}
\begin{aligned}
 \int_0^t   \int_{\Omega} |\chi_t|^2 \dd x \dd r  & +\io\left(
 \frac12 |\nabla\chi(t)|^2 +  W(\chi(t))\right)\dd x\\ & \leq\io \left(
 \frac12|\nabla\chi_0|^2+ W(\chi_0)\right)\dd x
  +\int_0^t  \int_\Omega \chi_t \left(- b'(\chi)
  \frac{\varepsilon(\ub)\elm\varepsilon(\ub)}2
+\teta\right)\dd x \dd r.
\end{aligned}
\end{align}
\end{itemize}
}
\end{definition}

We are in the position now to state the main existence result of this section.
\begin{theorem}[Existence of entropic solutions, $\mu=1$ and $p=2$]
\label{teorSec6}  Let $\Omega$ be a bounded connected domain with \emph{Lipschitz boundary}. Assume  \textbf{Hypotheses (I)--(III)}  with
\begin{equation}
b'(x)\geq 0
\quad \text{for all } x \in \R,
\end{equation}
 and, in addition, \textbf{Hypothesis (IV)}
(i.e., $\widehat{\beta}=I_{[0,+\infty)}$), as well as
 conditions
\eqref{bulk-force}--\eqref{datou}   on the data $\mathbf{f},$ $g$, $h$,
$\teta_0,$ $\uu_0,$ $\vv_0,$  and \eqref{new-dato-chi} on
 $\chi_0$.
 Then, there exists an entropic solution (in the sense of Definition \ref{prob-rev1Sec6})
  $(\teta,\uu,\chi)$ to the initial-boundary value problem for system
\eqref{eq0}--\eqref{eqI}, \eqref{eqII2},  \beo such that
$\log(\teta)$ complies with \eqref{BV-log},
$\xi$ in \eqref{xi-defSec6} is given by \eqref{xi-specific} and
$\teta$ satisfies  the strict positivity property  \eqref{strictpos}.\eo
 \end{theorem}

\begin{remark}\label{omegaLip}
Let us note that in Thm.~\ref{teorSec6} we are able to deal with the case of a Lipschitz domain $\Omega$ and we do not need  $\mathrm{C}^2$-regularity
of $\Omega$ \eqref{smoothness-omega}. The latter condition  was exploited in the previous sections in order to perform the
elliptic regularity estimate on $\uu$ (cf.\ the {\em Fifth estimate} \eqref{palla}), which is not carried out here. Indeed the regularity requirement \eqref{reg-uSec6}  on $\uu$ we ask for in Definition~\ref{prob-rev1Sec6},  and prove in Theorem \ref{teorSec6},  is weaker than the one prescribed in Section~\ref{s:main} (cf., e.g.,  \eqref{reg-u}).
Moreover, for the same reason, in this case we could also consider more general boundary conditions on $\uu$ than the homogeneous Dirichlet \eqref{intro-b.c.}: for example mixed Dirichlet-Neumann conditions could be taken into account,  without any restriction on the geometry of the domain.
\end{remark}

\begin{proof}
 Let $(\teta_\delta,\uu_\delta,\chi_\delta)$ be a  suitable  family of
 \emph{entropic}
solutions to
 the initial-boundary value problem for
\eqref{eq0}--\eqref{eqI}, supplemented with initial data $(\teta_0,\uu_0,\vv_0)$ fulfilling
\eqref{datoteta}--\eqref{datou}, and with a sequence of
data $(\chi_0^\delta)_\delta$ such that
\begin{equation}
\label{conv-ini-chi}
(\chi_0^\delta)_\delta \subset W^{1,p}(\Omega),\quad
0 \leq \chi_0^\delta(x) \leq 1 \text{ for all  $x \in \Omega$  for all } \delta>0,  \quad \chi_0^\delta \to
 \chi_0 \text{ in } H^1(\Omega).
\end{equation}
  Observe that we cannot rigorously perform on the entropic formulation of \eqref{eq0}--\eqref{eqI}
 the a priori estimates in Section \ref{s:aprio}. Therefore we need to confine the discussion only to the  entropic solutions which arise from the time-discretization scheme set up in Sec.\ \ref{s:time-discrete}.
 In the present framework (i.e.\ with $p=2$ and $\mu=1$, and no upper bound on $\kappa$, cf.\ Hypothesis (V)),
 the a priori estimates for the time-discrete solutions
in Prop.\ \ref{prop:discrete-aprio} are inherited in the time-continuous limit by the entropic solutions, with the exception of those corresponding to the \emph{Fifth}, the \emph{Seventh}, and the \emph{Eighth a priori estimates} in Sec.\ \ref{s:aprio}, cf.\ also Remark \ref{rmk:role-p-Lapl}.
\beo Concerning the \emph{Sixth estimate}, as pointed out in Sec.\ \ref{ss:3.2} we are only able to render a surrogate of it (i.e.\ \eqref{weaker-sixth}) on the time-discrete level. Still, this provides  sufficient information to pass to the limit, cf.\ Lemma \ref{l:compactness}. We shall exploit this
also within the present proof.\eo

\par
The convergences from Lemma \ref{l:compactness} combined with lower semicontinuity arguments indeed ensure that
the strict positivity of $\teta_\delta$ (cf.\ \eqref{teta-pos}), as well as
estimates
 \eqref{est1}, \eqref{crucial-est3.2}, \eqref{necessary-added},
  \eqref{est5}, \beo\eqref{est5-added}, \eo
 \eqref{est6}, hold with constants uniform w.r.t.\ $\delta$. Moreover, combining the fact that $\widehat\beta = I_{[0,+\infty)}$ with the unidirectional character of the evolution
 and with the fact that $\chi_\delta (0) = \chi_0^\delta \in [0,1]$ on $\Omega$,
  we infer that
 \begin{equation}
 \label{chi-infty-sec6}
 \exists\, C>0 \ \ \forall\, \delta>0\, : \qquad
 \| \chi_\delta \|_{L^\infty (Q)}\leq C.
   \end{equation}
   Therefore, repeating the compactness arguments in the proof of Lemma \ref{l:compactness}, based on the compactness results in \cite{simon}
   \beo (cf.\ also Theorem \ref{th:mie-theil} in the Appendix), \eo
 for every vanishing sequence $\delta_k \down 0$ as $k \to \infty$ there exist a not relabeled subsequence and
a triple $(\teta,\uu,\chi)$, along which
 there  holds as $k \to \infty$:
\begin{align}\label{convteta}
&\teta_{\delta_k}\weaksto\teta \quad \hbox{in }  L^2(0,T;H^1(\Omega))\,,\\  
\label{convu1}
& \uu_{\delta_k}\weaksto \uu \quad \hbox{in } H^2(0,T; H^{-1}(\Omega; \RR^d))\cap W^{1,\infty}(0,T;L^2(\Omega; \RR^d))\cap H^1(0,T;H^1(\Omega; \RR^d))\,,\\
\label{convu2}
&\dt \uu_{\delta_k}\to \dt \uu \quad \hbox{in } L^2(0,T; L^2(\Omega; \RR^d))\,,\\
\label{convchi}
&\chi_{\delta_k}\weaksto\chi\quad\hbox{in }H^1(0, T; L^2(\Omega))\cap L^\infty (0,T; H^1(\Omega))\,,\\
\label{convchis}
&\chi_{\delta_k}\to\chi\quad\hbox{in }L^h(\Omega\times (0,T))\quad
\text{for all } h\in [1, +\infty)\,,\\
\label{convlog} &\log(\teta_{\delta_k})\to  \log(\teta) \quad\hbox{in
}L^2(0,T;L^s(\Omega))\quad  \text{for all $s\in (1,6)$ for $d=3$ and for all $s\in (1, +\infty)$ for
$d=2$}\,,
\\
&
\label{convtetas}
 \teta_{\delta_k}\to \teta\quad \hbox{in $L^h(\Omega\times (0,T))$,
for every $h\in [1,8/3) $ for $d=3$ and $h\in [1, 3)$ if $d=2$,}
\end{align}
and in addition $\teta \in L^\infty (0,T; L^1(\Omega))$.
%

Now, in order to pass to the limit as $\delta\searrow0$ we need to prove,  in addition, that
$\partial_t \uu_{\delta_k} \to \partial_t \uu$ \emph{strongly} in $L^2 (0,T; H^1(\Omega;\R^d))$.
 Observe that, in the case of the $p$-Laplacian regularization for $\chi$,
  we were able to prove an additional
 the strong convergence for  (the sequence approximating)
$\partial_t \uu$ in $L^2 (0,T; H^1(\Omega;\R^d))$. Our argument
resulted  from  compactness arguments, relying on the \emph{Fifth a priori estimate} (i.e.\ the elliptic regularity estimate on $\uu$). The latter is no longer at our disposal, now.  The argument we will develop in the following lines is instead direct, and    strongly based on the irreversible character of our system.

\paragraph{\bf Strong convergence of  $\partial _t \uu_{\delta_k}$  in $L^2 (0,T; H^1 (\Omega;\R^d))$.}
Let us test the
 weak formulation
\eqref{weak-momentum} of  momentum
equation
fulfilled by the approximate solutions $(\teta_{\delta_k},\uu_{\delta_k},\chi_{\delta_k})_k$,
 by
$\dt(\uu_{\delta_k}-\uu)$, where $\uu$ is the limit of $(\uu_{\delta_k})_k$ as in
\eqref{convu1}--\eqref{convu2}. We get
\begin{align}\no
&0=\itt\io \partial_{tt}^2\uu_{\delta_k}\dt(\uu_{\delta_k}-\uu)\dd x \dd s+\itt \bilj{a(\chi_{\delta_k})}{\dt \uu_{\delta_k}}{\dt(\uu_{\delta_k}-\uu)}\dd s\\
\no &+\itt \bilh{b(\chi_{\delta_k})}{\uu_{\delta_k}}{\dt(\uu_{\delta_k}-\uu)}\dd
s- \rho\itt\io \teta_{\delta_k}\dive(\dt(\uu_{\delta_k}-\uu))\dd x\dd s-\itt \io {\bf
f}\dt(\uu_{\delta_k}-\uu)\dd x\dd s=:\sum_{i=1}^{5} I_i\,.
\end{align}
Let us now deal separately with the single integrals  $I_1, \ldots, I_5$:
\begin{align}\no
I_1:&= \itt\io \partial_{tt}^2\uu_{\delta_k}\dt(\uu_{\delta_k}-\uu)\dd x \dd s=\itt\io \partial_{tt}^2(\uu_{\delta_k}-\uu)\dt(\uu_{\delta_k}-\uu)\dd x \dd s \\
\no
&\quad+\itt \pairing{}{H^1(\Omega;\RR^d)}{ \partial_{tt}^2\uu}{\dt(\uu_{\delta_k}-\uu)} \dd s\\
\no & = \frac12 \|\dt(\uu_{\delta_k}-\uu)(t)\|_{L^2(\Omega;\RR^d)}^2-\frac12
\|\dt(\uu_{\delta_k}-\uu)(0)\|_{L^2(\Omega;\RR^d)}^2+\itt
\pairing{}{H^1(\Omega;\RR^d)}{\partial_{tt}^2\uu}{\dt(\uu_{\delta_k}-\uu)}\dd s\,,
\end{align}
and the third integral tends to 0 when ${\delta_k}\searrow0$ due to
\eqref{convu1}. Moreover,
\begin{align}\no
I_2:&=\itt \bilj{a(\chi_{\delta_k})}{\dt \uu_{\delta_k}}{\dt(\uu_{\delta_k}-\uu)}\dd s\\
\no &=\itt \bilj{a(\chi_{\delta_k})}{\dt(\uu_{\delta_k}-\uu)}{\dt(\uu_{\delta_k}-\uu)}\dd
s+\itt \bilj{a(\chi_{\delta_k})}{\dt \uu}{\dt(\uu_{\delta_k}-\uu)}\dd s\,.
\end{align}
 Now, observe that
\begin{equation}
\label{strong-a-chi}
a(\chi_{\delta_k}) \partial_t \uu \to a(\chi)\partial_t \uu \qquad \text{in
$L^2 (0,T; H^1(\Omega;\R^d))$.}
\end{equation}
 This follows from the fact that
$a(\chi_{\delta_k}) \uu_t \to a(\chi)\uu_t$ and
$a(\chi_{\delta_k})\eps(\uu_t)\to a(\chi)\eps(\uu_t)$ a.e.\ in $\Omega \times (0,T)$,
in view of convergence \eqref{convchis} and of the continuity of $a$. Moreover, also due to
\eqref{chi-infty-sec6}, we have  that
$\|a(\chi_{\delta_k})\uu_t\|_{H^1(\Omega;\R^d)}\leq
C\|\uu_t\|_{H^1(\Omega;\R^d)}$ for a constant  independent of  $k\in \N$. Therefore, using the Lebesgue theorem
 the desired convergence
 \eqref{strong-a-chi} ensues.
  This implies that   $ \itt \bilj{a(\chi_{\delta_k})}{\dt \uu}{\dt(\uu_{\delta_k}-\uu)}\dd s$
 tends to 0 when ${\delta_k}\searrow0$, due to
\eqref{convu1}. 
Integrating by parts in time, we get
\begin{align}\no
I_3:&=\itt \bilh{b(\chi_{\delta_k})}{\uu_{\delta_k}}{\dt (\uu_{\delta_k}-\uu)}\dd s\\
\no
&=\itt \bilh{b(\chi_{\delta_k})}{(\uu_{\delta_k}-\uu)}{\dt(\uu_{\delta_k}-\uu)}\dd s+\itt \bilh{b(\chi_{\delta_k})}{\uu}{\dt(\uu_{\delta_k}-\uu)}\dd s\\
\no
&=-\itt\io b'(\chi_{\delta_k})\dt \chi_{\delta_k}\frac{\eps(\uu_{\delta_k}-\uu)\elm\eps(\uu_{\delta_k}-\uu)}{2}\dd x \dd s+\frac12\bilh{b(\chi_{\delta_k}(t)}{(\uu_{\delta_k}-\uu)(t)}{(\uu_{\delta_k}-\uu)(t)}\\
\no
&\quad-\frac12\bilh{b(\chi_{\delta_k}(0))}{(\uu_{\delta_k}-\uu)(0)}{(\uu_{\delta_k}-\uu)(0)}+
\itt\bilh{b(\chi_{\delta_k})}{\uu}{\dt(\uu_{\delta_k}-\uu)}\dd
s\,,
\end{align}
where the last integral tends to 0 (this can be shown arguing in the same way as  for
 the last term contributing to
 $I_2$), while the first integral is non-negative
due to the fact that $\dt \chi_{\delta_k}\leq 0$ a.e.\ on $\Omega \times (0,T)$ and that $b'\geq 0$. 
 This is the point where  we exploit the unidirectional character of the system  (i.e. $\mu=1$).
Finally,
\[
I_4:=-\itt\io \teta_{\delta_k}\eps(\dt(\uu_{\delta_k}-\uu))\dd x\dd s\to 0\,,\quad
I_5:=-\itt \io {\bf f}\dt(\uu_{\delta_k}-\uu)\dd x\dd s\to 0\,,
\]
as ${\delta_k}\searrow0$, due to the convergences \eqref{convu1},
\eqref{convtetas}, as well as assumption \eqref{bulk-force} on ${\bf
f}$. Ultimately, we  get
\[
\|\dt(\uu_{\delta_k}-\uu)(t)\|_{L^2(\Omega; \RR^d)}^2+\itt
\bilj{a(\chi_{\delta_k})}{\dt(\uu_{\delta_k}-\uu)}{\dt(\uu_{\delta_k}-\uu)}\dd
s+\bilh{b(\chi_{\delta_k}(t)}{(\uu_{\delta_k}-\uu)(t)}{(\uu_{\delta_k}-\uu)(t)}\to 0
\]
as ${\delta_k}\searrow0$, which entails
\begin{equation}\label{convus}
\uu_{\delta_k}\to \uu \quad \hbox{strongly in
} W^{1,\infty}(0,T;L^2(\Omega;\RR^d))\cap H^1(0,T; H^1(\Omega; \RR^d))\,.
\end{equation}

\paragraph{\bf Conclusion of the proof.} Using this strong convergence, we can now pass to the limit  as $k \to \infty$ in
 the  energy-dissipation inequality
 \eqref{ineq-chi}  featuring  in the weak formulation of the equation for $\chi_{\delta_k}$
 as follows. 
  We have to identify the weak limit of
 \begin{equation}
 \label{xi-deltak}
 \xi_{\delta_k}(x,t) = -   \mathcal{I}_{\{\chi_{\delta_k}=0\}}  (x,t) \left(\gamma(\chi_{\delta_k}(x,t)) + b'(\chi_{\delta_k}(x,t)) \frac{\eps(\uu_{\delta_k}(x,t)) \elm(x)  \eps(\uu_{\delta_k}(x,t))  }{2}  - \teta_{\delta_k}(x,t)\right)^+.
 \end{equation}
 First of all note that  $(\mathcal{I}_{\{\chi_{\delta_k}=0\}})_k$   is bounded in $L^\infty(Q)$ independently of $k\in \N$. Hence, we can select a subsequence  $(\mathcal{I}_{\{\chi_{\delta_k}=0\}})_k$  weakly star converging  in $L^\infty(Q)$  to some   $ \mathcal{J}$.   Observe that we cannot establish that $ \mathcal{J} =  \mathcal{I}_{\{\chi=0\}}$.
 On the other hand, it follows from the previously proved convergences that
$(\gamma(\chi_{\delta_k}) + b'(\chi_{\delta_k}) \frac{\eps(\uu_{\delta_k}) \elm  \eps(\uu_{\delta_k})  }{2}  - \teta_{\delta_k})^+$ strongly converges in $L^1(Q)$
to $(\gamma(\chi) + b'(\chi) \frac{\eps(\uu) \elm  \eps(\uu)  }{2}  - \teta)^+$. Hence we identify
\begin{equation}
\label{xi-limite} \xi =- \mathcal{J} (x,t)(\gamma(\chi(x,t)) +
b'(\chi(x,t)) \frac{\eps(\uu(x,t)) \elm \eps(\uu(x,t)) }{2} -
\teta(x,t))^+
\end{equation}
and observe that $\xi_{\delta_k} \weakto \xi $ in $L^1(Q). $
 Then, integrating \eqref{ineq-chi}$_{\delta_k}$ from 0 to   $T$ and passing to the limit as $k\to\infty$, using the fact that or all $\psi\in L^p(0,T;W^{1,p}_-(\Omega))\cap L^\infty(Q)$
\[
\left|\int_0^T\int_\Omega\delta_k|\nabla\chi_{\delta_k}|^{p-2}\nabla\chi_{\delta_k}\cdot\nabla\psi \dd x \dd t \right|\leq \delta_k\|\nabla\chi_{\delta_k}\|_{L^{p-1}(Q;\R^d)}^{p-1}\|\nabla\psi\|_{L^p(Q;\R^d)}\to0\,,
\]
we get
\begin{equation}\label{conv-chineq}
\int_0^T\int_\Omega  \Big( \chi_t (t) \psi     +
\nabla\chi(t)\cdot \nabla \psi  +
\gamma(\chi(t)) \psi 
 + b'(\chi(t))\frac{\varepsilon(\ub(t))
\elm\varepsilon(\ub(t))}{2}\psi  -\teta(t) \psi \Big)
\,
\mathrm{d}x \,
\mathrm{d}t
  \geq -\int_0^T\int_\Omega\xi(t) \psi\,
\,
\mathrm{d}x\mathrm{d}t\,,
\end{equation}
for all $\psi\in L^p(0,T;W^{1,p}_-(\Omega))\cap L^\infty(Q)$, where
$\xi$ is defined in  \eqref{xi-limite}.  From
\eqref{conv-chineq}, we get \eqref{ineq-chiSec6}.

It remains to show that $\chi$ complies with the variational
inequality \eqref{xi-defSec6}. To do so, we have to pass to the
limit in \eqref{xi-def}$_{\delta_k}$, whence   we have
\[
\int_0^T\left( \int_\Omega
\xi_{\delta_k}(\psi-\chi_{\delta_k}(t))\,\mathrm{d}x \right) \zeta(t)\,\mathrm{d}t
\geq 0 \quad \text{for all   $\psi\in W^{1,p}_+ (\Omega) $  
and all
$\zeta\in L^\infty(0,T)$ with $\psi,\,\zeta\geq 0$.}
\]
Observe that the two weak convergences $\chi_{\delta_k} \weaksto
\chi$ in $L^\infty (Q)$ and $\xi_{\delta_k} \weakto \xi $ in $L^1(Q)
$ do not allow for a direct limit passage in the term $\iint_Q
\xi_{\delta_k} \chi_{\delta_k} \zeta \dd x \dd t$, which equals zero
for all $k\in \N$ due to \eqref{xi-deltak}. Indeed, we need to
argue in a more refined way. It follows from \eqref{convchis} that
$\chi_{\delta_k}$ converges almost uniformly to $\chi$ in $Q$, i.e.
for every $\epsilon>0$ there exists $Q_\epsilon \subset Q$ such that
$|Q\setminus Q_\epsilon|<\epsilon$ and $\chi_{\delta_k} \to \chi$
uniformly on $Q_\epsilon$. The latter property implies that
\begin{equation}
\label{speriamo-bene}
  \mathcal{J} \equiv 0 \ \text{ on } \ Q_\epsilon \cap \{  \mathcal{I}_{\{\chi =0\}}  \equiv 0 \} \,.
%
\end{equation}
Indeed, $\mathcal{I}_{\{\chi =0\}} (x,t) = 0$   implies $\chi(x,t) \neq 0$.
Since  $\chi_{\delta_k}$ converges to $\chi$ uniformly on
$Q_\epsilon$, there exists an index $\bar k $, independent of $(x,t)$, such that for all $k
\geq \bar k$, 
 $\chi_{\delta_k}(x,t)
\neq 0$, hence  $\mathcal{I}_{\{\chi_{\delta_k} =0\}} (x,t) =0$.
With this argument we conclude that
$\mathcal{I}_{\{\chi_{\delta_k} =0\}} \equiv 0$  on  $ Q_\epsilon
\cap \{   \mathcal{I}_{\{\chi =0\}}  \equiv 0 \} $, whence
\eqref{speriamo-bene}.
%
 It
follows from \eqref{speriamo-bene} and \eqref{xi-limite} that
\[
\xi(x,t) \chi(x,t) =0 \quad \foraa\, (x,t) \in Q_\epsilon,
\text{
whence
}
\iint_{Q_\epsilon}\xi(x,t)  \chi(x,t) \zeta(t)\dd x \dd t =0\,.
\]
On the other hand, using the properties of the Lebesgue integral we
have that
\[
\forall\, \eta >0  \ \ \exists \epsilon=\epsilon_\eta>0 \ \, : \ \
|Q \setminus Q_\epsilon|< \epsilon \ \ \Rightarrow \ \ \iint_{Q
\setminus Q_\epsilon} |\xi(x,t)  \chi(x,t) \zeta(t)| \dd x \dd t
<\eta.
\]
Therefore we conclude that
\[
\forall\, \eta >0  \qquad \left|\iint_Q \xi(x,t)  \chi(x,t) \zeta(t) \dd
x \dd t\right| <\eta,
\]
i.e.\
\[
\iint_Q \xi(x,t)  \chi(x,t) \zeta(t) \dd x \dd t=0 = \lim_{k \to
\infty}\iint_Q \xi_{\delta_k} \chi_{\delta_k} \zeta \dd x \dd t
\]
Hence
\begin{align}\no
0\leq
\iint_Q\xi_{\delta_k}(\psi-\chi_{\delta_k})\zeta\,\mathrm{d}x\,\mathrm{d}t
 \to\iint_Q \xi(\psi-\chi)\zeta\,\mathrm{d}x\,\mathrm{d}t=\int_0^T\left(\int_\Omega \xi(\psi-\chi(t))\,\mathrm{d}x\right)\zeta(t)\,\mathrm{d}t,
\end{align}
 which implies
\[
\int_\Omega\xi(t)(\psi-\chi(t))\, \mathrm{d}x  \geq 0  \quad\hbox{for
a.e. }t\in (0,T) \qquad \text{for all } \psi \in   W_+^{1,p}(\Omega).
\]
 With  a density argument we get \eqref{xi-defSec6} for all  $\psi \in  W_+^{1,2}(\Omega) \cap L^\infty (\Omega)$.

Convergences \eqref{convteta}--\eqref{convtetas} also guarantee the passage to the limit in the momentum equation, whence \eqref{weak-momentumSec6}.

Finally, we pass to the limit in the entropy inequality \eqref{entropy-ineq} and in the total energy inequality
\eqref{total-enid}
by
the very same  compactness/lower semicontinuity arguments as in the proof of Theorem \ref{teor3}, thus deducing
\eqref{entropy-ineq} and  the total energy inequality \eqref{total-enidSec6} on the generic interval $(0,t)$.
\QED


\begin{remark}
Notice that,
 we have been able to obtain
 the energy inequalities  \eqref{energ-ineq-chiSec6}  and  \eqref{total-energySec6}
 only on  intervals of the type $(0,t)$, and not on the generic interval $(s,t)\subset (0,T)$,
  due to the weak convergence of $(\nabla\chi_{\delta_k})$ in $L^2(Q;\RR^d)$, which does not yield  the pointwise-in-time convergence
  required to take the limit of the right-hand sides of \eqref{energ-ineq-chi}  and  \eqref{total-enid}. It is an open problem to improve the convergence of
   $(\nabla\chi_{\delta_k})$ to a strong one.

This limit passage also reveals that the notion of entropic solution enjoys stability properties. It   seems to be  the right one in the
present framework, and, apparently, the entropy inequality cannot be improved to  a suitable variational formulation of the heat equation like in the case of Theorem \ref{teor1},  at least with these techniques.
%
\end{remark}

\end{proof}

\RRRNEW
\appendix
\section{Auxiliary compactness results}
\label{s:a-1}
 \noindent

 The main compactness result of this Appendix, Theorem
 \ref{th:mie-theil} below,
 hinges
 on a compactness argument drawn from the theory of
  \emph{parameterized} (or \emph{Young}) measures with values in an infinite-dimensional  space.

 Hence,
 for the reader's convenience, we preliminarily  collect here the definition
 of Young measure with values in a \underline{reflexive} Banach space $X$. We then
 recall
   the Young measure
 compactness result from \cite{MRS2013}, which was proved in  \cite{Rossi-Savare06} in the case when $X$ is a Hilbert space, extending to the frame of the
 weak topology
classical results within Young measure theory (see  e.g.\ \cite[Thm.\,1]{Balder84}, \cite{Ball89}
\cite[Thm.\,16]{Valadier90}).

We start by fixing some
\begin{notation}
\label{not-append-1}
\upshape
Given an interval $I \subset \R$, we  denote by $\mathscr{L}_{I}$
the $\sigma$-algebra of the Lebesgue measurable subsets of $I$ and,
given
   a reflexive Banach space $X$,
 by $\mathscr B(X)$ its Borel $\sigma$-algebra.
 \end{notation}
 \begin{definition}[\bf (Time-dependent) Young measures]
  \label{parametrized_measures}
  A \emph{Young measure} in the space $X$
  is a family
  $\bfmu:=\{\mu_t\}_{t \in (0,T)} $ of Borel probability measures
  on $X$
  such that the map on $(0,T)$
\begin{equation}
\label{cond:mea} t \mapsto \mu_{t}(A) \quad \mbox{is}\quad
{\mathscr{L}_{(0,T)}}\mbox{-measurable} \quad \text{for all } A \in
\mathscr{B}(X).
\end{equation}
We denote by $\mathscr{Y}(0,T; X)$ the set of all Young
measures in $X $.
\end{definition}

The following result
subsumes part of the statements of  \cite[Theorems A.2, A.3]{MRS2013}: its
crucial finding  for our purposes concerns the characterization of the limit points in the weak topology of
$L^p (0,T;X)$, $p \in (1,+\infty] $, of a bounded sequence $(\ell_n)_n \subset L^p (0,T;X)$. Every limit point arises as the barycenter
of the limiting Young measure $\bfmu=(\mu_t)_{t\in (0,T)}$
 associated with  (a suitable subsequence $(\ell_{n_k})_k$ of) $(\ell_n)_n$. In turn, for almost all $t\in (0,T)$
 the support of the measure $\mu_t$ is concentrated in the set of limit points of $(\ell_{n_k}(t))_k$ with respect to the weak topology of $X$.
This information will play a crucial role in the proof of Theorem \ref{th:mie-theil} ahead.
\begin{theorem}{\cite[Theorems A.2, A.3]{MRS2013}}
\label{thm.balder-gamma-conv}
Let $p>1$ and let
  $(w_n)_n \subset L^p
(0,T;X)$ be a  bounded sequence.
Then, there exist a subsequence $(w_{n_k})_k$ and a Young
measure $\bfmu=\{\mu_t\}_{t \in (0,T)}$ such that for a.a. $t \in
(0,T)$
\begin{equation}
\label{e:concentration}
\begin{gathered}
  \mbox{$ \mu_{t} $ is
      concentrated on
      the set
      $
      \bigcap_{p=1}^{\infty}\overline{\big\{w_{n_k}(t)\,: \ k\ge p\big\}}^{\mathrm{weak}{\text{-}X}}$}
      \end{gathered}
  \end{equation}
of the limit points of the sequence $(w_{n_k}(t))$ with respect to
the weak topology of $X$ and,
  setting
  \[
w(t):=\int_{X}
\omega \, \dd \mu_t (\omega)  \qquad \foraa\, t
\in (0,T)\,,
  \]
there holds
\begin{equation}
  \label{eq:35}
w_{n_k} \weakto w \ \ \text{ in $L^p (0,T;X)$} \qquad \text{as } k \to \infty
\end{equation}
with $\weakto$ replaced by $\weaksto$ if $p=\infty$.
\end{theorem}

The  statement  of Theorem
\ref{th:mie-theil} ahead
features two  reflexive Banach spaces $V$ and $Y$.  Further, we will use
 the following
 \begin{notation}
\label{not-append-2}
\upshape
We denote by
 $\overline{B}_{1,Y}(0)$  the closed unitary ball in $Y$, and we will work with the space
 \begin{equation}
 \label{measurable-allt}
 \mathrm{B}([0,T];Y^*) := \{ \ell: [0,T]\to Y^*\, : \  \text{measurable, such that } \ell(t) \text{ is defined at every } t \in [0,T]\}.
 \end{equation}
 Moreover, for given $\ell \in
\mathrm{B}([0,T];Y^*)$, $\varphi \in Y$, and $[a,b]\subset [0,T]$, we set
\begin{equation}
\label{var-notation}
\begin{aligned}
\mathrm{Var}(\pairing{}{Y}{\ell}{ \varphi}; [a,b] ) : =  \sup \{  \sum_{i=1}^J
\left |\pairing{}{Y}{\ell(\sigma_{i})}{ \varphi} - \pairing{}{Y}{\ell(\sigma_{i-1})}{ \varphi}  \right|
\, :  \
a =\sigma_0 < \sigma_1 < \ldots < \sigma_J =b \} \,.
\end{aligned}
\end{equation}
\end{notation}
We are now in the position to state and  prove the main result of this section, combining Thm.\ \ref{thm.balder-gamma-conv}
with ideas from  \cite[Thm.\ 6.1]{MieTh04}.
\begin{theorem}
\label{th:mie-theil}
Let $V$ and $Y$ be two (separable) reflexive Banach spaces  such that
$V \subset Y^*$ continuously. Let
 $(\ell_n)_n \subset L^p
(0,T;V) \cap \mathrm{B} ([0,T];Y^*)$ be  bounded  in $L^p
(0,T;V) $ and suppose in addition  that
\begin{align}
\label{ell-n-0}
&
\text{$(\ell_n(0))_n\subset Y^*$ is bounded},
\\
&
\label{BV-bound}
\exists\, C>0 \  \ \forall\, \varphi \in \overline{B}_{1,Y}(0)  \ \  \forall\, n \in \N\, : \quad
 \mathrm{Var}(\pairing{}{Y}{\ell_n}{ \varphi}; [0,T] )  \leq C.
\end{align}

Then, there exists  a subsequence
 $(\ell_{n_k})_k$ of $(\ell_n)_n$
 and a function $\ell \in L^p (0,T;V) \cap L^\infty (0,T; Y^*) $ 
 such that  as $k\to \infty$
 \begin{align}
 \label{weak-LpB}
 &
 \ell_{n_k} \weaksto \ell \quad \text{ in } L^p (0,T;V) \cap L^\infty (0,T;Y^*),
 \\
\label{weak-ptw-B}
&
\ell_{n_k}(t) \weakto \ell(t) \quad \text{ in } V\quad \foraa\, t \in (0,T).
\end{align}
\end{theorem}
\begin{proof}
We split the proof in two claims. For the first one, we closely follow the arguments from the proof of \cite[Thm.\ 6.1]{MieTh04}.

\underline{\textbf{Claim $1$:}}
\emph{Let $\mathcal{F} \subset \overline{B}_{1,Y}(0)$ be countable and  dense in
$ \overline{B}_{1,Y}(0)$. There exist a subsequence $(\ell_{n_k})_k$ of $(\ell_n)_n$,  and for every $\varphi \in \mathcal{F}$
a function
$\mathscr{L}_\varphi: [0,T] \to \R$
such that   for every $ \varphi \in \mathcal{F} $
\begin{equation}
\label{mie-th-conv}
\pairing{}{Y}{\ell_{n_k}(t)}{\varphi} \to \mathscr{L}_\varphi(t)
 \quad \text{as $k \to \infty$}
 \quad \text{for every } t \in [0,T].
\end{equation}
 }
 \\
With every  $\varphi \in \overline{B}_{1,Y}(0)$ we may associate the monotone functions
$\mathcal{V}_{n}^{\varphi}: [0,T] \to [0,+\infty)$
defined by
$\mathcal{V}_{n}^{\varphi}(t) :=  \mathrm{Var}(\pairing{}{Y}{\ell_n}{ \varphi}; [0,t] )  $
for every $t\in [0,T]$.
Let now $\mathcal{F} \subset \overline{B}_{1,Y}(0)$ be countable and dense and let us consider the family of functions
$(\mathcal{V}_{n}^{\varphi})_{n\in \N, \, \varphi \in \mathcal{F}}$.
 It follows from estimate \eqref{BV-bound} and  from Helly's principle, combined with a  diagonalization procedure based on the countability of
 $\mathcal{F}$, that
there exist a sequence of indexes  $(n_k)_k$ and for every $\varphi \in \mathcal{F}$ a  monotone function
$\mathcal{V}_\infty^{\varphi} : [0,T]\to [0,+\infty)$
 such that
\begin{equation}
\label{helly+diag}
\mathcal{V}_{n_k}^{\varphi}(t) \to \mathcal{V}_\infty^{\varphi}(t) \quad \text{as } k \to \infty \quad \text{for every } t \in [0,T].
\end{equation}
Since $\mathcal{V}_\infty^{\varphi} $ is monotone, it has an at most countable jump set $J_\varphi$. The set $J:= \cup_{\varphi \in \mathcal{F}}
J_\varphi$ is still countable, and therefore we may choose a sequence $\mathcal{T}: = (t_m)_m \subset [0,T]$, dense in $[0,T]$, such that
$J \subset \mathcal{T} $.

Observe that \eqref{ell-n-0} and \eqref{BV-bound} yield that
\begin{equation}
\label{sup-n-t}
\exists\, C>0 \ \ \forall\, n \in \N \ \forall\, t \in [0,T]\, : \quad \|\ell_n(t)\|_{Y^*} \leq C.
\end{equation}
Since $Y^*$ is reflexive, with a diagonal argument we may extract a further, not relabeled, subsequence such that for every $t \in \mathcal{T} = (t_m)_{m}$
\begin{equation}
\label{1st-convergence}
\ell_{n_k}(t)\weakto \tilde{\ell}(t)\qquad \text{  in $Y^*$ as $k \to \infty$}
\end{equation}
for some $\tilde{\ell}(t) \in Y^*$.
We now show that
for every $\varphi \in \mathcal{F}$ the map
$t \mapsto \pairing{}{Y}{\tilde{\ell}(t)}{\varphi}$ is weakly continuous at every point  $t \in \mathcal{T}\setminus  J$. Indeed,
for every $t_1,\, t_2 \in \mathcal{T}$ with $t_1 \leq t_2$  and for every $\varphi \in \mathcal{F}$ one has
\begin{equation}
\label{var-infty-lim}
\begin{aligned}
& \left| \pairing{}{Y}{\tilde{\ell}(t_1) - \tilde{\ell}(t_2)}{\varphi}  \right|  = \lim_{k\to \infty} \left| \pairing{}{Y}{\ell_{n_k}(t_1) - \ell_{n_k}(t_2)}{\varphi}\right|
\\ &\qquad  \leq  \lim_{k\to \infty}
\mathrm{Var}(\pairing{}{Y}{\ell_{n_k}}{ \varphi}; [t_1,t_2] ) =
 \lim_{k\to \infty}\mathcal{V}_{n_k}^{\varphi}(t_2) - \lim_{k\to \infty} \mathcal{V}_{n_k}^{\varphi}(t_1) =  \mathcal{V}_{\infty}^{\varphi}(t_2) -  \mathcal{V}_{\infty}^{\varphi}(t_1)\,.
 \end{aligned}
\end{equation}
Therefore, for every $\varphi \in \mathcal{F}$ one extends the map $t\mapsto \pairing{}{Y}{\tilde{\ell}(t) }{\varphi}$ to all of $[0,T]$ by continuity
and thus obtains a function $\mathscr{L}_\varphi : [0,T]\to \R$ such that
\begin{equation}
\label{l-revealed}
\mathscr{L}_\varphi(t) =  \pairing{}{Y}{\tilde{\ell}(t) }{\varphi} \qquad \text{for every } t\in \mathcal{T}.
\end{equation}

Let us now prove that
\begin{equation}
\label{weak-strange-convergence}
\pairing{}{Y}{\ell_{n_k}(t) }{\varphi} \to
\mathscr{L}_\varphi(t) 
 \qquad \text{as } k \to \infty \quad \text{for every } t \in [0,T] \text{ and } \varphi \in \mathcal{F}.
\end{equation}
In view of \eqref{1st-convergence}, we may assume
$t\in [0,T]\setminus J$. Then, we have
\[
\begin{aligned}
&
\left|  \pairing{}{Y}{\ell_{n_k}(t) }{\varphi} - 
\mathscr{L}_\varphi(t)
\right|
\\
&
\leq \left|  \pairing{}{Y}{\ell_{n_k}(t) }{\varphi} - \pairing{}{Y}{\ell_{n_k}(t_m) }{\varphi} \right|
+ \left|  \pairing{}{Y}{\ell_{n_k}(t_m) }{\varphi} - \pairing{}{Y}{\tilde{\ell}(t_m) }{\varphi} \right|
+\left|  \pairing{}{Y}{\tilde{\ell}(t_m) }{\varphi} - \mathscr{L}_\varphi(t) \right|
\\
&
\doteq \Delta_1 + \Delta_2+\Delta_3
\end{aligned}
\]
with $t_m$  in the dense set $ \mathcal{T}$ suitably chosen.
Now, on account of \eqref{helly+diag} for every fixed $\eps>0$ there exists $k_\eps^1\in \N$ such that for
$k \geq k_\eps^1$ one has
\[
\Delta_1 \leq \left| \mathcal{V}_{n_k}^{\varphi}(t) - \mathcal{V}_{n_k}^{\varphi}(t_m) \right| \leq  \left| \mathcal{V}_{\infty}^{\varphi}(t) - \mathcal{V}_{\infty}^{\varphi}(t_m) \right| + \frac{\eps}4.
\]
Using that $\mathscr{L}_\varphi(t) = \lim_{t_j \to t} \pairing{}{Y}{\tilde{\ell}(t_j)}{\varphi}$ for some $(t_j)_j \subset \mathcal{T}$ with $t_j \to t$,
we also have
\[
\begin{aligned}
\Delta_3   = \left|  \pairing{}{Y}{\tilde{\ell}(t_m) }{\varphi} - \lim_{j\to \infty}  \pairing{}{Y}{\tilde{\ell}(t_j)}{\varphi} \right|
&
 \stackrel{(1)}{\leq}
\lim_{j\to \infty}   \left| \mathcal{V}_{\infty}^{\varphi}(t_m) - \mathcal{V}_{\infty}^{\varphi}(t_j) \right|
\\
&
 \stackrel{(2)}{=} \left| \mathcal{V}_{\infty}^{\varphi}(t_m) - \mathcal{V}_{\infty}^{\varphi}(t) \right|,
 \end{aligned}
\]
where $(1)$ follows from \eqref{var-infty-lim} and $(2)$ from the fact that $t \in [0,T]\setminus J$ is a
is a continuity point of
$\mathcal{V}_\infty^{\varphi}$. In view of the latter fact, we may choose $t_m $ sufficiently close to $t$ such that $| \mathcal{V}_{\infty}^{\varphi}(t) - \mathcal{V}_{\infty}^{\varphi}(t_m) |\leq \frac{\eps}4$.  Finally, in view of \eqref{1st-convergence}, there exists $k_\eps^2 \in \N$ such that
$\Delta_2 \leq \frac{\eps}4$ for $k\geq k_\eps^2$.
All in all, we conclude that for $k \geq \max\{ k_\eps^1, k_\eps^2
\}$ we have $\left|  \pairing{}{Y}{\ell_{n_k}(t) }{\varphi} - \mathscr{L}_\varphi(t) \right|\leq \eps$, which yields \eqref{weak-strange-convergence}.

\underline{\textbf{Claim $2$:}}
\emph{Let
 $(\ell_{n_k})_k$ be a (not relabeled) subsequence of the sequence from Claim $1$,  with which a limiting Young measure $\bfmu= \{\mu_t\}_{t \in (0,T)}$ is associated according to Theorem
\ref{thm.balder-gamma-conv}.
Then, for almost all $t\in (0,T)$ the probability measure $\mu_t$ is a Dirac mass $\delta_{\ell(t)}$, and  \eqref{weak-ptw-B} holds as $k\to\infty$.}
\\
In order to show that  $\mu_t$
is a Dirac mass, we are going to prove that the set of the limit points of $(\ell_{n_k}(t))_{k}$ in the weak topology of $V$ is a singleton.
To this aim, let
us pick two points $\ell_\infty^1, \, \ell_\infty^2 \in \bigcap_{p=1}^{\infty}\overline{\big\{\ell_{n_k}(t)\,: \ k\ge p\big\}}^{\mathrm{weak}{\text{-}V}}$,
and two subsequences
$(\ell_{n_k^1}(t))_k, \, (\ell_{n_k^2}(t))_k$, possibly depending on $t$, such that
$\ell_{n_k^i}(t) \weakto \ell_\infty^i$ in $V$ as $k\to\infty$ for $i=1,2$.
Then
$\ell_{n_k^i}(t) \weakto \ell_\infty^i$ in $Y^*$. In view of \eqref{mie-th-conv} conclude that
\begin{equation}
\label{crucial-for-argum}
\pairing{}{Y}{\ell_\infty^1(t)}{\varphi} = \mathscr{L}_\varphi(t)= \pairing{}{Y}{\ell_\infty^2(t)}{\varphi} \qquad \text{for every } \varphi \in \mathcal{F}.
\end{equation}
Since $\mathcal{F}$ is dense in $\overline{B}_{1,Y}(0)$, we  deduce that $\ell_\infty^1$ and $\ell_\infty^2$ coincide on all the elements in
$\overline{B}_{1,Y}(0)$. But then by linearity we have that $\ell_\infty^1 = \ell_\infty^2 $ in $Y^*$, hence in $V$.
Therefore
for almost all $t\in (0,T)$
 the set $\bigcap_{p=1}^{\infty}\overline{\big\{\ell_{n_k}(t)\,: \ k\ge p\big\}}^{\mathrm{weak}{\text{-}V}}$
is a singleton $\{\ell(t)\}$, and \eqref{weak-ptw-B} ensues.
Observe that the functions $\ell$ and $\tilde{\ell}$ from  Claim $1$ need not coincide, as
\eqref{l-revealed} only holds on the set $\mathcal{T}$ with zero Lebesgue measure.

Finally, \eqref{weak-LpB} results from
\eqref{eq:35}
taking into account that $
\ell(t)=\int_{V}
l \, \dd \delta_{\ell(t)} (l)  $ and that the sequence $(\ell_n)$ is also bounded in $L^\infty (0,T; Y^*)$, cf.\ \eqref{sup-n-t}.
\end{proof}

\noindent
{\bf \large Acknowledgments.} The work of Elisabetta Rocca was
supported by the FP7-IDEAS-ERC-StG Grant \#256872
(EntroPhase) and by GNAMPA (Gruppo Nazionale per l'Analisi Matematica, la Probabilit\`a e le loro Applicazioni) of INdAM (Istituto Nazionale di Alta Matematica). Riccarda Rossi was
partially supported by a MIUR-PRIN'10-'11 grant for the project
``Calculus of Variations'',  and by GNAMPA (Gruppo Nazionale per l'Analisi Matematica, la Probabilit\`a e le loro Applicazioni)  of  INdAM (Istituto Nazionale di Alta Matematica). The authors thank the anonymous referees for their careful reading of the manuscript, which greatly improved the quality of the paper.

\bibliographystyle{alpha}

\begin{thebibliography}{99}

\beruno
\bibitem{abels}
H. Abels: Existence of weak solutions for a diffuse interface model
for viscous, incompressible fluids with general densities, Comm.
Math. Phys. {\bf 289} (2009),  45--73.
\eruno


\berdue
\bibitem{akagi}
G. Akagi: Maximal monotonicity for the sum of two subdifferential
operators in $L^p$-spaces, Nonlinear Anal. {\bf 74} (2011),
1664--1671.
\erdue


\beo
\bibitem{Balder84}
 E.J. Balder: A general approach to lower semicontinuity and lower closure in optimal control theory,
SIAM J. Control Optim. {\bf 22}  (1984), 570--598.

\bibitem{Ball89}
 J.M.  Ball: A version of the fundamental theorem for Young measures, in \emph{PDEs and continuum models of
phase transitions} (Nice, 1988), vol. 344 of Lecture Notes in Phys, pp. 207--215. Springer, Berlin (1989).
\eo










\beruno
\bibitem{BoBo} E. Bonetti, G. Bonfanti: Well-posedness results
for a model of damage in thermoviscoelastic materials, Ann. Inst. H.
Poincar\'e Anal. Non Lin\'eaire {\bf 25} (2008),  1187--1208.

\bibitem{bosch}
E. Bonetti, G. Schimperna: Local existence to Fr\'emond's model for
damaging in elastic materials, Contin. Mech. Thermodyn. {\bf 16}
(2004), 319--335.
\eruno



\beruno
\bibitem{bmr}
G. Bouchitt\'e, A. Mielke, T. Roub\'{\i}\v cek: A complete-damage
problem at small strains, Z. Angew. Math. Phys. {\bf  60} (2009),
205--236.
\eruno

\bibitem{boufeima}
M. Bul\'icek, E. Feireisl, J. M\'alek:
A Navier-Stokes-Fourier system for incompressible fluids with
    temperature dependent material coefficients,
Nonlinear Analysis: Real World Applications
{\bf 10} (2009), 992--1015.

\bibitem{brezis}
 H. Brezis:
 Op\'erateurs Maximaux Monotones et S\'emi-groupes de
  Contractions dans les Espaces de Hilbert,
North-Holland Mathematics Studies, no. 5.,  North-Holland Publishing
Co., Amsterdam-London; American Elsevier Publishing Co., Inc., New
York, 1973.


\bibitem{ciarlet}
P. Ciarlet: Mathematical elasticity. Vol. I. Three-dimensional
elasticity, Studies in Mathematics and its Applications, 20,
North-Holland Publishing Co., Amsterdam, 1988.


\bibitem{CV}
P. Colli, A. Visintin: On a class of doubly nonlinear evolution
equations, Comm. Partial Differential Equations {\bf 15} (1990),
 737--756




\beo
\bibitem{edwards}  R.E. Edwards,
Functional Analysis. Theory and Applications. Holt, Rinehart and
Winston, New York, 1965.
\eo

\bibitem{ers}
M. Eleuteri, E. Rocca, G. Schimperna:
On a non-isothermal diffuse interface model for two-phase flows of incompressible fluids,
 preprint arXiv:1401.3244  (2014), 1--26, \beo Discrete Contin. Dyn. Syst., to appear (2014). \eo

\bibitem{fei}
E.~Feireisl:
Mathematical theory of compressible, viscous, and heat conducting fluids,
 Comput.\ Math.\ Appl. {\bf 53} (2007), 461--490.

\bibitem{ffrs}
E. Feireisl, M. Fr\'emond, E. Rocca, G. Schimperna:
A new approach to non-isothermal models for nematic liquid crystals,
Arch.\ Rational Mech.\ Anal.
\textbf{205} (2012), 651--672.


\bibitem{fpr09} E. Feireisl, H. Petzeltov\'a, E. Rocca:
 Existence of solutions to a phase transition model with microscopic
    movements, Math. Methods Appl. Sci. {\bf 32} (2009),  1345--1369.

\bibitem{frsz1} E. Feireisl, E. Rocca, G. Schimperna, A. Zarnescu,
Evolution of non-isothermal Landau-de Gennes nematic liquid crystals flows with singular potential, Comm. Math. Sci., {\bf 12} (2014), 317--343.

\bibitem{frsz2} E. Feireisl, E. Rocca, G. Schimperna, A. Zarnescu,
Nonisothermal nematic liquid crystal flows with the Ball-Majumdar free energy, preprint arXiv:\-1310.8474 (2013), 1--27, WIAS Preprint No. 1865, (2013), \beo Annali di Matematica Pura ed Applicata, doi: 10.1007/s10231-014-0419-1, to appear (2014). \eo

\bibitem{fremond}
   M. Fr\'emond:
    Non-smooth thermomechanics,
     Springer-Verlag, Berlin, 2002.


\bibitem{fre-newbook}
  M. Fr\'emond:
Phase Change in Mechanics, Lecture Notes of the Unione Matematica
Italiana 13 Springer-Verlag, Berlin, 2012.

\beruno
\bibitem{fne} M. Fr\'emond, B. Nedjar: Damage, gradient of damage and principle of virtual power,
Internat. J. Solids Structures {\bf 33} (1996), 1083--1103.
\eruno






\bibitem{germain}
    P. Germain:
    Cours de m\'echanique des milieux continus,
    Tome I: Th\'eorie g\'en\'erale. Masson et Cie, \'Editeurs, Paris, 1973.




\bibitem{gmrs} M. Grasselli, A. Miranville, R. Rossi, G. Schimperna,
Analysis of the Cahn-Hilliard equation with a chemical potential
dependent mobility, Comm. Partial Differential Equations {\bf 36}
(2011), 1193--1238.

\bibitem{hk1} C. Heinemann, C. Kraus:
Existence results of weak solutions for Cahn-Hilliard systems
coupled with elasticity and damage: Adv. Math. Sci. Appl. {\bf 21} (2011), 321--359.


\bibitem{hk2} C. Heinemann, C. Kraus:
Existence results for diffuse interface models describing phase
separation and damage: European J. Appl. Math. {\bf 24} (2013), 179--211.

\beruno
\bibitem{HR} C. Heinemann, E. Rocca:
Damage processes in thermoviscoelastic materials with damage-dependent thermal expansion coefficients, preprint arXiv:1402.4177, (2014), Math. Methods Appl. Sci., to appear (2014).
\eruno

\beo
\bibitem{Hughes}
J.\ E.\ Marsden, T.\ J.\ R.\  Hughes:
     Mathematical foundations of elasticity,
Dover Publications, Inc., New York, 1994.
\eo

\bibitem{ioffe}
\newblock {A.D. Ioffe}:
\newblock {On lower semicontinuity of integral functionals},
\newblock SIAM J.~Control Optimization {\bf 15} (1977), 521--538.

\beo
 \bibitem{kle12} R. Klein: Laser Welding of Plastics, WILEY-VCH, 2012.
 \eo

\beo
\bibitem{krz2}
D.\ Knees, R.\ Rossi, C.\ Zanini:  A quasilinear differential inclusion for viscous and rate-independent damage systems in non-smooth domains, WIAS preprint No. 1867,  (2013).
\eo



\bibitem{kr-ro} P. Krej\v c\'{\i}, E. Rocca: Well-posedness of an extended model for water-ice phase transitions, Discrete Contin. Dyn. Syst. Ser.
S {\bf 6} (2013), 439--460.




\beruno
\bibitem{KreRoSprWilm}
P. Krej\v c\'{\i}, E. Rocca, J. Sprekels: Liquid-solid phase
transitions in a deformable container, contribution to the book
``Continuous Media with Microstructure'' on the occasion of
Krzysztof Wilmanski's 70th birthday, Springer (2010), 285--300.
\eruno



\bibitem{LM}
 J.-L. Lions, E. Magenes:
 Non-homogeneous boundary value problems and applications,
Vol. I., Springer-Verlag, New York-Heidelberg, 1972.



\beruno
\bibitem{la} E. Lorentz, S. Andrieux: Analysis of non-local models through energetic formulations,
International Journal of Solids and Structures {\bf 40} (2003),
2905--2936.

\eruno




\bibitem{Marcus-Mizel}
M. Marcus, V. Mizel:
Every superposition operator mapping one Sobolev space into another is continuous,
J. Funct. Anal. {\bf 33} (1979), 217--229.


\beo
\bibitem{MRS2013}
A.\ Mielke, R.\ Rossi, G.\ Savar\'e:  Nonsmooth analysis of doubly nonlinear evolution equations,
                  Calc. Var. Partial Differential Equations {\bf 46} (2013), 253--310.
\eo







\beruno
\bibitem{MieRou06}
A. Mielke, T. Roub\'{\i}\v cek: Rate-independent damage processes in
nonlinear elasticity Math. Models Methods Appl. Sci. {\bf 16}
(2006), 177--209.

\bibitem{mrz}
A. Mielke, T. Roub\'{\i}\v cek, J. Zeman: Complete damage in elastic
and viscoelastic media and its energetics,
 Comput. Methods Appl. Mech. Engrg. {\bf 199} (2010),  1242--1253.
\eruno

\bibitem{MieTh04}
A. Mielke, F. Theil: On rate-independent hysteris models, Nonlinear
Diff. Eq. Appl. {\bf 11} (2004), 151--189.


\beruno
\bibitem{mt}
A. Mielke, M. Thomas: Damage of nonlinearly elastic materials at
small strain--existence and regularity results, ZAMM Z. Angew. Math.
Mech. {\bf 90} (2010), 88--112.

\eruno

\bibitem{necas}
J.\ Ne\v{c}as.: Les m\'ethodes directes en th\'eorie des \'equations
elliptiques.  Masson et Cie, \'Editeurs, Paris, 1967.


\bibitem{nier}
L. Nirenberg: On elliptic partial differential equations, Ann.
Scuola Norm. Sup. Pisa (3) {\bf 13}  (1959), 115--162.

\bibitem{nsv}
R.H. Nochetto, G. Savar\'e, C. Verdi: A posteriori error estimates
for variable timestep discretizations of nonlinear evolution
equations, Comm. Pure Appl. Math. {\bf 53} (2000), 525--589.

\beo
\bibitem{Rossi-Savare06}
R.\ Rossi, G.\ Savar\'e:  Gradient flows of non convex functionals in
                  {H}ilbert spaces and applications,
                  ESAIM Control Optim. Calc. Var. {\bf 12} (2006), 564--614.
\eo


\bibitem{rocca-rossi1}
E. Rocca, R. Rossi: Analysis of a nonlinear degenerating PDE system
for phase transitions in thermoviscoelastic materials, J.
Differential Equations {\bf 245} (2008), 3327--3375.

\bibitem{rocca-rossi2}
E. Rocca, R. Rossi: Global existence of strong solutions to the
one-dimensional full model for phase transitions
 in thermoviscoelastic materials, Appl. Math. {\bf 53} (2008),
485--520.

\bibitem{rocca-rossi-deg}
 E. Rocca, R. Rossi, A degenerating PDE system for
phase transitions and damage, \beo Math. Models Methods Appl. Sci., {\bf 24} (2014), 1265--1341. \eo



\bibitem{roub-NPDE} T. Roub\'\i\v cek:
Nonlinear Partial Differential Equations with Applications,
Birkh\"auser, Basel, 2005.


\bibitem{roubiSIAM10}
T. Roub\'\i\v cek: Thermodynamics of rate independent processes in
viscous solids at small strains, SIAM J. Math. Anal. {\bf 40}
(2010), 256--297.


\bibitem{savare98}
G. Savar\'e: Regularity results for elliptic equations in Lipschitz
domains, J. Funct. Anal. {\bf 152} (1998), 176--201.

\bibitem{simon}
J. Simon: Compact sets in the space $L^p(0,T;B)$, Ann. Mat. Pura
Appl. {\bf 146} (1987), 65--96.


\beo
\bibitem{Valadier90}
 M. Valadier: Young measures. In: \emph{Methods of Nonconvex Analysis} (Varenna, 1989), pp. 152--188.
Springer, Berlin (1990).
\eo

\beruno

\bibitem{zr}
J.B.~Zel\'dovich, Y.P.~Raizer:
Physics of shock waves and high-temperature hydrodynamic phenomena, Academic Press,  New York,
1966.

\eruno

\bibitem{zimmer}
J. Zimmer: Global existence for a nonlinear system in thermoviscoelasticity with nonconvex energy,
J. Math. Anal. Appl. {\bf 292} (2004), 589--604.



\end{thebibliography}

\end{document}